\newif\ifsiam

\siamtrue
\siamfalse

\newif\ifarxiv
\ifsiam \arxivfalse \else \arxivtrue \fi

\ifsiam
\documentclass[onefignum,onetabnum]{siamart190516}
\usepackage{amsfonts,amssymb,amsopn} 

\usepackage[mathletters]{ucs}
\usepackage[utf8x]{inputenc}

\usepackage{interval}
\usepackage{lipsum}
\usepackage{graphicx}
\usepackage{epstopdf}
\usepackage{algorithmic}
\usepackage{comment}
\usepackage{todonotes}
\usepackage{qtree}
\usepackage[caption=false]{subfig}
\ifpdf
  \DeclareGraphicsExtensions{.eps,.pdf,.png,.jpg}
\else
  \DeclareGraphicsExtensions{.eps}
\fi

\ifsiam
\newsiamremark{remark}{Remark}
\crefname{remark}{Remark}{Remarks}
\newsiamremark{hypothesis}{Hypothesis}
\crefname{hypothesis}{Hypothesis}{Hypotheses}
\newsiamthm{claim}{Claim}
\fi

\usepackage[frozencache=true,cachedir=minted-cache]{minted}

\usepackage{pgfplots}
\usetikzlibrary{plotmarks}
\usepgfplotslibrary{colorbrewer}
\usepgfplotslibrary{groupplots}
\usetikzlibrary{pgfplots.groupplots}
\usetikzlibrary{calc}
\pgfplotsset{compat=1.3,
  ylabsh/.style={every axis y label/.style={at={(0,0.5)}, xshift=#1, rotate=90}}} 

\usetikzlibrary{external} 
\tikzsetexternalprefix{images/} 
\tikzset{external/system call={lualatex -shell-escape -halt-on-error -interaction=batchmode -jobname "\image" "\texsource"}}
\usepackage{pgfplotstable}

\definecolor{royalblue}{rgb}{0.00000,0.44700,0.74100}%
\definecolor{royalorange}{rgb}{0.85000,0.32500,0.09800}%
\definecolor{royalyellow}{rgb}{0.92900,0.69400,0.12500}%
\definecolor{purple}{rgb}{0.5804, 0.0, 0.82745098}%
\definecolor{applegreen}{rgb}{0.55, 0.71, 0.0}
\definecolor{bittersweet}{rgb}{1.0, 0.44, 0.37}

\definecolor{purple1}{HTML}{771155}
\definecolor{purple2}{HTML}{AA4488}
\definecolor{purple3}{HTML}{CC99BB}
\definecolor{blue1}{HTML}{114477}
\definecolor{blue2}{HTML}{4477AA}
\definecolor{blue3}{HTML}{77AADD}
\definecolor{green1}{HTML}{117777}
\definecolor{green2}{HTML}{44AAAA}
\definecolor{green3}{HTML}{77CCCC}
\definecolor{green4}{HTML}{117744}
\definecolor{green5}{HTML}{44AA77}
\definecolor{green6}{HTML}{88CCAA}
\definecolor{yellow1}{HTML}{777711}
\definecolor{yellow2}{HTML}{AAAA44}
\definecolor{yellow3}{HTML}{DDDD77}
\definecolor{brown1}{HTML}{774411}
\definecolor{brown2}{HTML}{AA7744}
\definecolor{brown3}{HTML}{DDAA77}
\definecolor{red1}{HTML}{771122}
\definecolor{red2}{HTML}{AA4455}
\definecolor{red3}{HTML}{DD7788}

\ifsiam
\headers{Homoclinic predictors near BT bifurcation}{M.M. Bosschaert and Yu.A.  Kuznetsov}
\fi

\title{Interplay between normal forms and center manifold reduction for
homoclinic predictors near Bogdanov-Takens bifurcation\ifsiam\thanks{Submitted to the editors DATE.}\fi}

\author{ 
  M.M. Bosschaert\thanks{Department of Mathematics, Hasselt University,
    Diepenbeek Campus, Agoralaan Gebouw D, 3590 Diepenbeek, Belgium
    (\email{maikel.bosschaert@uhasselt.be}).}
  \and
  Yu.A. Kuznetsov\thanks{Department of Mathematics, Utrecht University, 
    Budapestlaan 6, 3508 TA Utrecht, The Netherlands and\newline 
    Department of Applied Mathematics, University of Twente, Zilverling Building, 
    7500AE Enschede, The Netherlands (\email{I.A.Kouznetsov@uu.nl}).}
}

\DeclareMathOperator{\sech}{sech}
\DeclareMathOperator{\arcsech}{arcsech}

\DeclareMathOperator{\AINV}{A^{\text{INV}}}

\usepackage{subfiles} % Best loaded last in the preamble

\fi

\ifarxiv
\documentclass[english]{myarticle}

\fi

\ifpdf%
\hypersetup{%
  pdftitle={Interplay between normal forms and center manifold reduction for
homoclinic predictors near Bogdanov-Takens bifurcation},
  pdfauthor={M.M. Bosschaert, Yu.A.Kuznetsov}
}
\fi

\newif\ifcompileimages%
\compileimagestrue%
\compileimagesfalse%

\begin{document}

\maketitle

% REQUIRED
\begin{abstract}
This paper provides for the first time correct third-order homoclinic
predictors in $n$-dimensional ODEs near a generic Bogdanov-Takens bifurcation
point.  To achieve this, higher-order time approximations to the nonlinear time
transformation in the Lindstedt-Poincar\'e method are essential. Moreover, a
correct transform between approximations to solutions in the normal form and
approximations to solutions on the parameter-dependent center manifold is
needed.  A detailed comparison is done between applying different normal forms
(smooth and orbital), different phase conditions, and different perturbation
methods (regular and Lindstedt-Poincar\'e) to approximate the homoclinic
solution near Bogdanov-Takens points.  Examples demonstrating the correctness
of the predictors are given. The new homoclinic predictors are implemented in
the open-source MATLAB/GNU Octave continuation package MatCont.
\end{abstract}

% REQUIRED
\begin{keywords}
  codimension 2 Bogdanov-Takens bifurcation, homoclinic solution, 
  Lindstedt-Poincar\'e method, regular perturbation method, blow-up
  transformation, continuation software
\end{keywords}

% REQUIRED
\begin{AMS}
  37M20, 65P30, 34C37, 34B08, 34B15, 34B40, 34E10
\end{AMS}

\section{Introduction} 
Let $f\colon\mathbb R^n \times \mathbb R^2 \to \mathbb R^n$ with $n\geq 2$, be
smooth and suppose that the autonomous ordinary differential equation (ODE)
\begin{equation}
    \label{eq:ODE}
    \dot x(t) = f(x(t),\alpha)
\end{equation}
has equilibrium  $x_0= 0$ that undergoes a codimension two local bifurcation at
the critical parameter value $\alpha_{0}=0$. Here the dot means the derivative
with respect to the independent variable $t\in\mathbb R$. To understand the
dynamics near the bifurcation point for nearby parameter values, one typically
first restricts the ODE to the center manifold. By projecting the solutions on
the center manifold onto the center subspace, one then obtains a
$n_c$-dimensional ODE that locally governs the restricted dynamics.  Using the
normal form theory, one further tries to transform the restricted ODE into a
simpler form, called the \emph{critical normal form}.

If the canonical unfolding of the critical normal form is known and only
qua\-li\-tative behavior near the equilibrium is of interest, one can stop
here. However, if one is interested in relating solutions of the unfolding to
those of the original system \cref{eq:ODE} near the bifurcation point, one
needs a relation between the \emph{parameter-dependent normal form} and the
restricted ODE, and also a relation between this restricted ODE on the
parameter-dependent center manifold and the original system \cref{eq:ODE}.
These two relations can be found simultaneously utilizing the \emph{homological
equation} approach, see~\cite{Beyn2002Chapter4}. 

The solutions of interest here are the codimension one bifurcation curves
emanating from a codimension two point and the corresponding orbits in phase
space. In general, the bifurcation curves in the 
parameter-dependent normal form are not known exactly, but only by an
approximation up to a certain order.  Si\-mi\-lar\-ly, the transformation
between of the normal form to the (parameter-dependent) center manifold is
generally also only known up to a certain order. Then, by combining these two
transformations, an approximation to the codimension one bifurcation curve and
the corresponding phase orbits is obtained for the original system
\cref{eq:ODE}.

These approximations are particularly useful in numerical continuation software
to start the continuation of the codimension one bifurcation curves emanating
from the codimension two bifurcation points, where the defining systems for the
orbits of interest become degenerate. A codimension two bifurcation that has
attracted much attention is the \emph{Bogdanov-Takens bifurcation} at which the
cri\-ti\-cal equilibrium has a double zero eigenvalue. It is well known that
under certain non-degeneracy and transversality conditions, three codimension
one bifurcation curves emanate from the Bogdanov-Takens point: a saddle-node,
an (Andronov-)Hopf, and a saddle-homoclinic bifurcation curve. Since the
standard defining systems for the equilibrium bifurcations are non-degenerate
at the Bogdanov-Takens point, one does not need an approximation to start
continuation there.  On the contrary, the standard defining system for the
homoclinic solution does become degenerate, which is easily seen since the
homoclinic orbit shrinks to the equilibrium point at the Bogdanov-Takens
bifurcation.

Starting continuation of the homoclinic orbits from a Bogdanov-Takens point in
ODEs attracted much attention. In planar systems, Melnikov's method was first
applied to solve this problem in~\cite{Rodriguez-Luis1990}. A first attempt to
provide asymptotic approximations to the homoclinic bifurcation curve near a
generic codimension two Bogdanov-Takens bifurcation point in general
$n$-dimensional systems was made in~\cite{Beyn_1994}. By applying a singular
rescaling to the (one of the equivalent) parameter-dependent normal form on the
center manifold, a perturbed planar Hamiltonian system is obtained. The
unperturbed Hamiltonian system contains an explicit homoclinic solution. A
first-order correction in parameter-space can subsequently be obtained by
setting up the problem as a branching problem in a suitable Banach space,
see~\cite{Beyn_1994}.  Then, by using the regular perturbation method,
higher-order approximations to the homoclinic bifurcation curve can be
obtained. Unfortunately in~\cite{Beyn_1994}, even the first-order correction in
the phase-space was not derived.  Nonetheless, it was proven that the obtained
homoclinic predictor converges to the true solution under the Newton iterations
in the perturbed Hamiltonian systems.

In~\cite{Kuznetsov2014improved} the work continued by obtaining a second-order
correction in parameter \emph{and} phase-space to the homoclinic bifurcation
curve for the perturbed Hamiltonian system. However, a new problem was
overlooked.  The normal form used in~\cite{Kuznetsov2014improved} is a normal
form for $C^\infty$-equivalence (also called \emph{smooth orbital
equivalence}), i.e., besides a $C^\infty$-coordinate change, also a time
reparametrization must be taken into account, which was not the case
in~\cite{Kuznetsov2014improved}. In the subsequent paper~\cite{Gray-Scott2015},
this problem was resolved by considering a smooth normal form for the
Bogdanov-Takens bifurcation point, which is a normal form for
$C^\infty$-conjugacy (\emph{smooth equivalence}). 

In the follow-up paper~\cite{Al-Hdaibat2016}, progress was made in obtaining a
uniform approximation in the time along the homoclinic solution, using a
generalization of the Lindstedt-Poincar\'e method. This removes the so-called
parasitic turns near the saddle point, as observed
in~\cite{Kuznetsov2014improved}. Although, as pointed out
by~\cite{Algaba_2019}, there were mistakes in the third-order approximation
with the Lindstedt-Poincar\'e method, the homoclinic predictor from
\cite{Kuznetsov2014improved} for the smooth normal form improved significantly
in the phase space.

Nonetheless, the problem of correctly lifting the asymptotics in the normal
form to the parameter-dependent center manifold remained unnoticed and
unsolved.  Effectively, only the zeroth-order approximation to the homoclinic
solutions in the phase space, i.e., a transformed homoclinic solution of the
unperturbed Hamiltonian system, was available for a general $n$-dimensional
system.

In this paper, we will provide for the first time the third-order homoclinic
predictor for the homoclinic solutions emanating from a generic Bogdanov-Takens
point for a general $n$-dimensional system. For this, we need to consider
several additional systems to be solved in the homological equation method that
were previously not taken into account.  During the derivation of the
coefficients of the normal form and the transformations, we will show that
there is no need to solve certain systems simultaneously, see the so-called
`big' system in ~\cite{Kuznetsov2014improved,Gray-Scott2015,Al-Hdaibat2016}.
Furthermore, by allowing a transformation of time between the normal form and
the original system, we can use the parameter-dependent \emph{smooth orbital}
normal form of the codimension two Bogdanov-Takens bifurcation point when
approximating the homoclinic solution up to order three. This normal form is
considerably simpler than previously employed smooth normal forms. The
derivation of the coefficients will be the subject of
\cref{sec:Center_manifold_reduction_ODE}.

Having derived the parameter-dependent center manifold transformation suitable for
lifting the third-order homoclinic asymptotics present in different generic
Bogdanov-Takens normal forms, we turn our attention to obtain the asymptotics
in \cref{sec:asymptotics}. We will revisit both regular perturbation method and
the generalized Lindstedt-Poincar\'e method considered previously.

The non-uniqueness of the homoclinic solution due to a time shift results in
the non-uniqueness for the systems to be solved in the regular perturbation
method. To obtain uniqueness, a so-called \emph{phase condition} needs to be
satisfied. The phase condition used in~\cite{Kuznetsov2014improved} originates
from a theoretical setting in~\cite{Beyn_1994}. In
\cref{sec:RPM_norm_minimizing_phase condition} we use a geometrically motivated
phase condition which slightly improves the regular perturbation solution.
Furthermore, by modifying~\cite[Proposition 4.3]{Beyn_1994}, we use symmetry
arguments to simplify the calculations.

In \cref{sec:PolynomailLindstedtPoincare} the generalized Lindstedt-Poincar\'e
method for the approximation of homoclinic orbits is improved by introducing an
additional transformation of time after applying the usual nonlinear time
transformation. The resulting algorithm solitary relies on polynomial division
and does not involve any hyperbolic or trigonometric functions as
in~\cite{Algaba_2019,Al-Hdaibat2016}. We show that for the quadratic Bogdanov-Takens
normal form, we can represent the homoclinic solution in phase-space with
only one single parameter.

In \cref{sec:third_order_homoclinic_approximation_LP} we provide an explicit
third-order homoclinic approximation in the perturbed Hamiltonian system using
the algorithm described in \cref{sec:PolynomailLindstedtPoincareMethod}. Here we also
provide a third-order approximation to the reparametrization of time. The
profiles of the homoclinic solution will only then be approximated accurately,
resulting in a robust initial predictor for starting continuation of the branch
of homoclinic orbits.  In~\cite{Al-Hdaibat2016} the importance of the
time-reparametrization was not recognized, and the zeroth-order approximation
was used. We will demonstrate in detail that it is essential to use the higher
time reparametrization by comparing the Lindstedt-Poincar\'e method with and
without the higher-order time-reparametrization.  Effectively, using the
Lindstedt-Poincar\'e method without the higher-order time-reparametrization is
equivalent to the zeroth-order regular perturbation method.

The algorithm given in \cref{sec:PolynomailLindstedtPoincareMethod} is implemented in
\cref{sec:case_study_BT2} in the programming language
Julia~\cite{bezanson2017julia} for the quadratic normal form for the
Bogdanov-Takens codimension two bifurcation. Here we gain some insight about the
finite convergence radius of the homoclinic asymptotics and the speed of the
algorithm.

By combining the homoclinic asymptotics derived in \cref{sec:asymptotics} with
the parameter-dependent center manifold transformation obtained in
\cref{sec:Center_manifold_reduction_ODE}, we get the correct homoclinic
predictor for a general $n$-dimensional system.  It will be shown in
\cref{sec:homoclinic_asymptotics_n_dimension} how to incorporate the time
translation into the homoclinic predictor.
 
Then we compare the homoclinic predictor for the smooth
normal form and the smooth orbital normal form. In
\cref{sec:comparison_homoclinic_predictors}, it will be shown that these two
predictors are asymptotically equivalent, up to a phase shift. Then, by
choosing the constants of integration in the time translation in a specific
manner, we show equivalence between the predictors.

All the above methods are implemented in the open-source bifurcation and
continuation software MatCont ~\cite{Dhooge_2008}.  In
\cref{sec:implementation} we describe the new implementation of the homoclinic
predictor in the latest version of MatCont. We show how to use the obtained
predictors to construct an initial prediction for the defining system of the
homoclinic solutions. Besides an initial prediction also an initial tangent
vector is necessary to start continuation.  Our implementation resolves the
issue of possible continuation in the wrong direction, i.e., towards the
Bogdanov-Takens point. 

The effectiveness of the new predictors is demonstrated on the topological
normal form and on two four-dimensional models from Neuroscience and Quantum
Field Theory in \cref{sec:examples}. A comparison between the new homoclinic
predictor near a generic codimension 2 Bogdanov-Takens bifurcation and the
predictor from~\cite{Al-Hdaibat2016} is given. It will be shown that the order
of the higher-order approximations to the homoclinic solutions in the normal
form is preserved under the parameter-dependent center manifold transformation.
Complementary to \cref{sec:examples} an 
\href{https://mmbosschaert.github.io/MatCont7p2NewInitBTHom-/}{online Jupyter Notebook}
is provided in which ten different models are considered using the new
homoclinic predictor and comparing different approximation methods in detail.

\section{Center manifold reduction combined with normalization and time
reparametrization}
\label{sec:Center_manifold_reduction_ODE}

Consider system \cref{eq:ODE}, which has a codimension two bifurcation at
$\alpha = 0$ of equilibrium $x\equiv 0$. Let the normal form on the
$n_c$-dimensional center manifold be given by
\begin{equation}
				\label{eq:G}
				\frac{d}{d\eta} w(\eta) = G(w(\eta), \beta),
								\qquad G\colon\mathbb R^{n_c} \times \mathbb
								R^2 \to \mathbb R^{n_c},
\end{equation}
where $G$ is assumed to be one of the (known) equivalent normal forms.
Suppose that a parameter-dependent approximation to an emanating codimension
one bifurcation curve in \cref{eq:G} is given by
\begin{equation}
    \label{eq:general_approximation} 
    \epsilon \mapsto (w_\epsilon(\eta), \beta_\epsilon).
\end{equation}
By the Shoshitaishvili reduction principle the solutions on the
parameter-dependent center manifold can be parametrized by 
\begin{align}
				\label{eq:H}
        x(t(\eta)) &= H(w(\eta), \alpha), 
					 & H\colon\mathbb R^{n_c} \times \mathbb R^2 \to \mathbb R^n, \\
				\label{eq:K}
        \alpha &= K(\beta), 
           & K\colon\mathbb R^2 \to \mathbb R^2.
\end{align}
Now let the time $\eta$ be defined through the time rescaling
\begin{equation}
				\label{eq:theta}
				\frac{dt}{d\eta} = \theta(w, \beta), \qquad 
								\theta\colon \mathbb R^{n_c} \times \mathbb R^2 \to \mathbb R^n.
\end{equation}
Then the invariance of the center manifold implies the \emph{homological
equation}
\begin{equation}
				\label{eq:homological_equation}
				f(H(w,\beta), K(\beta)) \theta(w, \beta) = H_w(w, \beta) G(w,\beta).
\end{equation}

The unknowns in \cref{eq:homological_equation} are $H$, $K$, $\theta$, and the
normal form coefficients in \cref{eq:BT_smooth_nf}. By expanding the functions
$H,K, \theta$, and $f$ and collecting terms of equal power in $(w,\beta)$, we
obtain linear systems which can be solved at each order, potentially depending
on non-uniqueness present in lower order systems. 

To determine which coefficients are needed to include in the expansions of
$H,K$, and $\theta$, we need to consider which terms in the expansion of the
reduced system on the center manifold of \cref{eq:ODE} affect the approximation
\cref{eq:general_approximation}.  It is important here to not only take into
account the terms that affect the approximation that are present in the normal
form $G$ but also terms that are \emph{not in the normal form}, as long as the
approximation \cref{eq:general_approximation} is affected by the terms. This
has not been understood correctly and will be made explicit for the
approximation of the homoclinic bifurcation curve emanating from the generic
codimension two Bogdanov-Takens bifurcation point in the next section.

\subsection{Parameter-dependent normal form}%
Suppose that the ODE \cref{eq:ODE} undergoes a generic codimension two
Bogdanov-Takens bifurcation at $(x,\alpha) \equiv (x_0,\alpha_0)$. That is the
linearization of \cref{eq:ODE} has a double, but not semisimple, zero
eigenvalue, while all other eigenvalues are away from the imaginary axis. The
critical smooth normal form on the two-dimensional center manifold is given
by~\cite{Arnold_1983,guckenheimer1983nonlinear}
\begin{equation*}
\begin{aligned}
\begin{cases}
\dot w_0 = w_1, \\
\dot w_1 = a w_0^2 + b w_0 w_1 + \mathcal O(\|w\|^3),
\end{cases}
\end{aligned}
\end{equation*}
where 
\begin{equation*}
ab \neq 0,
\end{equation*}
$w_i$ is a shorthand notion for $w_i(\eta)$ for $i=0,1$, and the dot is
the derivative with respect to $\eta$.

Under these non-degeneracy and certain transversality conditions, the
\emph{topological normal form} for the codimension two Bogdanov-Takens
bifurcation is given by  
\begin{equation}
\label{eq:universal_unfolding}
\begin{aligned}
\begin{cases}
\dot w_0 = w_1, \\
\dot w_1 = \beta_1 + \beta_2 w_1 + a w_0^2 + b w_0 w_1,
\end{cases}
\end{aligned}
\end{equation}
see~\cite{Bogdanov1975,Bogdanov1976,Takens1974,guckenheimer1983nonlinear,Kuznetsov2004}. 
It is well known that in system (\ref{eq:universal_unfolding}) three codimension one
bifurcation curves emanate from $(\beta_1,\beta_2)=(0,0)$: a saddle-node, a Hopf,
and a saddle-homoclinic bifurcation curve. 

By using either the regular perturbation or the Lindstedt-Poincar\'e method,
an approximation to the homoclinic bifurcation curve and the corresponding
solution can, theoretically, be obtained up to any order in the
singular-rescaling parameter $\epsilon$,
see~\cite{Kuznetsov2014improved,Gray-Scott2015, Al-Hdaibat2016, Algaba_2019}.

To obtain the second-order homoclinic approximation to the homoclinic solutions
on the center manifold in \cref{eq:ODE}, it is, in general, insufficient to
only consider the topological normal form \cref{eq:universal_unfolding},
see~\cite{Broer1991}.  One way to deal with this problem is to consider the
smooth parameter-dependent normal form 
\begin{equation}
\label{eq:BT_smooth_nf}
\begin{cases}
\begin{aligned}
\dot{w}_0 = & w_1,\\
\dot{w}_1 = & \beta_1+\beta_2 w_1+\left(a+a_1\beta_2\right)w_0^{2}
 +\left(b+b_1\beta_2\right) w_0w_1+ew_0^{2}w_1+dw_0^{3} + g(w,\beta),
\end{aligned}
\end{cases}
\end{equation}
with
\[
    g(w,\beta) = \mathcal O(|\beta_1|\|w\|^2 + |\beta_2| w_1^2 + \|\beta\|^2\|w\|^2
					 + \|\beta\|\|w\|^3 + \|w\|^4)
\] 
as in~\cite{Gray-Scott2015,Al-Hdaibat2016}. Here $w_i=w_i(t)(i=0,1)$ now
depends explicitly on $t$ as in the original ODE \cref{eq:ODE}.

However, in this paper, we will allow for a time-reparametrization and use the
$C^\infty$-equivalent normal form 
\begin{equation}
\label{eq:normal_form_orbital}
\begin{cases}
\begin{aligned}
	\dot w_0 &= w_1, \\
	\dot w_1 &= \beta_1 + \beta_2 w_1 + aw_0^2 + b w_0 w_1 + w_0^2 w_1
								h(w_0,\beta) + w_1^2 Q(w_0,w_1,\beta),
\end{aligned}
\end{cases}
\end{equation}
where $h$ is $C^\infty$ and $Q$ is $N$-flat, see~\cite{Broer1991}. Here the dot
represents the derivative with respect to the new time $\eta$ of
$w_i(\eta)(i=0,1)$.  Furthermore, we will show that we can assume $h(0,0)=0$.
Note that we do not impose the coefficients to be $a=1$ and $b=\pm 1$ as
in~\cite{Broer1991}. This simplifies the systems to be solved in the next
section without complicating the solutions for the homoclinic corrections.
Indeed, we can scale-out the coefficients $a$ and $b$ in the
singular-rescaling. Also note that the normal form
\cref{eq:normal_form_orbital} was used in~\cite{Broer1991} to study degenerate
(codimension 3) Bogdanov-Takens bifurcations, while we found it to be essential
for constructing homoclinic predictors in the case of generic codimension two
Bogdanov-Takens bifurcations.

To approximate the homoclinic solutions emanating from the Bogdanov-Takens point
we apply the singular rescaling
\begin{equation}
\label{eq:blowup}				
w_0 = \frac a{b^2} u \epsilon^2, \quad
w_1 = \frac{a^2}{b^3} v \epsilon^3, \quad
\beta_1 = -4 \frac{a^3}{b^4} \epsilon^4, \quad 
\beta_2 = \frac a b \tau \epsilon^2, \quad 
s = \frac ab \epsilon \eta, \quad (\epsilon \neq 0),
\end{equation}
to \cref{eq:normal_form_orbital} with $h(0,0)=0$ to obtain the second
order nonlinear oscillator
\begin{equation}
\label{eq:second_order_nonlinear_oscillator}
				\ddot u = -4 + u^2 + \dot u \left( u + \tau \right)\epsilon + \mathcal
								O(\epsilon^4).
\end{equation}
Here the dot represents the derivative with respect to $s$.

\subsection{Center manifold reduction for smooth orbital normal form}
\label{subsec:center_manifold_tranformation_orbital}
We want to relate the third-order homoclinic approximation in the smooth orbital
normal form \cref{eq:normal_form_orbital} to the homoclinic solutions of
\cref{eq:ODE} near $(x_0,\alpha_0)$. The third-order approximation depends, by
definition, on the coefficients in $\epsilon$ up to order three in the
perturbed Hamiltonian system \cref{eq:second_order_nonlinear_oscillator}, see
\cref{sec:PolynomailLindstedtPoincare}.  By inspecting the blowup
transformation \cref{eq:blowup} we can determine exactly which coefficients in
\cref{eq:ODE} must be included in the expansion of $H,K$ and which multilinear
forms to include in the expansion of $f$. Indeed, we search for those terms in
the expansion of the reduced system on the center manifold of \cref{eq:ODE}
that affect the coefficients in $\epsilon$ up to order three in
\cref{eq:second_order_nonlinear_oscillator}. These are determined by solving
the linear Diophantine equation
\begin{equation}
				\label{eq:diophantine}
				4i + 2j + 2k + 3l - 4 = n, \qquad n\in\{-2,-1,0,1,2,3\},
\end{equation}
for $i,j,k,l \in \mathbb N_0$. In \cref{fig:terms_affecting_predicor} the
solutions to \cref{eq:diophantine} are listed. 
\begin{table}
\begin{center}
\begin{tabular}{ll}
\hline
order & terms \\
\hline%
\(\epsilon^{-2}\) & \(u_0,\alpha_2\)\\
\(\epsilon^{-1}\) & \(u_1\)\\
\(\epsilon^0\)  & \(u_0^2,u_0 \alpha_2,\alpha_2^2,\alpha_1\)\\
\(\epsilon^1\)  & \(u_0u_1, u_1\alpha_2\) \\
\(\epsilon^2\)  & \(u_1^2,u_0^3,u_0^2 \alpha_2 ,u_0 \alpha_2^2 ,\alpha_2^3,
									u_0\alpha_1 ,\alpha_1 \alpha_2\)\\
\(\epsilon^3\)  & \(u_0^2 u_1,u_0 u_1 \alpha_2, u_1 \alpha_2^2, u_1 \alpha_1\)\\
\hline
\end{tabular}
\caption{\label{fig:terms_affecting_predicor}
        Terms in the reduced system on the center manifold that affect the
    third-order predictor.} 
\end{center}
\end{table}
To transform away these terms (into higher order terms), one needs exactly the
corresponding coefficients in the expansions of $H$ and $K$. To be concrete,
suppose the term $\alpha_1 \alpha_2$ is present in the reduced system on the
center manifold. Applying the blowup transformation \cref{eq:blowup}, the term
$\alpha_1 \alpha_2$ will show up in the coefficient of $\epsilon^2$  in
\cref{eq:second_order_nonlinear_oscillator}. Since we will derive a third-order
approximation for \cref{eq:second_order_nonlinear_oscillator} in which the
corresponding term $\beta_1 \beta_2$ is not present, this term needs to be
transformed away. It is not too difficult to show that the coefficients needed
to perform this operation in phase and parameter-space are precisely $H_{0011}$
and $K_{11}$. In \cref{app:incorrect_predictor} an explicit example is given to
show that the transformation in~\cite{Al-Hdaibat2016} leads to an incorrect
predictor for the parameters.

Thus, we expand the mappings $H$, $K$, and $\theta$, including precisely those
coefficients needed to transfer the homoclinic predictor in the normal form to
the center manifold maintaining the accuracy. Using
\cref{fig:terms_affecting_predicor} we write:

\begin{align}
f(x,\alpha) ={}&
Ax+J_1\alpha+\frac12 B(x,x)+A_1(x,\alpha)+\frac12 J_2(\alpha,\alpha) 
+\frac16 C(x,x,x) \label{eq:f_expansion} \\
& \quad +\frac12 B_1(x,x,\alpha)+\frac12 A_2( x,\alpha, \alpha)
	+ \frac16 J_3(\alpha, \alpha, \alpha)  
  +\mathcal{O}\left(\|x\|^4+\|\alpha\|^3\right), \nonumber \\
H(w,\beta)={}& q_0w_0 + q_1w_1 + H_{0010}\beta_1 + H_{0001} \beta_2 
  + \frac12 H_{2000}w_0^2 + H_{1100}w_0w_1 + \frac12 H_{0200}w_1^2 
  \label{eq:H_expansion} \\
  & + H_{1010}w_0\beta_1 + H_{1001}w_0\beta_2 + H_{0110}w_1\beta_1 
  + H_{0101}w_1\beta_2 + \frac12 H_{0002}\beta_2^2\nonumber \\
  & + H_{0011}\beta_1\beta_2 + \frac16 H_{3000}w_0^3 + \frac12 H_{2100}w_0^2w_1 
  + H_{1101}w_0w_1\beta_2 + \frac12 H_{2001}w_0^2\beta_2\nonumber \\
  & + \frac{1}{6}H_{0003}\beta_2^3 + \frac12 H_{1002}w_0\beta_2^2 
  + \frac12 H_{0102}w_1\beta_2^2 \nonumber \\
  & + \mathcal{O}(|w_1|^3+|w_0w_1^2|+|\beta_2w_1^2|+|\beta_1|\|w\|^2
  +|\beta_1^2|\|w\| + |\beta_1^2| + \|(w,\beta)\|^4), \nonumber \\
K(\beta)={}& K_{10}\beta_1 + K_{01}\beta_2 + \frac{1}{2}K_{02}\beta_2^{2} 
	+ K_{11}\beta_1\beta_2 + K_{03} \frac16 \beta_2^3
  \label{eq:K_expansion} \\
  &+ \mathcal{O}(|\beta_1|^2+|\beta_1||\beta_2|^2 + |\beta_1|^2|\beta_2|  
	+ |\|\beta\|^4), \nonumber \\
\label{eq:theta_expansion}
\theta(w,\beta) ={}& 1 + \theta_{1000}w_0 + \theta_{0001} \beta_2 
				+ \mathcal O\left(|w| + |\beta_2| + \|(w,\beta)\|^2\right).
\end{align}
where $A=f_x(x_0,\alpha_0)$, $J_1=f_{\alpha}(x_0,\alpha_0)$, and
$B,J_2,J_3,C,A_1,A_2$ and $B_1$ are the standard multilinear forms, introduced
for readability.

\begin{remark}
Notice that compared with~\cite{Al-Hdaibat2016} there are four additional
terms in the expansion of $H$, i.e. with  coefficients $H_{0011},H_{1002},H_{0102}$, and
$H_{0003}$, and two additional terms in the expansion of $K$, with  coefficients
$K_{11}$ and $K_{03}$.
\end{remark}

\begin{figure}
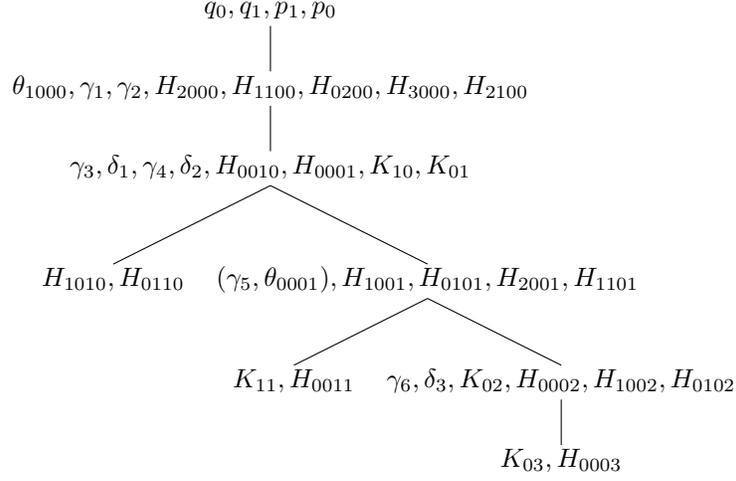

\renewcommand{\qtreeunaryht}{4ex}
\Tree%
[.{$q_0, q_1, p_1, p_0$} 
[.{$\theta_{1000}, \gamma_1, \gamma_2, H_{2000}, H_{1100}, H_{0200}, H_{3000}, H_{2100}$} 
		[.{$\gamma_3, \delta_1, \gamma_4, \delta_2, H_{0010}, H_{0001}, K_{10}, K_{01}$}
			{$H_{1010}, H_{0110}$} 
			[.{$(\gamma_5, \theta_{0001}), H_{1001}, H_{0101}, H_{2001}, H_{1101}$}
				{$K_{11}, H_{0011}$} 
				[.{$\gamma_6, \delta_3, K_{02}, H_{0002}, H_{1002}, H_{0102}$} 
					{$K_{03},H_{0003}$} ] ] ] ] ]
\caption{Schematic overview is which order the coefficients in the expansion of
$H$, $K$ and $\theta$ are derived.}
\end{figure}

\subsubsection{(Generalized) eigenvectors}

We assume that the equilibrium $(x_0, \alpha_0)$ has a double (but not
semisimple) zero eigenvalue, while all other eigenvalues are away from the
imaginary axis. Thus, there exist two real linearly independent (generalized)
eigenvectors, $q_0, q_1 \in \mathbb{R}^{n}$, of $A$, such that 
\begin{equation}
\label{eq:eigenvectors}
Aq_0=0,\qquad Aq_1=q_0,
\end{equation}
and two left (generalized) eigenvectors $p_1^T, p_0^T \in\mathbb{R}^{n}$,
of $A$, such that
\[
p_1 A=0,\qquad p_0 A=p_1.
\]
These vectors can be normalized to satisfy
\[
p_i q_j=\delta_{ij},\qquad i=0,1,\,j=0,1.
\]
As in~\cite{Kuznetsov2005practical}, we impose the condition
\begin{equation}
\label{eq:q0} 
q_0^T q_0=1,\qquad q_1^T q_0=0,
\end{equation}
to uniquely define the vectors $\{q_0,q_1,p_1,p_0\}$ up to a plus or minus sign.

Note that collecting the coefficients of the linear terms in $w$ in the
homological equation are precisely the systems defining the (generalized)
eigenvectors \cref{eq:eigenvectors}. 

\subsubsection{Critical normal form coefficients}%

Collecting the quadratic coefficients in $w$ from the homological
\cref{eq:homological_equation} yields the systems
\begin{align}
-AH_{2000} &= B(q_{0,}q_0) - 2aq_1, \label{eq:AH2000} \\
-AH_{1100} &= B(q_0,q_1) - bq_1 + \theta_{1000}q_0 - H_{2000},  \label{eq:AH1100} \\
-AH_{0200} &= B(q_1,q_1) - 2H_{1100}. \label{eq:AH0200} 
\end{align}

The Fredholm solvability condition for the first two systems yields the well
known expressions
\begin{align*}
a &= \frac12 p_1 B(q_0,q_0), \\
b &= p_1 B(q_0, q_1) + p_0 B(q_0, q_0),
\end{align*}
for the critical coefficients, see for example~\cite{Kuznetsov1999}. By the
non-degeneracy conditions, we have that $ab \neq 0$.

\begin{remark}
Since we assume that $p_1 B(q_0,q_0) \neq 0$ we can use the freedom in the
eigenvectors,
\[
(q_0, q_1) \to c_1 (q_0,q_1), \qquad
(p_1, p_0) \to \frac1{c_1} (p_1, p_0),
\]	
to normalize the critical coefficient
\[
a=p_1 B(q_0, q_0),
\]
to one. Solving for $c_1$ then gives
\[
c_1=\frac1{p_1 B(q_0, q_0)}.
\]
Alternatively, the freedom could have been used to set $b=1$. To have the
situation $a=1$ and $b=\pm1$, as in~\cite{Broer1991}, the coefficient in front
of the constant term in the expansion of $\theta$, i.e., $\theta_{0000}$, should
be used. For convenience, we fixed this constant to $1$.
\end{remark}

Now that \cref{eq:AH2000,eq:AH1100} are solvable, we can define
\begin{align*}
\hat H_{2000} &= -\AINV \left( B(q_0,q_0) - 2aq_1 \right), \\
\hat H_{1100} &= -\AINV \left( B(q_0,q_1) - bq_1 - \hat H_{2000} \right).
\end{align*}
The expression $x=\AINV y$ is defined by solving the non-singular
bordered system
\begin{equation*}
\begin{pmatrix}
A & p_1^T \\ q_0^T & 0
\end{pmatrix}
\begin{pmatrix}
x\\ s
\end{pmatrix}
=
\begin{pmatrix}
y\\ 0
\end{pmatrix},
\end{equation*}
for the unknown $(x,s) \in \mathbb R^{n+1}$ that necessarily satisfies $s = 0$.
The properties of bordered linear systems and their role in numerical
bifurcation analysis are discussed in~\cite{Keller1987Numerical}
and~\cite[Chapter 3]{govaerts2000numerical}.

It follows that the general solutions to the coefficients $H_{2000}$ and
$H_{1100}$ are given by
\begin{align*}
H_{2000} &= \hat H_{2000} + \gamma_1 q_0, \\
H_{1100} &= \hat H_{1100} + \gamma_1 q_1 - \theta_{1000} q_1 + \gamma_2 q_0.
\end{align*}
The constant $\gamma_1$ is determined by the solvability condition from
\cref{eq:AH0200}, which gives
\begin{equation*}
\gamma_1 = p_0  \left( B(q_0,q_1) -\hat H_{2000} \right) 
								+ \frac12 p_1 B(q_1,q_1) + \theta_{1000}.
\end{equation*}
To determine the constant $\gamma_2$ and the coefficient $\theta_{1000}$ we
consider the $w_0^3$ and $w_0^2w_1$ terms in the homological equation. After
some simplification, we obtain
\begin{align}
\label{eq:AH3000}
-A H_{3000} ={}& 3 B(H_{2000},q_0) + C(q_0,q_0,q_0) + 6a \theta_{1000} q_1
									- 6a H_{1100}, \\
\label{eq:AH2100}
-A H_{2100} ={}& -2 a H_{0200}-2 b H_{1100}-H_{3000} + 2 B(H_{1100},q_0) 
								+B(H_{2000},q_1) \\
				& + 2 \theta_{1000} (bq_1 - \theta_{1000}q_0
								+H_{2000})+C(q_0,q_0,q_1). \nonumber
\end{align}
The solvability condition of the first equation determines $\theta_{1000}$ as
\begin{equation}
\label{eq:theta1000}
\theta_{1000} = -\frac1{12a} p_1 \left\{ 
			3B(\hat H_{2000},q_0) + C(q_0,q_0,q_0)
			\right\} + \frac12 p_1 \hat H_{1100}.
\end{equation}
After a rather lengthy calculation, we obtain that $\gamma_2$ is determined by 
\begin{align}
\label{eq:gamma_2}
\gamma_2 &= \frac1{6a} 
\bigg[ p_1\left\{ 2 B(\hat H_{1100},q_0)+B(\hat H_{2000},q_1)
+ C(q_0,q_0,q_1) \right\} \\
  & \qquad +2a p_0 B(q_1,q_1) + 2b p_0 \left( B(q_0,q_1) 
    - \hat H_{2000} \right) \nonumber \\
	& \qquad + p_0 \left( 3B(\hat H_{2000},q_0) 
	  + C(q_0,q_0,q_0) \right) \nonumber \\
	& \qquad + \gamma_1 b - 10 a p_0 \hat H_{1100} + 2 b \theta_{1000} \bigg].
				\nonumber 
\end{align}
Since the systems in \cref{eq:AH0200,eq:AH2100,eq:AH3000} are now all consistent,
we are allowed to take the bordered inverses to obtain
\begin{align}
	H_{0200} ={} & -\AINV \left[B(q_1,q_1) - 2H_{1100} \right], \nonumber          \\
  \label{eq:H3000}
	H_{3000} ={} & -\AINV \left[ 3B(H_{2000},q_0) + C(q_0,q_0,q_0) +
		6a\theta_{1000}q_1 - 6a H_{1100} \right],                            \\
  \label{eq:H2100}
	H_{2100} ={} & -\AINV \left[ -2 a H_{0200}-2 b H_{1100}-H_{3000} + 2
	B(H_{1100},q_0), \right.                                              \\
	             & \left. \quad + B(H_{2000},q_1) + 2 \theta_{1000}
	(bq_1 - \theta_{1000}q_0 +H_{2000})+C(q_0,q_0,q_1) \right]. \nonumber
\end{align}

\subsubsection{Parameter-dependent linear terms}
The coefficients of the linear terms in $\beta$ give the
systems
\begin{equation}
\label{eq:AH0010}
\begin{aligned}
-AH_{0001} &= J_1K_{01}, \\
-AH_{0010} &= J_1K_{10} - q_1.
\end{aligned}
\end{equation}
Since $p_1$ and $J_1$ are known, we can calculate 
\begin{equation*}
    \nu =\begin{pmatrix} \tau_1 \\ \tau_2 \end{pmatrix} := (p_1 J_1)^T.
\end{equation*}
By the transversality condition, the vector $\nu$ is nonzero. It then follows
from the Fredholm alternative that
\begin{equation*}
\begin{aligned}
K_{01}   &= \delta_1\hat K_{01}, \\
H_{0001} &= \delta_1 \left( \hat H_{0001} + \gamma_3 q_0 \right), \\
K_{10}   &= \hat K_{10} + \delta_2 K_{01}, \\
H_{0010} &= \hat H_{0010} + \delta_2 H_{0001} + \gamma_4 q_0.
\end{aligned}
\end{equation*}
where
\begin{equation*}
\begin{aligned}
\hat K_{10} &= \frac1{\|\tau\|^2}\nu, \\
\hat H_{0010} &= \AINV \left(q_1 - J_1\hat K_{10}\right), \\
\hat K_{01} &=
\begin{pmatrix}
	0 & -1 \\ 1 & 0
\end{pmatrix} \hat K_{10}, \\
\hat H_{0001} &= -\AINV J_1\hat K_{01}.
\end{aligned}
\end{equation*}
and $\delta_{1,2}$, $\gamma_{3,4}$ are real constants determined by the
solvability condition of the $w\beta$ terms in the homological equation.
Collecting the corresponding systems in the homological equation yields 
\begin{equation*}
\begin{aligned}
-AH_{1001} &= B(H_{0001},q_0)+A_1(q_0,K_{01}), \\
-AH_{0101} &= B(H_{0001},q_1)+A_1(q_1,K_{01})-H_{1001}-q_1+\theta_{0001}q_0, \\
-AH_{1010} &= B(H_{0010},q_0)+A_1(q_0,K_{10})-H_{1100} +\theta_{1000}q_1,\\
-AH_{0110} &= B(H_{0010},q_1)+A_1(q_1,K_{10})-H_{0200}-H_{1010}.
\end{aligned}
\end{equation*}
The solvability condition for the first two systems yields
\begin{equation*}
\begin{aligned}
\gamma_3 &= -\frac{p_1 \left( 
				B(\hat H_{0001},q_0)+A_1(q_0,\hat K_{01}) \right)}{2a}, \\
\delta_1 &= \frac{1}{p_1 \left(
			B(\hat H_{0001},q_1)+A_1(q_1,\hat K_{01}) \right) + p_0
      \left( B(\hat H_{0001},q_0)+A_1(q_0,\hat K_{01}) \right) + \gamma_3 b},
\end{aligned}
\end{equation*}
while the solvability condition for the latter two systems yields
\begin{equation*}
\begin{aligned}
\gamma_4 &= \frac{p_1H_{1100} - \theta_{1000} - p_1 \left( B(\hat
				H_{0010},q_0)+A_1(q_0,\hat K_{10})\right)}{2a} ,\\
\delta_2 &= -p_1 \left( B(\hat H_{0010},q_1) + A_1(q_1,\hat K_{10}) \right)
				- \gamma_4 b + p_1 H_{0200} \\
	& \quad - p_0 \left( B(\hat H_{0010},q_0)+A_1(q_0,\hat K_{10})-H_{1100}
	\right). \\
\end{aligned}
\end{equation*}
Note that the denominator in $\delta_1$ is nonzero by the transversality
condition.

\subsubsection{Coefficients \texorpdfstring{$H_{1010} \text{ and }
								H_{0110}$}{H1010 and H0110}}

\begin{equation*}
\begin{aligned}
H_{1010} &= -\AINV \left[ B(H_{0010},q_0)+A_1(q_0,K_{10})-H_{1100}
								+\theta_{1000}q_1 \right],\\
H_{0110} &= -\AINV \left[B(H_{0010},q_1)+A_1(q_1,K_{10})-H_{0200}-H_{1010}
\right].
\end{aligned}
\end{equation*}

\subsubsection{Coefficients \texorpdfstring{$(\theta_{0001},\gamma_5),H_{1001},
H_{0101}, H_{2001}, H_{1101}$}{(theta0001,gamma5),H1001,H0101,H2001,H1101}}

Define
\begin{equation*}
\begin{aligned}
\hat H_{1001} &= -\AINV \left[ B(H_{0001},q_0)+A_1(q_0,K_{01}) \right], \\
\hat H_{0101} &= -\AINV \left[
				B(H_{0001},q_1)+A_1(q_1,K_{01})-H_{1001}-q_1 \right], \\
\end{aligned}
\end{equation*}
so that
\begin{equation*}
\begin{aligned}
H_{1001} &= \hat H_{1001} + \gamma_5 q_0, \\
H_{0101} &= \hat H_{0101} + \gamma_5 q_1 - \theta_{0001} q_1. \\
\end{aligned}
\end{equation*}

In order to determine $\gamma_5$ and $\theta_{0001}$, we consider the systems
corresponding to the $w_0^2\beta_2$ and $w_0w_1\beta_2$ terms in the homological
equation. These are given by
\begin{equation}
\label{eq:AH2001_AH1101}
\begin{aligned}
-AH_{2001} &= -2 a H_{0101} + A_1(H_{2000},K_{01}) +
	B(H_{0001},H_{2000}) + 2 B(H_{1001},q_0) \\
				& \qquad + 2 a \theta_{0001} q_1 +
				B_1(q_0,q_0,K_{01}) + C(H_{0001},q_0,q_0), \\
-AH_{1101} &= -b H_{0101} - H_{1100} - H_{2001} + A_1(H_{1100},K_{01}) + \\
				& \qquad \theta_{1000} (H_{1001} + q_1 - \theta_{0001} q_0) + 
				B(H_{0001},H_{1100}) + B(H_{0101},q_0) + \\
				& \qquad B(H_{1001},q_1) + \theta_{0001} (H_{2000} + b q_1 -
				\theta_{1000} q_0) + B_1(q_0,q_1,K_{01}) \\
				& \qquad + C(H_{0001},q_0,q_1).
\end{aligned}
\end{equation}
The Fredholm solvability condition leads to the following system to be solved
\begin{equation}
\label{eq:gamma_5_theta0001}
\begin{pmatrix}
				-2a & -4a \\
				 -b &  -b 
\end{pmatrix}
\begin{pmatrix}
				\gamma_5 \\
				\theta_{0001}
\end{pmatrix}
=
\begin{pmatrix}
				\zeta_1 \\
				\zeta_2 
\end{pmatrix},
\end{equation}
where
\begin{equation}
\label{eq:zeta1_zeta2}
\begin{aligned}
\zeta_1 &=  p_1 \left[ -2 a\hat H_{0101}  
			+ A_1(H_{2000},K_{01}) + B(H_{0001},H_{2000}) \right.  \\
			& \left. \qquad + 2 B(\hat H_{1001},q_0)
			+ B_1(q_0,q_0,K_{01}) + C(H_{0001},q_0,q_0) \right], \\
\zeta_2 &= p_1 \left[ -b \hat H_{0101} - H_{1100}  +
				A_1(H_{1100},K_{01}) + \right. \\
				& \qquad \theta_{1000} (\hat H_{1001} + q_1) + 
				B(H_{0001},H_{1100}) + B(\hat H_{0101},q_0) + \\
				& \qquad \left. B(\hat H_{1001},q_1) + B_1(q_0,q_1,K_{01})
				+ C(H_{0001},q_0,q_1) \right]  \\
				& \qquad + p_0 \left[ -2 a \hat H_{0101} + A_1(H_{2000},K_{01}) +
				B(H_{0001},H_{2000}) \right. \\
				& \qquad + \left. 2 B( \hat H_{1001},q_0) +
				B_1(q_0,q_0,K_{01}) + C(H_{0001},q_0,q_0) \right] .
\end{aligned}
\end{equation}	
Notice that the matrix has a non-vanishing determinant by the non-degeneracy
condition. Now that the systems in \cref{eq:AH2001_AH1101} are solvable we
obtain
\begin{equation}
\label{eq:H2001_H1101}
\begin{aligned}
H_{2001} &= -\AINV \left[ -2 a H_{0101} + A_1(H_{2000},K_{01}) +
	      B(H_{0001},H_{2000})  \right. \\
				& \qquad \left. + 2 B(H_{1001},q_0) + 2 a \theta_{0001} q_1 +
				B_1(q_0,q_0,K_{01}) + C(H_{0001},q_0,q_0) \right], \\
H_{1101} &= -\AINV \left[ -b H_{0101} - H_{1100} - H_{2001} +
				A_1(H_{1100},K_{01}) + \right. \\
				& \qquad \theta_{1000} (H_{1001} + q_1 - \theta_{0001} q_0) + 
				B(H_{0001},H_{1100}) + B(H_{0101},q_0) + \\
				& \qquad B(H_{1001},q_1) + \theta_{0001} (H_{2000} + b q_1 -
								\theta_{1000} q_0) + B_1(q_0,q_1,K_{01}) \\
				& \qquad  \left. + C(H_{0001},q_0,q_1) \right].
\end{aligned}
\end{equation}

\subsubsection{Coefficients \texorpdfstring{$K_{11} \text{ and } H_{0011}$}{K11
				and H0011}}

Collecting the systems corresponding to the $\beta_1 \beta_2$ term in the
homological equation yields
\begin{equation}
\label{eq:AH0011}
\begin{aligned}
-A H_{0011} &= J_1 K_{11} + A_1(H_{0001},K_{10}) + A_1(H_{0010},K_{01}) \\
						& \qquad +
						B(H_{0001},H_{0010})+J_2(K_{01},K_{10})+\theta_{0001}q_1-H_{0101}.
\end{aligned}
\end{equation}
Using the identity 
\begin{equation*}
   p_1J_1K_{10}=1 
\end{equation*}
from the second system in \cref{eq:AH0010} combined with the solvability
condition yields
\begin{equation}
\begin{aligned}
K_{11}={}& -p_1\left[A_1(H_{0001},K_{10})+A_1(H_{0010},K_{01}) \right. \nonumber\\
				 &\left. + B(H_{0010},H_{0001}) +
				 J_2(K_{10},K_{01}) + \theta_{0001} q_1 - H_{0101} \right]K_{10}.
\end{aligned}
\end{equation}
It follows that 
\begin{equation}
\begin{aligned}
H_{0011} &= -\AINV \left[ J_1 K_{11} + A_1(H_{0001},K_{10}) +
				A_1(H_{0010},K_{01}) \right. \\
				 & \qquad \left. +
                    B(H_{0001},H_{0010})+J_2(K_{01},K_{10})+\theta_{0001}q_1-H_{0101}
				 \right].
\end{aligned}
\end{equation}

\subsubsection{Coefficients
				\texorpdfstring{$K_{02},H_{0002},H_{1002},H_{0102}$}
				{H0002,K02,H1002,H0102}}

The systems corresponding to the $\beta_2^2, w_0\beta_2^2$ and $w_1\beta_2^2$,
terms in the homological equation yields
\begin{equation}
\label{eq:AH0002_AH1002_AH0102}
\begin{aligned}
-A H_{0002}={}&  J_1K_{02} + 2A_1(H_{0001},K_{01})
							 + B(H_{0001},H_{0001}) + J_2(K_{01},K_{01}), \\
-A H_{1002}={}& 2A_1(H_{1001},K_{01}) + A_1(q_0,K_{02}) + A_2(q_0,K_{01},K_{01})
					 \\ & + B(q_0,H_{0002}) + 2B(H_{0001},H_{1001}) +
					 2B_1(q_0,H_{0001},K_{01}) \\
  & + C(q_0,H_{0001},H_{0001}), \\
-A H_{0102}={}& 2A_1(H_{0101},K_{01}) + A_1(q_1,K_{02}) + A_2(q_1,K_{01},K_{01}) \\
  & + B(q_1,H_{0002}) + 2B(H_{0001},H_{0101}) + 2B_1(q_1,H_{0001},K_{01}) \\
	& + C(q_1,H_{0001},H_{0001}) +2\theta_{0001} (H_{1001} + q_1 -
	  \theta_{0001}q_0)  \\
  & - 2H_{0101} - H_{1002}.
\end{aligned}
\end{equation}
The first system is solved similarly as \cref{eq:AH0011}. However, here we need
to use hypernormalization in order to make the second and third systems
consistent. Thus, we define
\begin{equation*}
\begin{aligned}
\hat K_{02}={}&-p_1\left[2A_1(H_{0001},K_{01})+B(H_{0001},H_{0001})
				+J_2(K_{01},K_{01})\right]K_{10}, \\
\hat H_{0002}={}& -\AINV\left[ J_1 \hat K_{02} + 2A_1(H_{0001},K_{01})
							 + B(H_{0001},H_{0001}) + J_2(K_{01},K_{01})\right].
\end{aligned}
\end{equation*}
Then the general solutions can be written as
\begin{equation*}
\begin{aligned}
				K_{02}={}& \hat K_{02} + \delta_3 K_{01}, \\
				H_{0002}={}& \hat H_{0002} +  \delta_3 H_{0001} + \gamma_6 q_0.
\end{aligned}
\end{equation*}
Substituting these two expressions into \cref{eq:AH0002_AH1002_AH0102} and
using the solvability condition yields
\begin{equation*}
\begin{aligned}
				\gamma_6 ={}& -\frac1{2a} p_1 \left[ 2A_1(H_{1001},K_{01}) +
				A_1(q_0,\hat K_{02}) + A_2(q_0,K_{01},K_{01}) \right. \\
	& \qquad + B(q_0,\hat H_{0002}) + 2B(H_{0001},H_{1001}) +
				2B_1(q_0,H_{0001},K_{01}) \\
  & \left. \qquad + C(q_0,H_{0001},H_{0001}) \right],  \\
\delta_3 ={}& -p_1 \left[ 2A_1(H_{0101},K_{01}) + A_1(q_1,\hat K_{02}) +
    A_2(q_1,K_{01},K_{01}) \right. \\
	& \qquad + B(q_1,\hat H_{0002}) + 2B(H_{0001},H_{0101}) +
	    2B_1(q_1,H_{0001},K_{01}) \\
	& \left. \qquad +  C(q_1,H_{0001},H_{0001}) +2\theta_{0001} (H_{1001} + q_1) -
	  2H_{0101} \right] \\
	& \qquad - p_0 \left[ 2A_1(H_{1001},K_{01}) + A_1(q_0,\hat K_{02}) +
    A_2(q_0,K_{01},K_{01}) \right. \\ 
	& \qquad + B(q_0,\hat H_{0002}) + 2B(H_{0001},H_{1001}) +
    2B_1(q_0,H_{0001},K_{01}) \\
  & \left. \qquad + C(q_0,H_{0001},H_{0001}) \right] - \gamma_6 b.
\end{aligned}
\end{equation*}
Now that the last two systems in \cref{eq:AH0002_AH1002_AH0102} are consistent,
we obtain
\begin{equation*}
\begin{aligned}
H_{1002}={}& -\AINV \left[ 2A_1(H_{1001},K_{01}) + A_1(q_0,K_{02}) +
				A_2(q_0,K_{01},K_{01}) \right. \\ 
			  & + B(q_0,H_{0002}) + 2B(H_{0001},H_{1001}) +
				2B_1(q_0,H_{0001},K_{01})  \\
				& + \left. C(q_0,H_{0001},H_{0001}) \right], \\
H_{0102}={}& -\AINV \left[ 2A_1(H_{0101},K_{01}) + A_1(q_1,K_{02}) +
				A_2(q_1,K_{01},K_{01}) \right. \\
  & + B(q_1,H_{0002}) + 2B(H_{0001},H_{0101}) + 2B_1(q_1,H_{0001},K_{01}) \\
	& + C(q_1,H_{0001},H_{0001}) +2\theta_{0001} (H_{1001} + q_1 -
	  \theta_{0001}q_0)  \\
	& \left. - 2H_{0101} - H_{1002} \right].
\end{aligned}
\end{equation*}

\subsubsection{Coefficients \texorpdfstring{$K_{03} \text{ and } H_{0003}$}{K03
				and H0003}}

Collecting the systems corresponding to the $\beta_2^3$ term in the
homological equation yields
\begin{equation*}
\begin{aligned}
-A H_{0003} ={}& J_1 K_{03} + A_1(H_{0001},K_{02}) + A_1(H_{0002},K_{01})
				+ 2 (A_1(H_{0001},K_{02}) \\
				& + A_1(H_{0002},K_{01}) + 3 B(H_{0001},H_{0002}) + 3 J_2(K_{01},K_{02})
				\\
				& + 3 A_2(H_{0001},K_{01},K_{01}) + 3 B_1(H_{0001},H_{0001},K_{01}) \\
				& + C(H_{0001},H_{0001},H_{0001}) + J_3(K_{01},K_{01},K_{01}).
\end{aligned}
\end{equation*}
This equation is solved similarly as equation \cref{eq:AH0011}. We obtain
\begin{equation*}
\begin{aligned}
K_{03}={}& -p_1 \left[ A_1(H_{0001},K_{02}) + A_1(H_{0002},K_{01})
				+ 2 A_1(H_{0001},K_{02}) \right. \\
				& + 2 A_1(H_{0002},K_{01}) + 3 B(H_{0001},H_{0002}) + 3 J_2(K_{01},K_{02})
				\\
				& + 3 A_2(H_{0001},K_{01},K_{01}) + 3 B_1(H_{0001},H_{0001},K_{01}) \\
				& \left. + C(H_{0001},H_{0001},H_{0001}) +
				J_3(K_{01},K_{01},K_{01}) \right] K_{10}, \\
H_{0003} ={}& -\AINV \left[ J_1 K_{03} + A_1(H_{0001},K_{02}) +
				A_1(H_{0002},K_{01}) + 2 A_1(H_{0001},K_{02}) \right. \\
				& + 2 A_1(H_{0002},K_{01}) + 3 B(H_{0001},H_{0002}) + 3 J_2(K_{01},K_{02})
				\\
				& + 3 A_2(H_{0001},K_{01},K_{01}) + 3 B_1(H_{0001},H_{0001},K_{01}) \\
				& \left. + C(H_{0001},H_{0001},H_{0001}) +
				J_3(K_{01},K_{01},K_{01})\right] .
\end{aligned}
\end{equation*}

\subsection{Center manifold reduction for smooth normal form} 
\label{sec:center-manifold-reduction-without-time-reparametrization}
If we do not allow for a reparametrization of time, we can no longer consider the
normal form \cref{eq:normal_form_orbital}. Instead, we need to use the smooth
normal form as introduced in~\cite{Gray-Scott2015}, i.e., equation
\cref{eq:BT_smooth_nf}. Applying the blowup transformation
\begin{equation}
\label{eq:blowup_smooth}				
\beta_1 = - \frac4a \epsilon^4, \quad 
\beta_2 = \frac b a \tau \epsilon^2, \quad 
w_0 = \frac1a u \epsilon^2, \quad
w_1 = \frac1a v \epsilon^3, \quad
s = \epsilon t, \quad (\epsilon \neq 0),
\end{equation}
to the smooth normal form, we obtain the second-order nonlinear differential
equation 
\begin{equation}
\label{eq:second_order_nonlinear_oscillator_smooth_normalform}
				\ddot u = -4 + u^2 + \frac ba \dot u \left( u + \tau \right)\epsilon 
				+ \frac1{a^2} u^2\left(\tau b a_1 + d u \right) \epsilon^2
        + \frac1{a^2} u\dot u \left( \tau b b_1 + e u \right) \epsilon^3
				+ \mathcal{O}(\epsilon^4).
\end{equation}
Here the dot represents the derivative with respect to $s$. 

Note that, by using hypernormalization, we can still simplify the smooth normal form.
Indeed, as already remarked in~\cite{Kuznetsov2005practical} the coefficient
$e$ can be set to zero.  Furthermore, it can be seen from the system in
\cref{eq:gamma_5_theta0001} that either the coefficient $a_1$ or $b_1$ can also
be removed. The natural choice here is for the coefficient $b_1$ to
be set to zero in the normal form. The parameter-dependent center manifold
transformation in this situation is obtained by first setting the coefficients
$\theta_{1000}$ and $\theta_{0001}$ to zero in
\cref{subsec:center_manifold_tranformation_orbital}. Equation \cref{eq:AH3000}
becomes
\begin{equation}
    \label{eq:AH3000_smooth}
    -A H_{3000} = 3 B(H_{2000},q_0) + C(q_0,q_0,q_0) - 6d q_1 - 6a H_{1100},
\end{equation}
and equation \cref{eq:theta1000} is removed. After $\gamma_2$ in
\cref{eq:gamma_2} has been calculated, the Fredholm solvability condition yields
that
\begin{equation*}
    d = \frac16 p_1 \left[ 3B(H_{2000},q_0) + C(q_0,q_0,q_0)  - 6 a \hat H_{1100}\right].
\end{equation*}
Now that \cref{eq:AH3000_smooth} is consistent, we can replace \cref{eq:H3000}
with
\begin{equation*}
	H_{3000} = -\AINV \left[ 3B(H_{2000},q_0) + C(q_0,q_0,q_0) - 6dq_1 - 6a H_{1100} \right].
\end{equation*}

Next, we replace the first equation in \cref{eq:AH2001_AH1101} with
\begin{equation}
\label{eq:AH2001}
\begin{aligned}
-AH_{2001} ={}& - 2 a_1 q_1 -2 a H_{0101} + A_1(H_{2000},K_{01}) + B(H_{0001},H_{2000}) \\
              & + 2 B(H_{1001},q_0) + B_1(q_0,q_0,K_{01}) + C(H_{0001},q_0,q_0)
\end{aligned}
\end{equation}
and the system in \cref{eq:gamma_5_theta0001} becomes the single equation
\[
\gamma_5 = -\frac{\zeta_2}{b}.
\]
Here $\zeta_2$ is still given by the second equation in \cref{eq:zeta1_zeta2}
(with $\theta_{1000}$ still set to zero), while $\zeta_1$ is no longer needed.
Applying the Fredholm solvability condition to \cref{eq:AH2001} yields 
\begin{equation*}
\begin{aligned}
    a_1 &= \frac{1}{2} \left[-2 a H_{0101} + A_1(H_{2000},K_{01}) +
	B(H_{0001},H_{2000}) + 2 B(H_{1001},q_0) \right. \\
        & \left. \qquad + B_1(q_0,q_0,K_{01}) + C(H_{0001},q_0,q_0) \right].
\end{aligned}
\end{equation*}
Since \cref{eq:AH2001} is now consistent, we can replace the first equation in
\cref{eq:H2001_H1101} with
\begin{equation*}
\begin{aligned}
    H_{2001} &{}= -\AINV \left[ -2 a H_{0101} + A_1(H_{2000},K_{01}) + B(H_{0001},H_{2000})  \right. \\
				& \left. + 2 B(H_{1001},q_0) - 2 a_1 q_1 + B_1(q_0,q_0,K_{01}) + C(H_{0001},q_0,q_0) \right].
\end{aligned}
\end{equation*}
The remaining systems and equations are unchanged.

To compare the homoclinic predictors under different transformations, we also
provide the parameter-dependent center manifold transformation for the smooth
normal form \cref{eq:BT_smooth_nf} without transforming away the coefficients
$e$ and $b_1$. In this case, in addition to the modification given above, we also set
$\gamma_2$ and $\gamma_5$ to zero. Then the system in \cref{eq:AH2100} becomes
\begin{equation}
\label{eq:AH2100smooth}
\begin{aligned}
-A H_{2100} ={}& - 2 e q_1 -2 a H_{0200}-2 b H_{1100}-H_{3000} + 2 B(H_{1100},q_0) \\
				& + B(H_{2000},q_1) + C(q_0,q_0,q_1),
\end{aligned}
\end{equation}
while the second systems in \cref{eq:AH2001_AH1101} should be replaced with
\begin{equation}
\label{eq:AH1101smooth}
\begin{aligned}
-AH_{1101} ={}& -b_1 q_1 -b H_{0101} - H_{1100} - H_{2001} + A_1(H_{1100},K_{01}) + \\
              &+ B(H_{0001},H_{1100}) + B(H_{0101},q_0) + B(H_{1001},q_1) \\
              &+ B_1(q_0,q_1,K_{01}) + C(H_{0001},q_0,q_1).
\end{aligned}
\end{equation}
Applying the Fredholm solvability condition to these equations gives
\begin{equation*}
\begin{aligned}
e ={}& \frac{1}{2} p_1 \left[ -2 a H_{0200}-2 b H_{1100}-H_{3000} + 2 B(H_{1100},q_0) \right.  \\
     & \left. + B(H_{2000},q_1) + C(q_0,q_0,q_1) \right], \\ 
b_1 ={}& p_1 \left[ -b H_{0101} - H_{1100} - H_{2001} + A_1(H_{1100},K_{01}) + B(H_{0001},H_{1100})  \right. \\
      & \left. + B(H_{0101},q_0) + B(H_{1001},q_1) + B_1(q_0,q_1,K_{01}) + C(H_{0001},q_0,q_1) \right].
\end{aligned}
\end{equation*}
Now that \cref{eq:AH2100smooth,eq:AH1101smooth} are consistent, we can replace
\cref{eq:H2100} and the second system in \cref{eq:H2001_H1101} with
\begin{align*}
	H_{2100} ={} & -\AINV \left[ -2e q_1 - 2 a H_{0200}-2 b H_{1100}-H_{3000} + 2 B(H_{1100},q_0) \right. \\
	             & \left. \quad + B(H_{2000},q_1) + C(q_0,q_0,q_1) \right]. \nonumber
\end{align*}
and
\begin{equation*}
\begin{aligned}
H_{1101} ={}& -\AINV \left[ -b_1 q_1 -b H_{0101} - H_{1100} - H_{2001} + A_1(H_{1100},K_{01}) + \right. \\
				& + B(H_{0001},H_{1100}) + B(H_{0101},q_0) + B(H_{1001},q_1) \\
				& \left. + B_1(q_0,q_1,K_{01}) + C(H_{0001},q_0,q_1) \right],
\end{aligned}
\end{equation*}
respectively. The remaining systems and equations are unchanged.

\begin{remark}
The derivation in~\cite{Al-Hdaibat2016,Kuznetsov2014improved} leads to a `big'
system in which equations need to be solved simultaneously. The derivation
presented here does not involve a `big' system, making the expressions more
suitable to implement for infinitely-dimensional ODEs generated by partial and
delay differential equations, to which the (parameter-dependent) center
manifold theorem applies.
\end{remark}

\section{Homoclinic asymptotics}
\label{sec:asymptotics}

In this section, we derive third-order asymptotics to the homoclinic solution
near the generic Bogdanov-Takens point. We revisited the standard regular
perturbation method, but with a different phase condition. In
\cref{sec:examples} we will show that this improves the accuracy of the
homoclinic approximation in the normal form. Then, in
\cref{sec:PolynomailLindstedtPoincare} we revisited the Lindstedt-Poincar\'e
method. By an additional non-linear time transformation, we obtain a very simple
algorithm to approximate the homoclinic solution.

\subsection{The Regular Perturbation Method with norm minimizing phase condition}
\label{sec:RPM_norm_minimizing_phase condition}
For $\epsilon=0$, \cref{eq:second_order_nonlinear_oscillator} is a Hamiltonian
system with the first integral
\begin{equation*}
    H(u,\dot u) = \frac12 \dot u^2+4u-\frac13 u^3.
\end{equation*}
The Hamiltonian system has the well-known explicit homoclinic solution
$(u_0(s),\dot u_0(s))$ given by
\[
    u_{0}(s) = 6 \tanh^2(s) - 4.
\]
Thus, for $(u, \epsilon, \tau) = (u_0, 0, \tau)$ there exists a trivial branch
of homoclinic orbits in \cref{eq:second_order_nonlinear_oscillator}.
In~\cite{Beyn_1994} it is shown that there exists a bifurcation point at
$\tau=\frac{10}{7}$ from which a smooth non-trivial branch emanates
transversally. Parametrizing this branch by $\epsilon$, we formally have
\begin{equation}
    \label{eq:u_i_tau_i_RPM}
    u(s,\epsilon) = \sum_{i\geq 0} u_i(s) \epsilon^i, 
    \qquad 
    \tau(\epsilon) = \sum_{i\geq 0} \tau_i \epsilon^i.
\end{equation}
Substituting \cref{eq:u_i_tau_i_RPM} into
\cref{eq:second_order_nonlinear_oscillator} and collecting equal terms in
$\epsilon$ yields the following differential equations to be solved:
\begin{align}
  &\ddot u_0 - u_0^2 +4 = 0, \label{eq:z0} \\
  &\ddot u_i - 2 u_0 u_i
    =
    z_i, 
        \label{eq:regular_perturbation_method_equation_epsilon_i} \\
  &\dot u_i (\pm \infty) = \ddot u_i (\pm \infty) = 0,
        \qquad i\in \mathbb N. \nonumber
\end{align}
Here $z_i$ dependents on the sums and products of $u_j, \dot u_j$ and $\tau_{j-1}$
for $0\leq j < i$.  Multiplying equation
\cref{eq:regular_perturbation_method_equation_epsilon_i} by $\dot u$ and
integrating from $s_0$ to $s$ yields,
\[
  \int_{s_0}^s \dot u_0 \ddot u_i - 2 \dot u_0 u_0 u_i \;dx = \int_{s_0}^s \dot
  u_0 z_i \; dx.
\]
Using integration by parts twice then gives
\begin{equation}
  \label{eq:regular_perturbation_method_integration_by_parts}
  \left. \left( \dot u_0 \dot u_i - u_i \ddot u_0 \right)\right|_{s_0}^s = \int_{s_0}^s \dot
  u_0 z_i \; dx.
\end{equation}
Notice that solutions $(\dot u_i(s),\ddot u_i(s))$ must vanish at plus and minus
infinity. We obtain that $\tau_{i-1}$ is given by the condition
\begin{equation*}
    0 = \int_{-\infty}^\infty \dot u_0 z_i \; dx.
\end{equation*}
To simplify the equations that follow below we would like to
use~\cite[Proposition 4.2]{Beyn_1994}. 

However, we noticed that the proposition is not precise enough for the
conclusion to hold. Indeed, the proof relies on the uniqueness of the
non-trivial branch of homoclinic orbits. However, we see that the left-hand
side \cref{eq:regular_perturbation_method_equation_epsilon_i} is invariant
under the transformation
\[
u_i \to u_i + \gamma \dot u_0, \qquad \gamma \in\mathbb R, \; i\in\mathbb N.
\]
We, therefore, slightly modify the proposition with an additional assumption.

\begin{proposition}
    \label{proposition:symmetry}
    Assume that the perturbed Hamiltonian system
    \cref{eq:second_order_nonlinear_oscillator} is obtained from the normal form
    \cref{eq:normal_form_orbital} by the singular rescaling
    \cref{eq:blowup}. Then the non-trivial branch of homoclinic solutions 
    \[
        (u(s,\epsilon), \dot u(s, \epsilon), \epsilon, \tau(\epsilon)),
            \qquad \epsilon<|\epsilon_0|,
    \]
    for some $\epsilon_0>0$ satisfies
    \begin{equation}
        \label{eq:tau}
        \tau(\epsilon) = \tau(-\epsilon).
    \end{equation}
    Furthermore, if the solutions $u_i$ are even functions for $i$ even, then 
    the solution $u$ contains the additional symmetry
    \[
        u(s,-\epsilon) = u(-s,\epsilon)
    .\] 
\end{proposition}
\begin{proof}
    The proof follows almost entirely~\cite[Proposition 4.2]{Beyn_1994}. Thus,
    it can be shown that the transformation \cref{eq:blowup} induces the
    symmetry
    \[
        \varphi(s, D_0 (u^0, \dot u^0)^T, -\epsilon, \tau)
        =
        D_0 \varphi(-s, (u^0, \dot u^0)^T, \epsilon, \tau)
    \] 
    on the flow $\varphi$ of \cref{eq:second_order_nonlinear_oscillator}, where 
    \[
        D_0 = \begin{pmatrix} 1 & 0 \\ 0 & -1 \end{pmatrix}.
    \] 
    From this relation one can then conclude that if $(u(s,\epsilon), \dot
    u(s,\epsilon), \epsilon, \tau(\epsilon))$ is a homoclinic solution to
    \cref{eq:second_order_nonlinear_oscillator} then so is $(\tilde
    u(-s,\epsilon), -\dot{\tilde{u}}(-s,\epsilon), -\epsilon, \tau(\epsilon))$.
    The proof in~\cite{Beyn_1994} then finishes with the remark that these two
    homoclinic solutions must be equal by the uniqueness of the non-trivial
    branch. However, by the \emph{non-uniqueness} of the non-trivial branch, we
    obtain the relation
    \[
        (u(s,\epsilon), \epsilon, \tau(\epsilon))
        =
        (u(-s,-\epsilon), \epsilon, \tau(-\epsilon))
        +
        (\gamma(\epsilon) \dot u_0(s), \epsilon, 0),
    \]
    where
    \begin{equation}
        \label{eq:gamma}
        \gamma(\epsilon) = \sum_{i\geq 1} \gamma_i \epsilon^i,
        \qquad
        \gamma_i \in \mathbb R.
    \end{equation}
    Thus, $\tau$ is indeed an even function of $\epsilon$. Using the expansion
    for $u$ from \cref{eq:u_i_tau_i_RPM} we see that by inspecting the
    coefficients of equal powers in $\epsilon$ we only need to impose that
    $\gamma_i=0$ for $i$ even and then the assertion follows.
\end{proof}

It, therefore, follows from \cref{proposition:symmetry} together with equality
\cref{eq:regular_perturbation_method_integration_by_parts} that the condition for
solving $\tau_{i-1}$ in \cref{eq:tau_condition_RP} simplifies to
\begin{equation}
    \label{eq:tau_condition_RP}
    0 = \int_0^\infty \dot u_0 z_i \; dx.
\end{equation}
for $i$ even, whereas $\tau_i=0$ for $i$ odd.

From \cref{eq:regular_perturbation_method_integration_by_parts} we obtain the
solution
\begin{equation}
    \label{eq:u_i_RP}
    u_i = \dot u_0 \int \frac{1}{\dot u_0^2} \int \dot u_0 z_i \; dx \; dx,
        \qquad i \in \mathbb N,
\end{equation}
or
\begin{equation*}
    u_i = \left( \dot u_0 \int \frac{1}{\dot u_0^2} \, dx \right) \int \dot u_0 z_i(u,\dot u,\tau) \, dx
-  \dot u_0 \int \left( \int \frac{1}{\dot u_0^2} \, dx \right) \dot u_0 z_i(u,\dot u,\tau) \, dx. 
\end{equation*}
From \cref{eq:u_i_RP} we see that there are two integration constants involved.
The first integration constant, originating from the inner integral, is needed
to ensure the boundedness of the solution. The second integration constant
introduces precisely the freedom 
\begin{equation*}
    u_i \rightarrow u_i + \gamma_i \dot u_0,  \qquad i\in\mathbb N,
\end{equation*}
with $\gamma_i\in\mathbb{R}$ constants. In~\cite{Kuznetsov2014improved} the
condition
\begin{equation}
    \label{eq:RP_phase_condtion_unnatural}
    \dot u_i(0) = 0
\end{equation}
is imposed to ensure the uniqueness of the solution.  This phase condition is also
used in~\cite{Beyn_1994} a theoretical setting. However, a more natural phase
condition would be to minimize the $L^2$-distance between the current and
previous solution obtained from the regular perturbation method. This
phase condition is also used in~\cite{Doedel1986auto, Champneys1996,
DeWitte2012,Doedel@1989} for numerical continuation of heteroclinic and
homoclinic orbits. Using the $L^2$ phase condition yields
\begin{equation*}
  \int_{-\infty}^\infty 
        \langle 
            (\dot u_0(s), \ddot u_0(s)) (u_{i}(s) + \gamma_i \dot u_0(s),
            \dot u_{i}(s) + \gamma_i \ddot u_0(s)) 
        \rangle
        \, ds = 0,
    \qquad i \in \mathbb{N}.
\end{equation*}
By \cref{proposition:symmetry} this phase condition is
equivalent to the condition that
\begin{equation}
     \label{eq:u_i_L2_phase_condition}
  \int_0^\infty 
        \langle
        (\dot u_0(s), \ddot u_0(s)) (u_{i}(s) + \gamma_i \dot u_0(s),
        \dot u_{i}(s) + \gamma_i \ddot u_0(s)) 
        \rangle
        \, ds = 0,
    \qquad i \in \mathbb{N},
\end{equation}
for $i$ odd and $\gamma_i=0$, for $i$ even, if we ensure that $u_i$ is even for $i$
even. By using integration by parts together with \cref{eq:z0} and subsequently
solving \cref{eq:u_i_L2_phase_condition} for $\gamma_i$ we obtain
\begin{equation}
    \label{eq:c_i}
    \gamma_i = -\frac{35}{2592} 
            \int_0^\infty \dot u_0(1 - 2 u_0)u_{i}(s) \, ds,
        \qquad i \in \mathbb{N}.
\end{equation}
In~\cite{Kuznetsov2014improved}, the phase condition
\begin{equation}
     \label{eq:u_i_L2_phase_condition_u_only}
     \int_{-\infty}^\infty \left( u(s) - u_0(s) \right) \dot u_0(s) \, ds
     = 0,
    \qquad i \in \mathbb N,
\end{equation}
was also tested. This phase condition only minimizes the $L_2$-distance of the
$u$-component between the current and the zeroth-order solution obtained from
the regular perturbation method. It is reported in~\cite{Kuznetsov2014improved}
that for \cref{eq:u_i_L2_phase_condition_u_only} no substantial superiority
over using phase condition \cref{eq:RP_phase_condtion_unnatural} was found.
Our findings show that, at least for \cref{eq:u_i_L2_phase_condition}, this is
only partially true.  Indeed, the numerical simulations in
\cref{sec:topological_normal_form} show that, as one would expect, using the
phase condition \cref{eq:u_i_L2_phase_condition} does indeed improve the
approximation to the homoclinic orbit. However, when the homoclinic
approximations are lifted from the normal form to the center manifold, the
phase conditions are, in general, not preserved, and the improvements are no
longer observed.

As we will see below, the $L_2$ phase condition
\cref{eq:u_i_L2_phase_condition} is more difficult to solve. It is, therefore,
more efficient to use the orbital normal form \cref{eq:universal_unfolding}
instead of the smooth normal form \cref{eq:BT_smooth_nf}.

\subsubsection{Third-order homoclinic approximation}
\label{sec:third_order_homoclinic_approximation_RP}
For $i=1$ we obtain the equation
\begin{equation*}
    z_1(s) = (u_0(s)+\tau_0) \dot u_0(s).
\end{equation*}
Condition \cref{eq:tau_condition_RP} yields
\begin{equation*}
    \tau_0 = \frac{10}{7}.
\end{equation*}
Then from \cref{eq:u_i_RP} we obtain the solution
\begin{equation*}
    u_1(s) = -\frac{6}{7} \dot u_0(s) \log (\cosh (s)).
\end{equation*}
The $L_2$ phase condition then yields that
\[
\gamma_1 = -\frac{3}{245} (70 \log (2)-59)
.\] 
Note that the integral to be evaluated in \cref{eq:c_i} is labor-intensive and
prone to error.  Therefore, we used the (freely available) Wolfram
Engine~\cite{WolframEngine} (although not open source).  Correcting the
previous solution $u_1$ leads to the solution
\begin{equation*}
    u_1(s) = \frac{3}{245} (59-70 \log (2 \cosh (s))) \dot u_0(s).
\end{equation*}

Continuing with the second-order system we have the equation
\[
    z_2 = (u_0+\tau_0) \dot u_1+u_1 \dot u_0+u_1^2.
\] 
Here we directly used that $\tau_1=0$ by the symmetry as explained above.
From \cref{eq:u_i_RP} we obtain
\begin{align*}
    u_2(s) ={}& 
    \frac{1}{60025} 36 \sech^2 s  \left[3 \sech^2 s  \left\{70 \log (2 \cosh  s )
        (105 \log (2 \cosh s )-247) \right.\right. \\ 
              &{} \left.\left. +6289\right\}-2(3675 s \tanh  s +210 \log (2 \cosh  s )
              (35 \log (2 \cosh  s )-94)+7129)\right].
\end{align*}
Notice that, since $u_2$ is an even function, we automatically have that
$\gamma_2$ vanishes.

For $i=3$ we have
\begin{equation*}
    z_3 = (u_2+\tau _2) \dot u_0+(u_0+\tau _0) \dot u_2+u_1 \left(\dot u_1+2 u_2\right).
\end{equation*}
Condition \cref{eq:tau_condition_RP} yields
\begin{equation*}
    \tau_2 = \frac{288}{2401}.
\end{equation*}

Then from \cref{eq:u_i_RP} we obtain the solution
\begin{align*}
    u_3(s) ={}&
    \frac{216 \sech^2 s}{14706125}  
    \left[\sech^2 s  \left\{3675 s (210 \log (2 \cosh  s )-247) \right. \right. \\
              & +\tanh  s  \left(-171500 (\cosh (2 s)- 5) \log ^3(2 \cosh  s )+7350 (129 \cosh (2 s)-470) \right. \\
              & \left.\left. \log^2(2 \cosh  s )+4456830 \log (2 \cosh s)-966242\right)\right\} \\
              & \left.-70 \{210 s (35 \log (2 \cosh  s )-47)+30673 \tanh  s \log (\cosh  s )\}\right]
.\end{align*}

Trying to solve the integral in \cref{eq:c_i} with the Wolfram Engine yields
\begin{align}
    \label{eq:c_i_integral}
    \int_0^\infty \dot u_0(1 - 2 u_0)u_3(s) \, ds ={}& 
    \frac{16 \left(-5234558923+331676100 \pi^2+6260972760 \log (2)\right)}{514714375} \\
    &{} -\frac{155520}{343}\int_0^\infty 
                \log^3(2\cosh s )\sech^6 s \tanh^2 s  \, ds \nonumber \\ 
    &{}  \frac{31104}{343}
                \int_0^\infty \log^3(2\cosh s )\sech^6 s \tanh^2 s \cosh(2s) 
                \, ds \nonumber \\
    &{}  \frac{622080}{343}
            \int_0^\infty \log^3(2\cosh s )\sech^8 s \tanh^2 s  
            \, ds \nonumber \\ 
    &{} -\frac{124416}{343}
            \int_0^\infty \log^3(2\cosh s )\sech^8 s \tanh^2 s \cosh(2s) 
            \, ds, \nonumber
\end{align}
i.e., the Wolfram Engine was unable to solve the integral. We observe that, in
order to solve \cref{eq:c_i_integral}, it is sufficient to solve integrals of
the form
\begin{equation}
    \label{eq:I_n}
    I_n := \int_0^{\infty} \log^3 (2 \cosh s) \sech^n s \, ds,
\end{equation}
with $n=4,6,8$ and $n=10$. After a lengthy calculation, see \cref{sec:I_n}, we
obtain the closed form expression
\begin{multline*}
I_n = 2^{n-3} 3 \sum_{k=0}^{\frac{n}{2}-1} \binom{\frac{n}{2}-1}k 
            (-1)^k  \\
        \left[\frac1{(\frac{n}{2}+k)^4} +
            \frac{8}{2k+n} \left(\frac{H_{\frac{n}{2} + k}}{(2k+n)^2}
                + \frac{H^{(2)}_{\frac{n}{2}+k} - \zeta(2)}{2(2k+n)}
            + \frac{H^{(3)}_{\frac{n}{2}+k} - \zeta(3)}{4}
        \right)\right].
\end{multline*}
where $\zeta$ is the Riemann zeta function and $H_n^{(m)}$ is the $n$th
generalized harmonic number of order $m$. Explicitly we obtain
\begin{align*}
    I_4 ={}& \frac{82}{27}-\frac{5 \pi ^2}{36}-\zeta (3), \\
    I_6 ={}& \frac{38342}{16875}-\frac{47 \pi ^2}{450}-\frac{4 \zeta (3)}{5}, \\
    I_8 ={}& \frac{25545482}{13505625}-\frac{319 \pi ^2}{3675}-\frac{24 \zeta (3)}{35}, \\
    I_{10} ={}& \frac{5428830032}{3281866875}-\frac{7516 \pi ^2}{99225}-\frac{64 \zeta (3)}{105}.
\end{align*}
From the above integrals, we deduce that
\begin{equation*}
    \gamma_3 = \frac{264 \zeta(3)}{343}-\frac{884895199}{7147176750}
            -\frac{100 \pi^2}{3087}-\frac{1104228 \log 2}{420175}.
\end{equation*}
Since, by the symmetry, $\tau_3=0$, we obtain the third-order predictor 
\begin{equation}
\label{eq:third_order_predictor_RPM_tau}
\begin{cases}
\begin{aligned}
w_0(\eta)  &= \frac{a}{b^2} \left( \sum_{i=0}^3 u_i(\frac{a}{b}\epsilon\eta) \epsilon^i +
\mathcal{O}(\epsilon^4) \right)   \epsilon^2, \\
w_1(\eta)  &= \frac{a^2}{b^3} \left( \sum_{i=0}^3 \dot u_i(\frac{a}{b}\epsilon\eta) \epsilon^i +
\mathcal{O}(\epsilon^4) \right)   \epsilon^3, \\
\beta_1    &= -4 \frac{a^3}{b^4}\epsilon^4, \\
\beta_2    &= \frac{a}{b}\epsilon^2 \left( \frac{10}{7} + \frac{288}{2401} \epsilon^2 + \mathcal{O}(\epsilon^4) \right),
\end{aligned}
\end{cases}
\end{equation}
for the smooth orbital normal form \cref{eq:normal_form_orbital}.

\begin{remark}
In~\cite{Kuznetsov2014improved} there is the remark that the
author~\cite{Beyn_1994} was unable to find a tangent predictor due to the
(normalized) form. The system in~\cite[Equation (4.5)]{Beyn_1994} to be solved
is given by
\begin{equation}
  \label{eq:Beyn_tangent_system}
  \begin{cases}
  \begin{aligned}
  \dot x  - y &= \frac12 a_1 \hat x^2 + a_2 \tau_0 \hat x + a_3 \tau_0^2, \\
  \dot y -2\hat x x &=  b_1 \hat x \hat y + b_2 \tau_0 \hat y,
  \end{aligned}
  \end{cases}
\end{equation}
here $(\hat x, \hat y)$ is the zeroth-order solution $(u_0, \dot u_0)$. The
coefficients $a_1,a_2,b_1,b_2$ are different normal form coefficients then used
in this paper, but $\tau_0$ is identical to the $\tau_0$ used in this paper.
Using the same technique as above it is easy to derive that
\begin{align*}
x(s) ={}& \left(a -\frac{1}{16} a_3 s  \tau _0^2-\frac{1}{96} a_3 \tau _0^2
	\sinh (2 s ) -\frac{1}{48} a_2 \tau _0 \sinh (2 s )+\frac{1}{24} a_3 \tau _0^2
	\coth (s ) \right . \\ 
					&-\frac{1}{24} a_2 \tau _0 \coth (s )+\frac{a_1 s
	}{8}-\frac{1}{48} a_1 \sinh (2 s ) +\frac{1}{12} a_1 \coth (s )+\frac{1}{15}
	b_2 \tau_0 \cosh (2 s ) \\ 
	&+\frac{1}{240} b_2 \tau _0 \cosh (4 s
	)-\frac{2}{21} b_1 \cosh (2 s )-\frac{1}{168} b_1 \cosh (4 s ) \\ 
	&\left.  -\frac{6}{7} b_1 \log (\cosh (s )) \right) \hat y(s)
\end{align*}
is a solution to \cref{eq:Beyn_tangent_system}. Therefore,
by~\cite{Keller1977,Decker1986} the convergence of the homoclinic solution in
the perturbed Hamiltonian system follows. Although the author
in~\cite{Beyn_1994} was unable to provide a tangent approximation in
phase-space, they did prove, by refining the convergence cones
from~\cite{Jepson1986}, the convergence of the zeroth-order approximation in
phase-space in the perturbed Hamiltonian system.
\end{remark}

\subsection{A polynomial Lindstedt-Poincar\'e method}
\label{sec:PolynomailLindstedtPoincare}

The Lindstedt-Poincar\'e method considers a nonlinear time transformation
defined implicitly through the relation 
\begin{equation}
  \label{eq:first_non_linear_time_transformation}
  \frac{d\xi}{ds}=\omega(\xi),
\end{equation}
which can be used to remove the so-called \emph{secular terms}, i.e., terms
growing without bound, appearing in the process of approximating periodic
orbits in weakly nonlinear oscillators using the regular perturbation approach.

The Lindstedt-Poincar\'e method is also used to approximate homoclinic
solutions in nonlinear oscillators, referred to as the generalized
Lindstedt-Poincar\'e method,
see~\cite{Chen2009,Chen@2009A,Chen@2009B,Chen@2010,Chen@2012}. In this case,
there are no terms growing without bound when applying the regular perturbation
approach. Instead, there are so-called \emph{parasitic turns}, see~\cite[Figure
1]{Kuznetsov2014improved}.  The nonlinear transformation
\cref{eq:first_non_linear_time_transformation} can then be used to remove the
parasitic turns. In fact, using the nonlinear transformation, one can obtain a
very simple form for the solution of the homoclinic orbit in phase-space,
see~\cite[Equation 35]{Chen2009} and~\cite{Algaba_2019}.

In both cases, i.e., when approximating periodic orbits or homoclinic orbits,
we do the same: a nonlinear time transformation is used to obtain a uniform
approximation of the orbit in time. 

\subsubsection{General method}
\label{sec:PolynomailLindstedtPoincareMethod}
Substituting the parameterization of time $\omega$
\cref{eq:first_non_linear_time_transformation} into
\cref{eq:second_order_nonlinear_oscillator} yields
\begin{equation}
\label{eq:secondorder_gBT_omega}
\omega\dfrac{d}{d\xi}\left(\omega \hat u^\prime\right)-\hat u^2+4
  =\epsilon\omega \hat u^\prime (\hat u+\tau) + \mathcal O(\epsilon^4).
\end{equation}
were $\hat u(\xi(s)) = u(s)$.

We now perform one additional transformation of time
\begin{equation}
  \label{eq:second_time_transformation}
  \frac{d\zeta}{d\xi} = 1 - \zeta^2
\end{equation}
to simplify the solutions obtained below. Note that this transformation implies
that $\zeta = \tanh(\xi + c_1)$, where $c_1$ is some constant. Without loss of
generality, we can assume that $c_1=0$ since $c_1$ just shifts the homoclinic
solution in time. Substituting \cref{eq:second_time_transformation} into
\cref{eq:secondorder_gBT_omega} yields
\begin{equation}
\label{eq:secondorder_gBT_zeta}
(1 - \zeta^2)\tilde\omega\dfrac{d}{d\zeta}\left((1-\zeta^2)\tilde\omega \tilde
u^\prime\right)-\tilde u^2+4
  =\epsilon\tilde\omega (1-\zeta^2)\tilde u^\prime(\tilde u+\tau) +\mathcal{O}(\epsilon^4),
\end{equation}
where the prime ${}^\prime$ now represents the derivative with respect to the
variable $\zeta$, $\tilde u(\zeta(\xi(s))) = u(s)$, and $\tilde \omega
(\zeta) = \omega(\xi(\zeta))$.

Expanding $\tilde u$ and $\tau$ in $\epsilon$ 
\begin{equation}
  \label{eq:u_tilde_expansion_tau_expansion}
  \tilde u(\zeta) = \sum_{i=0} \tilde u_i(\zeta) \epsilon^i, \qquad \tau = \sum_{i=0} \tau_i \epsilon^i,
\end{equation}
substituting into \cref{eq:secondorder_gBT_omega}, and collecting terms of equal
power in $\epsilon$, we obtain the following systems to be solved:
\begin{align}
  &(1-\zeta^2)\tilde\omega_0
    \left((1-\zeta^2)\tilde\omega_0 \tilde u_0^\prime\right)^\prime-\tilde u_0^2 +4 
        {}= 0, \label{eq:0th_order_equation} \\
  &(1-\zeta^2)\left((1-\zeta^2) \tilde u_i^\prime\right)^\prime-2 \tilde u_0
  \tilde u_i   
    + 2 (1 - \zeta^2)\tilde\omega_i \left( (1 - \zeta^2) \tilde u_0^\prime \right)^\prime 
    + (1 - \zeta^2)^2 \tilde u_0^\prime \tilde\omega_i^{\prime}
    \label{eq:ith_order_equation} \\ 
  &\qquad = \tau_{i-1} \left(1- \zeta^2\right) \tilde u_0^\prime 
         + z_i, \qquad i \in \mathbb{N}. \nonumber
\end{align}
Here $z_i$ contains the sums and products of terms in $\tilde u_j,\tilde
\omega_j$ and $\tau_{j-1}$ with $0 \leq j \leq i-1$, with $\tau_{-1}$ is
defined to be zero.

\begin{theorem}
  \label{thm:i_order_solution}
  Equations \cref{eq:0th_order_equation,eq:ith_order_equation} are
  solvable for every $i\in\mathbb N_{0}$, with
  \begin{equation}
    \label{eq:u_i_form}
    \tilde u_i(\zeta) = \sigma_i \zeta^2 + \delta_i,
  \end{equation}
	where $\sigma_i$ and $\delta_i$ are constants to be determined.
\end{theorem}
\begin{proof}
It is easy to see that equation \cref{eq:0th_order_equation} is solvable with
$\sigma_0=6$, $\delta_0=-4$, and $\tilde\omega_0(\zeta) = 1$.

Assume that for $i=1,\dots,n-1$, the systems given by
\cref{eq:ith_order_equation} are solvable for $\tilde\omega_i$ and $\tilde u_i$.
Furthermore, also assume that for $i=1,\dots,n$, $\tilde u_i$ is of the form
\cref{eq:u_i_form}.  We will show that the system \cref{eq:ith_order_equation}
with $i=n$ is solvable for $\tilde \omega_n$.

First notice that \cref{eq:ith_order_equation} is just a first order ordinary
differential equation in $\tilde \omega_i$:
\begin{equation}
\label{eq:omegap_i_as_ode}
  \tilde\omega_i^{\prime} + \frac{2 \left( (1 - \zeta^2) \tilde u_0^\prime \right)^\prime}
				{(1 - \zeta^2) \tilde u_0^\prime} \tilde\omega_i
  = \frac{2 \tilde u_0\tilde u_i - (1-\zeta^2)\left((1-\zeta^2)\tilde u_i^\prime\right)^\prime
    + \tau_{i-1} \left(1- \zeta^2\right) \tilde u_0^\prime + z_i}
				{(1 - \zeta^2)^2 \tilde u_0^\prime}.
\end{equation}
Multiplying by the integrating factor 
\begin{equation}
    \label{eq:integrating_factor}
    (1-\zeta^2)^2\left( \tilde u_0^\prime \right)^2
\end{equation}
and subsequently integrating with respect to $\zeta$ yields the
identity
\begin{align}
  \tilde\omega_i  &=  \frac{ (1-\zeta^2)\left((1-\zeta^2)\tilde u_0^\prime\right)^\prime \tilde u_i 
      - (1-\zeta^2)^2\tilde u_0^\prime \tilde u_i^\prime + (g_i(\zeta)-g_i(1)) }
          { \left( \left(1 -\zeta^2\right) \tilde u_0^\prime \right)^2} \nonumber \\
        &=  -\frac{\sigma_i}{12} 
            - \frac{ (1-\zeta^2)\left((1-\zeta^2)\tilde u_0^\prime\right)^\prime \tilde u_i 
            + (g_i(\zeta)-g_i(1)) }
          { \left( \left(1 -\zeta^2\right) \tilde u_0^\prime \right)^2}
          \label{eq:omega_i_explicit}
\end{align}
where
\begin{align*}
  g_i(\zeta) &= \tau_{i-1} \int \left(1- \zeta^2\right) \left(
               \tilde u_0^\prime \right)^2 \; d\zeta + \int \tilde u_0^\prime z_i \; d\zeta.
\end{align*}
Here we used identity 
\[
\left((1-\zeta^2)\left((1-\zeta^2)\tilde u_0^\prime\right)^\prime \right)' 
=2 \tilde u_0 \tilde u_0^\prime,
\]
obtained from differentiating equation \cref{eq:0th_order_equation} and then
using integrating by parts.  Furthermore, we have chosen the integration
constant $g_i(1)$ such that numerator in \cref{eq:omega_i_explicit} vanishes
for $\zeta=1$. Indeed, for $\tilde\omega_i$ to be well-defined, the numerator in
\cref{eq:omega_i_explicit} must have roots of at least multiplicity two at
$\zeta=0$ and $\zeta=\pm 1$. By setting $\zeta=-1,0$ in the numerator of
\cref{eq:omega_i_explicit}, we obtain the equations
\begin{align}
  0 & = g_i(-1) - g_i(1), \\
  0 & = 12 \delta_i - (g_i(0) - g_i(1)),
\end{align}
respectively. The first equation can be solved explicitly for $\tau_i$. Since
$g_i(0)=0$, it follows that $\delta_i = \frac{g_i(1)}{12}$. To show that the
roots $\zeta=0$ and $\zeta=\pm1$ have multiplicity two, we notice that
differentiation of the numerator in \cref{eq:omega_i_explicit} with respect to
$\zeta$ is equal to multiplying the right-hand side of \cref{eq:omegap_i_as_ode}
with the integrating factor
\cref{eq:integrating_factor}, i.e.,
\begin{equation}
    \label{eq:vanishing_condition_for_double_root}
    \tilde u_0' \left( 2 \tilde u_0\tilde u_i - (1-\zeta^2)\left((1-\zeta^2)\tilde u_i^\prime\right)^\prime
    + \tau_{i-1} \left(1- \zeta^2\right) \tilde u_0^\prime - z_i \right).
\end{equation}
Since $\tilde u_0^\prime = 12 \zeta$, we can factor out $\zeta=0$. Then substituting
$\zeta=\pm1$ into \cref{eq:vanishing_condition_for_double_root}, the following
equation needs to be satisfied
\begin{equation*}
    2 \tilde u_0(\pm1)\tilde u_i(\pm1) + z_i(\pm1) = 0.
\end{equation*}
Notice that this condition is equivalent to the condition obtained by
substituting $\zeta=\pm1$ into \cref{eq:ith_order_equation}. Therefore, by
solving the above equation for either $\pm 1$, yields
\begin{equation}
    \label{eq:sigma_i}
    \sigma_i = -\delta_i - \frac{z_i(1)}{4}.
\end{equation}

Lastly, notice that for $i=1$ we have the solution
\[
\tau_0 = \frac{10}{7}, \quad \sigma_1=0, \quad \delta_1=0,
				\quad \tilde\omega_1(\zeta) = \frac{6}{7} \zeta.
\]
\end{proof}

\begin{corollary}
    \label{corllary:rational_coefficients}
    For $i\in \mathbb N_0$ the polynomials $\tilde \omega_i$ 
    \cref{eq:omega_i_explicit} have rational
    coefficients. Also, the $\tau_i, \sigma_i$ and  $\delta_i$ are rational.
\end{corollary}

\begin{proof}
    The proof follows from a simple induction argument taking into account to
    structure of $z_i$, $i\in\mathbb N$, in \cref{eq:ith_order_equation}.
\end{proof}

\begin{corollary}
\label{corollary:delta_i_sigma_i}
The following relation holds
\[
\sigma_i = \delta_i = \tau_i = 0, \qquad \text{for $i$ odd}.
\]
\end{corollary}
\begin{proof}
From \cref{proposition:symmetry} we have that the branch of non-trivial
homoclinic orbits has the following symmetry
\[
u(-s, \epsilon)=u(s, -\epsilon)
    + \gamma \dot u_0(s),
\quad \tau(\epsilon) = \tau(-\epsilon),
\]
for $s \in \mathbb{R}$ and some open neighborhood of $\epsilon=0$.  Since
$u(s, \epsilon) = \tilde u(\zeta(\xi(s))), \epsilon)$ and
\begin{equation*}
    \tilde u(\zeta(\xi(s)), \epsilon) = \sigma(\epsilon) \zeta^2(\xi(s)) +
    \delta(\epsilon),
\end{equation*}
where $\sigma(\epsilon) = \sum_i \sigma_i \epsilon^i$ and $\delta(\epsilon) =
\sum_i \delta_i \epsilon^i$, it follows that
\begin{multline*}
    \sigma(\epsilon) \zeta^2(\xi(-s)) + \delta(\epsilon) = u(-s, \epsilon)
    = u(s, -\epsilon) + \gamma \dot u_0(s) \\
    = \sigma(-\epsilon) \zeta^2(\xi(s)) + \delta(-\epsilon)
    + \gamma(\epsilon) \left[1-\zeta^2(\xi(s))\right] 12 \zeta(\xi(s)).
\end{multline*}
Therefore, $\sigma, \delta$, and $\tau$ are even functions in $\epsilon$, from
which the assertion follows.
\end{proof}

\begin{corollary}
    \label{corollary:quadraticBTsigma_delta_relation}
    For the quadratic Bogdanov-Takens normal form \cref{eq:universal_unfolding}
    we have the relation that 
    \begin{equation}
        \label{eq:relation_sigma_delta}
        \sigma_i = -\delta_i, \qquad \text{for } i\geq 1.
    \end{equation}
\end{corollary}

\begin{proof}
    Applying the singular rescaling \cref{eq:blowup} to the normal form
    \cref{eq:universal_unfolding}, and consecutive applying the nonlinear time
    transformations \cref{eq:first_non_linear_time_transformation} and
    \cref{eq:second_time_transformation}, we obtain
    \cref{eq:secondorder_gBT_zeta} without the higher-order terms in $\epsilon$.
    After some calculations we obtain the explicit expression for $z_i$ with
    $i\geq 1$ in \cref{eq:ith_order_equation}, namely
    \begin{align*}
    z_i(\zeta) = 
         \sum_{k=1}^{i-1} u_k u_{i-k} 
         + (1-\zeta^2) \left\{ \sum_{l=1}^{i-1} 
          u_l^\prime \tau_{i-1-l}
         + \sum_{k=1}^{i-1} \sum_{l=0}^{i-1-k} 
           \omega_k u_l^\prime \tau_{i-1-l-k} + \right. \\
           \left. 
           \sum_{k=0}^{i-1} \sum_{l=0}^{i-1-k}
           \omega_k u_l^\prime u_{i-1-l-k} 
         - \sum_{l=1}^{i-1}
           \omega_l \left((1-\zeta^2)
           u_{i-l}^\prime \right)^\prime
         - \sum_{k=1}^{i-1} \sum_{l=0}^{i-k}
           \omega_l \left((1-\zeta^2)\omega_k u_{i-l-k}^\prime
           \right)^\prime \right\}. \nonumber \\
    \end{align*}
    From \cref{corollary:delta_i_sigma_i}, we have that $z_1(1)=0$. By
    assuming that the relation \cref{eq:relation_sigma_delta} holds for
    $i=1,2,\dots,n-1$,  $n\in\mathbb N$, we see directly that  $z_i(1)=0$. The
    assertion now follows by \cref{eq:sigma_i}, with $i=n$.
\end{proof}

\begin{remark}
    From \cref{corollary:delta_i_sigma_i} it follows that the solution $\tilde
    u$ for the quadratic Bogdanov-Takens normal form
    \cref{eq:universal_unfolding} can be represented by the single parameter
    $\sigma$
    \begin{equation*}
        \tilde u(\zeta) = 2 - (1-\zeta^2)\sum_{i\geq 0} \sigma_i \epsilon^i.
    \end{equation*}
    Consequently, $\hat u$ becomes
    \begin{equation*}
        \hat u(\xi) = 2 - \sech^2(\xi) \sum_{i\geq 0} \sigma_i \epsilon^i.
    \end{equation*}
\end{remark}

\subsubsection{Third-order orbital homoclinic approximation}
\label{sec:third_order_homoclinic_approximation_LP}
For the third-order homoclinic predictor we obtain
\begin{align}
				\sigma &= 6 + \frac{18}{49}\epsilon^2 + \mathcal{O}(\epsilon^4),
				\nonumber \\
				\delta &=-4 - \frac{18}{49}\epsilon^2 + \mathcal{O}(\epsilon^4),
				\nonumber \\
				\label{eq:tau_orbital}
				\tau   &= \frac{10}{7} + \frac{288}{2401} \epsilon^2 +
								\mathcal{O}(\epsilon^4), \\
				\tilde \omega(\zeta) &= 1 - \frac{6}{7}\zeta \epsilon	+ 
				\left(\frac{9}{98} + \frac{27}{98}\zeta^2 \right) \epsilon^2 +
				\left( -\frac{198}{2401}\zeta + \frac{18}{343}\zeta^3 \right)\epsilon^3+
								\mathcal{O}(\epsilon^4). \nonumber
\end{align}
From which it follows that
\begin{align}
    \label{eq:third_order_uhat}
    \tilde {u}(\zeta) 
    ={}& 2 - \left(1-\zeta^2\right) \left(6 + \frac{18}{49}\epsilon^2 \right) 
        + \mathcal{O}(\epsilon^4), \\
    \label{eq:third_order_vhat}
	\tilde  v(\zeta) 
  ={}& -2 \tilde\omega(\zeta) \sigma (1-\zeta^2)\zeta
	= -\left[ -12 + \frac{72}{7} \zeta \epsilon
			 - \left( \frac{90}{49} + \frac{162 }{49} \zeta^2 \right)\epsilon^2
		 \right. \\
		 & + \left. \left( \frac{3888}{2401} \zeta -
     \frac{216}{343}\zeta^3 \right) \epsilon^3 \right]
     (1-\zeta^2) \zeta + \mathcal{O}(\epsilon^4) \nonumber.
\end{align}

The relation $\xi(s)$ is obtained by solving the ODE
\begin{equation}
				\label{eq:third_order_dxi_ds}
				\frac{d\xi}{ds}(s) = \tilde \omega(\tanh(\xi(s))).
\end{equation}
Thus, we substitute 
\begin{equation*}
				\xi(s) = s + \xi_1(s)\epsilon + \xi_2(s)\epsilon^2
				+ \xi_3(s)\epsilon^3 + \mathcal{O}(\epsilon^4),
\end{equation*}
into \cref{eq:third_order_dxi_ds} and expand the resulting equation in
$\epsilon$ to obtain
\begin{align}
				\frac{d\xi_1}{ds}(s) &= -\frac{6\tanh(s)}{7}, 
            \label{eq:xi_1} \\
				\frac{d\xi_2}{ds}(s) &= \frac{18+54\tanh^2(s)
				-168\xi_1(s)+168\tanh^2(s)\xi_1(s)}{196},
            \label{eq:xi_2} \\
				\frac{d\xi_3}{ds}(s) &= -\frac{198 \tanh(s)}{2401}+
				\frac{18 \tanh^3(s)}{343}-\frac{27}{49}
				(-\tanh(s) \xi_1(s)+\tanh^3(s) \xi_1(s)) \nonumber \\
						& \qquad-\frac{6}{7} (-\tanh(s)
				\xi_1^2(s)+\tanh^3(s) \xi_1^2(s)+\xi_2(s)-\tanh^2(s) \xi_2(s)).
            \nonumber
\end{align}
Here we directly used that $\xi_0(s)=s$. By solving these equations recursively
we obtain
\begin{equation*}
\begin{aligned}
    \xi_1(s) ={}& c_1 - \frac{6}{7}\log(\cosh(s)), \\
    \xi_2(s) ={}& c_2 -\frac{18 s}{49}+\frac{45 \tanh (s)}{98}-\frac{6}{7} c_1
             \tanh (s)+\frac{36}{49} \tanh (s) \log (\cosh(s)), \\
    \xi_3(s) ={}& c_3 + \frac{1}{4802}\left( 3 \sech^2(s) \left(-504 \log^2(\cosh
            (s))-276 \cosh (2 s) \log (\cosh (s)) \right.\right. \\
            {}& + 102 \log (\cosh (s))+14 (18 s-49 c_2) \sinh (2 s)+1176
                c_1 \log (\cosh (s)) \\
            {}& \left. \left. + 546-686 c_1{}^2-441 c_1\right) \right).
\end{aligned}
\end{equation*}
The constants $c_i (i=1,2,3)$ lead to different phase conditions.  A
computationally simple phase condition is given by
\begin{equation}
    \label{eq:xi_i}     
    \xi_i(0)=0,\qquad \mbox{for } i=1,2,3.
\end{equation}
These results in the constraint $v(0)=0$, i.e., the phase condition used
in~\cite{Kuznetsov2014improved}. Solving \cref{eq:xi_i} leads to the solution
\begin{equation}
    \label{eq:c_i_first_phase_condition}
     c_1=0, \qquad c_2=0, \qquad c_3=-\frac{117}{343}.
\end{equation}
Substituting the above expression for $\xi$ into \cref{eq:blowup} we obtain the
third-order predictor
\begin{equation}
\label{eq:third_order_predictor_LP_tau}
\begin{cases}
\begin{aligned}
w_0(\eta)  &= \frac{a}{b^2} 
\tilde {u}\left(\tanh\left(\xi\left(\frac{a}{b}\epsilon\eta\right)\right)\right) \epsilon^2, \\
w_1(\eta)  &= \frac{a^2}{b^3}
\tilde {v}\left(\tanh\left(\xi\left(\frac{a}{b}\epsilon\eta\right)\right)\right) \epsilon^3, \\
\beta_1    &= -4 \frac{a^3}{b^4}\epsilon^4, \\
\beta_2    &= \frac{a}{b}\epsilon^2\tau,
\end{aligned}
\end{cases}
\end{equation}
where $\tau,\tilde{u}$ and $\tilde{v}$ are given by
\cref{eq:tau_orbital,eq:third_order_uhat,eq:third_order_vhat}, respectively.

\begin{remark}
By expanding $\tilde {u}\left(\tanh\left(\xi(s)\right)\right)$ in $\epsilon$ up
to third-order we obtain
\[
    u(s) = u_0(s) + u_1(s) \epsilon + u_2(s) \epsilon^2 + u_3(s) \epsilon^3,
\]
where
\begin{equation}
\label{eq:u_orbital}    
\begin{aligned}
    u_0(s) ={}& 6 \tanh^2(s) - 4, \quad
    u_1(s) = -\frac{72 b \tanh(s) \sech^2(s) \log(\cosh(s))}{7 a}, \\
    u_2(s) ={}& \frac{18}{49} \sech^2(s) \left(-12 s \tanh (s)-24 (\log(\cosh(s))-1) \log(\cosh(s)) \right. \\
              &  \left. +3 \sech^2(s) (32 \log(\sech(s))+12 \log(\cosh(s)) (\log(\cosh(s))+2)-5)+14\right), \\
    u_3(s) ={}& -\frac{27 \sech^5(s)}{2401}  \left(-273 \sinh(s)+91 \sinh(3 s)+84 s \cosh(3 s) (2 \log(\cosh(s))-1) \right. \\
              & -84 s \cosh(s) (6 \log(\cosh(s))-1)-1232 \sinh(s) \log^3(\cosh(s)) \\
              & +112 \sinh(3 s) \log^3(\cosh(s))+2016 \sinh(s) \log^2(\cosh(s)) \\
              & -336 \sinh(3s) \log^2(\cosh(s))+904 \sinh(s) \log(\cosh(s)) \\
              & \left. -104 \sinh(3 s) \log(\cosh(s))\right).
\end{aligned}
\end{equation}
Together with \cref{eq:tau_orbital}, this is precisely the solution obtained
by using the regular perturbation method to
\cref{eq:second_order_nonlinear_oscillator} with phase condition $\dot u(0)=0$.

Note that for the conjecture in~\cite[Section 7]{Al-Hdaibat2016} to
hold, the phase condition \cref{eq:c_i_first_phase_condition} must be satisfied.
\end{remark}

\subsubsection{Non-uniqueness homoclinic solution}
\label{sec:phase_condition}
Note that in \cref{thm:i_order_solution} we could have assumed the solutions of
\cref{eq:ith_order_equation} to be of the form
\begin{equation}
    \label{eq:u_zeta_second_form}
    \tilde u(\zeta)  = \sum_{i\geq 0} \left(\sigma_i \zeta^2 + \delta_i + \gamma_i (1-\zeta^2)
    \tilde u_0^\prime(\zeta) \right) \epsilon^i,
\end{equation}
where $\gamma_0=0$ and $\gamma_i\in\mathbb{R}$ are constants to
be determined by some phase condition. Thus, we have freedom in \cref{eq:xi_i}
and in \cref{eq:u_zeta_second_form} both originating from the non-uniqueness
of the homoclinic orbit. The solutions \cref{eq:xi_1,eq:xi_2} together with
\begin{align*}
    \xi_3(s) =&{}
    \frac{1}{4802} \left(18 \left[49 \gamma_1 (7 \gamma_1+4)+84 s \tanh (s)+92 \log
        (\sech(s))-105\right] \right. \\
        & -7 \sech^2(s) \left[-7 \gamma_1 (7 \gamma_1-3) (35 \gamma_1+9)+18 (7 \gamma_1 \
(7 \gamma_1+4)+9) \log (\sech(s)) \right. \\
        & \left.\left. +216 \log ^2(\cosh (s))-234\right] \right)
\end{align*}
and
\begin{equation*}
    \gamma_1 = \frac{1}{35} \left(-4-\frac{59}{\sqrt[3]{836+15
                \sqrt{4019}}}+\sqrt[3]{836+15 \sqrt{4019}}\right)
\end{equation*}
also leads to the phase condition $\xi_i(0)=0$ for $i=1,2,3$. However,
$v_3(0)=0$ no longer holds. The solutions $\tilde u$ and $\tilde v$ are now
given by
\begin{align*}
\tilde{u}(\zeta) 
={}& 2 + \left(1-\zeta^2\right) \left( -6 + 12 \gamma_1 \zeta  \epsilon
    + \left(6 \gamma_1^2-\frac{18}{49}\right)\epsilon^2 
    + \mathcal{O}(\epsilon^4) \right), \\
\tilde v(\zeta) 
    ={}& (1-\zeta^2) \tilde \omega (\zeta) u'(\zeta) 
    = (1-\zeta^2) \tilde \omega(\zeta) \sum_{i=0} \left( \sigma_i \zeta   +
         12 \gamma_i \left( 1 - 3 \zeta^2 \right) \right) \epsilon^i \\
	  ={}& (1-\zeta^2)\left[12\zeta + 
        \left(12\gamma_1 - 36\frac{2 + 7 \gamma_1\zeta^2}{7}\right) \epsilon +
        \nonumber \right.\\
       & 6\zeta\frac{15 - 168\gamma_1 - 245\gamma_1^2 + 
         3(9 + 7 \gamma_1(16 + 7 \gamma_1))\zeta^2}{49} \epsilon^2 + \nonumber \\
     {}& \left( 216\zeta^2\frac{-18 + 7\zeta^2}{2401} 
        + 6\gamma_1^3(-3 + 2\zeta^2 + \zeta^4) -
      72\gamma_1^2\frac{1 - 6\zeta^2 + 4\zeta^4}7  \right. \nonumber \\
     {}& \left. \left. - 54 \gamma_1\frac{-1 - 6\zeta^2 + 15\zeta^4}{49}
        \right)\epsilon^3 
        + \mathcal{O}(\epsilon^4) \right] \nonumber.
\end{align*}
The numerical simulations in \cref{sec:topological_normal_form} show that for
the normal form \cref{eq:universal_unfolding} these asymptotics are more
accurate than the asymptotics derived in
\cref{sec:third_order_homoclinic_approximation_LP}.

\subsubsection{Comparison with the nonlinear periodic time-reparametrization}
In~\cite{Algaba_2019} a different approach is used to approximate the homoclinic
solution near the quadratic normal form of a generic codimension 2
Bogdanov-Takens bifurcation. The approach there is an application of the
so-called Perturbation-Incremental Method described in~\cite{Xu_1996}. Consider
strongly nonlinear oscillators of the form
\begin{equation}
  \label{eq:nonlinear_oscillator}
  \ddot x + g(x) = \lambda f(x,\dot x, \mu) \dot x, \qquad \
\end{equation}
where $g$ and $f$ are arbitrary nonlinear functions, and
$\lambda$ and $\mu$ are parameters.

The authors in~\cite{Xu_1996} perform a  nonlinear periodic time
reparametrization of the form
\begin{equation}
\label{eq:nonlinear_periodic_reparametrization}  
\frac{d\phi}{dt} =\Phi(\phi), \qquad \Phi(\phi+2\pi)=\Phi(\phi)
\end{equation}
to the system \cref{eq:nonlinear_oscillator}. Then it is assumed that there is a
homoclinic orbit present which can be approximated by the solution
\[
u(\phi) = p \cos(2\phi) + q,
\]
where $p$ and $q$ are constants to be approximated. In~\cite[Theorem
1]{Algaba_2019} it is shown that the solutions $u(\phi)$ and $u(s)$ are related
to each other through
\[
    \Phi(\phi) = \frac{\sqrt 2}{2} \omega(\xi) \sin\phi
.\] 

It follows that one should be able to factor out the term
$\sin{\phi}$ in \cref{eq:nonlinear_periodic_reparametrization}. This is
indeed precisely what we see in the transformation $\Phi$
in~\cite[(41)]{Algaba_2019}. Thus, although the nonlinear periodic time
reparametrization is analytically equivalent to the polynomial generalized
Lindstedt-Poincar\'e method, it is geometrically \emph{less intuitive} than
using hyperbolic functions and computationally \emph{more expensive} than using
polynomials.

We also would like to point out that the singular rescaling
\begin{equation*}
    w_0 = u \epsilon^2, \quad
    w_1 = v \epsilon^3, \quad
    \beta_1 = -\epsilon^4, \quad 
    \beta_2 = \tau \epsilon^2, \quad 
    s = \epsilon \eta, \quad (\epsilon \neq 0),
\end{equation*}
used in~\cite{Algaba_2019} applied to the quadratic Bogdanov-Takens normal form
with coefficient $a=1$ and $b=1$ results in the $\sqrt 2$ turning up in the
calculations of homoclinic approximation. From a computational point of view,
this is less ideal to work with.

\begin{remark}
    Note that the polynomial generalized Lindstedt-Poincar\'e method described
    above only depends on the existence of the zeroth-order solution.
    Therefore, this method can be applied in similar situations where
    homoclinic orbits emanate from codimension two Bogdanov-Takens bifurcation
    points. For example, the transcritical codimension two bifurcation treated
    in~\cite{Hirschberg_1991}, see also~\cite[Appendix C.2]{Bosschaert@2016}.
    Furthermore, under certain symmetry present in the ODE heteroclinic
    solutions can also emanate from codimension two Bogdanov-Takens points
    which can be approximated using the polynomial Lindstedt-Poincar\'e method
    as described above.
\end{remark}

\section{Homoclinic asymptotic expansion in \texorpdfstring{$n$}{n}-dimensional
systems}
\label{sec:homoclinic_asymptotics_n_dimension}
In this section, we will provide third-order approximations to the
homoclinic solution for \cref{eq:ODE} emanating from a generic codimension two
Bogdanov-Takens bifurcation assumed to be at $x_0 \equiv 0$ and $\alpha_0 = 0$.
A distinction between asymptotics derived with the smooth orbital and the smooth
normal form is necessary. In case of the smooth orbital normal form, a further
subdivision is made between the perturbation method used. This is a price we
have to pay for using this simpler normal form. Using the obtained
transformations for lifting the homoclinic orbits from the normal form to the
parameter-dependent center manifold in $n$-dimensional systems, we show the
homoclinic asymptotics arising from the smooth orbital and smooth normal form
are (asymptotically) equivalent in \cref{sec:comparison_homoclinic_predictors}.
We finish this section with our implementation in MatCont.

\subsection{Homoclinic approximation using the smooth orbital normal form}
First, we consider the situation where we have obtained a homoclinic predictor for
the smooth orbital normal form \cref{eq:normal_form_orbital}. Depending on the
used method to approximate the homoclinic solution, i.e., the regular perturbation
or the Lindstedt-Poincar\'e method, we substitute either
\cref{eq:third_order_predictor_RPM_tau} or
\cref{eq:third_order_predictor_LP_tau}, into the parameter-dependent center manifold
transformation $H$ and $K$ defined in \cref{eq:H_expansion,eq:K_expansion}.
By truncating the higher-order terms in $w$ and $\beta$ we obtain the following
approximation $(\bar x^o, \bar \alpha^o)$ to the homoclinic solution
\begin{align}
\label{eq:x_eta_espilon} 
\bar x^o(\eta, \epsilon)={}& q_0w_0(\eta) + q_1w_1(\eta) + H_{0010}\beta_1 + H_{0001} \beta_2 
+ \frac12 H_{2000}w_0^2(\eta) + H_{1100}w_0w_1(\eta) \\
                         & + \frac12 H_{0200}w_1^2(\eta) + H_{1010}w_0(\eta)\beta_1 + H_{1001}w_0(\eta)\beta_2 + H_{0110}w_1(\eta)\beta_1 \nonumber \\
                         & + H_{0101}w_1(\eta)\beta_2 + \frac12 H_{0002}\beta_2^2+ H_{0011}\beta_1\beta_2 + \frac16 H_{3000}w_0^3(\eta) \nonumber \\
                         & + \frac12 H_{2100}w_0^2(\eta)w_1(\eta) + H_{1101}w_0(\eta)w_1(\eta)\beta_2 + \frac12 H_{2001}w_0^2(\eta)\beta_2 \nonumber \\
                         & + \frac{1}{6}H_{0003}\beta_2^3 + \frac12 H_{1002}w_0(\eta)\beta_2^2 + \frac12 H_{0102}w_1(\eta)\beta_2^2, \nonumber \\
\label{eq:alpha_espilon}
\bar \alpha^o(\epsilon) ={}& K_{10}\beta_1 + K_{01}\beta_2 + \frac{1}{2}K_{02}\beta_2^{2} + K_{11}\beta_1\beta_2 + K_{03} \frac16 \beta_2^3.
\end{align}

Next, we use \cref{eq:theta_expansion} to approximate $\eta(t)$ from the relation
\begin{equation}
    \label{eq:dt_deta}
		\frac{dt}{d\eta} = 1 + \theta_{1000}w_0(\eta) + \theta_{0001}\beta_2,
\end{equation}
where $w_0$ and $\beta_2$ are defined in \cref{eq:third_order_predictor_RPM_tau}
when using the predictor obtained by the regular perturbation method and
defined in \cref{eq:third_order_predictor_LP_tau} when using the predictor
obtained by the Lindstedt-Poincar\'e method. We will consider these two
cases separately below.  

\subsection*{Regular perturbation method}
Integrating \cref{eq:dt_deta} with respect to $\eta$ yields
\begin{equation}
\label{eq:tOfTauRegular}
\begin{aligned}
    t(\eta) ={}& \int 1 + \theta_{1000} \frac{a}{b^2}
				u\left(\frac{a}{b}\epsilon \eta \right) \epsilon^2
        + \theta_{0001}\frac{a}{b}\epsilon^2\left(\frac{10}{7}+ \frac{288}{2401}\epsilon^2
				+ \mathcal{O}(\epsilon^4) \right) \, d\eta \\
				={}& \eta \left( 1 + \theta_{0001}\frac{a}{b}\epsilon^2\left(\frac{10}{7}+ \frac{288}{2401} \epsilon^2 + \mathcal{O}(\epsilon^4) \right) \right) + 
				 \theta_{1000} \frac{a}{b^2} \epsilon^2 \int  
				     u\left(\frac{a}{b}\eta\epsilon \right)
					 \, d\eta,
\end{aligned}
\end{equation}
where $u$ is the third-order approximation given in
\cref{eq:third_order_predictor_RPM_tau}. To approximate the integral in the
equation above uniformly in $\eta$, we make the substitution $s = \frac a b \eta
\epsilon$. Then
\begin{equation*}
\begin{aligned}
\int  u\left(\frac{a}{b}\eta\epsilon\right) \, d\eta
= \frac{b}{a} \frac{1}{\epsilon} \int u(s) \, ds,
\end{aligned}
\end{equation*}
where the integral on the right-hand side can be calculated to be
\begin{equation}
\label{eq:u_int_s}
\begin{aligned}
    \int u(s) \, ds &= 2 (s-3 \tanh s) -\frac{9}{7} \sech^2s (\cosh (2 s)-4 \log (\cosh s)-1)\epsilon \\
                    &-\frac{9}{49} \sech^3 s \left[2 \sinh s \left(\cosh (2 s)-12 \log ^2(\cosh s)+6\right)-12 s \cosh s\right]\epsilon^2 \\ 
    &-\frac{27 \sech^4s}{2401} \left[\cosh(4s)+\cosh(2s) \left(-112
    \log^3(\cosh s)+168 \log ^2(\cosh s) \right. \right. \\
    &+188 \left. \log (\cosh s)+7\right)+8 \left(28 \log ^3(\cosh s)-21 \log ^2(\cosh s) \right. \\
    & \left.\left. -29 \log (\cosh s)-21 s \sinh (2 s) \log (\cosh s)-1\right)\right] \epsilon^3
        + \mathcal{O}(\epsilon^4).
\end{aligned}
\end{equation}
Here the constants of integration are calculated such that $t(0)=0$.  Thus, we
obtain a third-order approximation for $t(\eta)$. 

\subsection*{Lindstedt-Poincar\'e}
The method is similar as for the regular perturbation method. We first
integrate \cref{eq:dt_deta} with respect to $\eta$, so that
\begin{equation}
\label{eq:tOfTau}
\begin{aligned}
    t(\eta) ={}& \eta \left( 1 + \theta_{0001}\frac{a}{b}\epsilon^2\left(\frac{10}{7}+ \frac{288}{2401} \epsilon^2 + \mathcal{O}(\epsilon^4) \right) \right) + 
				 \theta_{1000} \frac{a}{b^2} \epsilon^2 \int  
				     \hat{u}\left(\xi\left(\frac{a}{b}\eta\epsilon \right)\right)
					 \, d\eta,
\end{aligned}
\end{equation}
To approximate the integral in the above equation uniformly in $\eta$, we make the
substitution $\tilde{\xi}(\eta) = \xi\left(\frac{a}{b}\epsilon \eta \right)$. Then
\begin{equation*}
\begin{aligned}
\int  \hat{u}\left(\xi\left(\frac{a}{b}\eta\epsilon \right)\right) \, d\eta
= \frac{b}{a} \frac{1}{\epsilon} \int \hat{u}(\tilde\xi)/\omega(\tilde\xi) \,
		d\tilde\xi.
\end{aligned}
\end{equation*}
Expanding the integrand $ \hat{u}(\tilde\xi)/\omega(\tilde\xi)$, up
to order three in $\epsilon$ and integrating with respect to  $\tilde\xi$ yields
\begin{equation}
\label{eq:int_u_s_regular_perturbation}
\begin{aligned}
\int
\frac{1}{\omega(\tilde\xi)} \hat{u}\left(\tilde\xi\right) \, d\tilde\xi
={}& 2 \tilde\xi -6 \tanh (\tilde\xi ) + 
     \left( \frac{18 \sech^2(\tilde\xi )}{7}+\frac{12}{7} \log (\cosh (\tilde\xi ))
		 \right) \epsilon \\ &+
		 \frac{9}{49} \left(4 \tilde\xi -9 \tanh (\tilde\xi )+5 \tanh (\tilde\xi ) \sech^2(\tilde\xi
		 )\right) \epsilon^2 \\ &+
		 \frac{18 \left(-21 \sech^4(\tilde\xi )+47 \sech^2(\tilde\xi )+8 \log (\cosh
		 (\tilde\xi ))\right)}{2401} \epsilon^3 + \mathcal{O}(\epsilon^4).
\end{aligned}
\end{equation}
Substituting $\tilde\xi$ with  $\xi(\frac{a}{b}\eta\epsilon)$ gives the relation
$t(\eta)$ up to order three in $\epsilon$. 

Since we are interested in the inverse relation, i.e., $\eta(t)$, we numerically
solve the equation 
\begin{equation*}
    t(\eta) - t = 0,
\end{equation*}
for $\eta$. This can easily be done within machine precision.

\subsection{Homoclinic approximation using the smooth normal form}
To lift the homoclinic approximation obtained for the smooth normal form to the
parameter-dependent center manifold, we simply substitute either
\cref{eq:third_order_predictor_LP_tau_smooth,eq:third_order_predictor_smooth_RPM_tau}
into $H$ and $K$. Thus, we obtain \cref{eq:x_eta_espilon,eq:alpha_espilon}, where
$\eta$ is replaced by $t$, i.e., 
\begin{align*}
\bar x^s(t, \epsilon)={}& q_0w_0(t) + q_1w_1(t) + H_{0010}\beta_1 + H_{0001} \beta_2 
+ \frac12 H_{2000}w_0^2(t) + H_{1100}w_0w_1(t) \\
                         & + \frac12 H_{0200}w_1^2(t) + H_{1010}w_0(t)\beta_1 + H_{1001}w_0(t)\beta_2 + H_{0110}w_1(t)\beta_1 \nonumber \\
                         & + H_{0101}w_1(t)\beta_2 + \frac12 H_{0002}\beta_2^2+ H_{0011}\beta_1\beta_2 + \frac16 H_{3000}w_0^3(t) \nonumber \\
                         & + \frac12 H_{2100}w_0^2(t)w_1(t) + H_{1101}w_0(t)w_1(t)\beta_2 + \frac12 H_{2001}w_0^2(t)\beta_2 \nonumber \\
                         & + \frac{1}{6}H_{0003}\beta_2^3 + \frac12 H_{1002}w_0(t)\beta_2^2 + \frac12 H_{0102}w_1(t)\beta_2^2, \nonumber \\
\bar \alpha^s(\epsilon) ={}& K_{10}\beta_1 + K_{01}\beta_2 + \frac{1}{2}K_{02}\beta_2^{2} + K_{11}\beta_1\beta_2 + K_{03} \frac16 \beta_2^3.
\end{align*}
Note however that the coefficients of the mappings $H$ and $K$ are calculated
as outlined in
\cref{sec:center-manifold-reduction-without-time-reparametrization}.

\subsection{Comparison between smooth and orbital homoclinic predictors}
\label{sec:comparison_homoclinic_predictors}
Using the above transformations, we show that the homoclinic predictor for the
smooth normal form \cref{eq:BT_smooth_nf}, see
\cref{sec:asymptotics-for-homoclinic-solution-for-the-smooth-normal-form}, is
asymptotically equivalent to the orbital predictor derived in
\cref{sec:third_order_homoclinic_approximation_LP}. Thus, we assume that
\cref{eq:ODE} is given by 
\begin{equation}
    \label{eq:compare:f_normal_form} 
    f(x_1(t), x_2(t), \alpha_1, \alpha_2) := \left(\begin{array}{l}
            x_1(t),\\
            \alpha_1 + \alpha_2 x_1(t) + ax_0^2(t) + b x_0(t)x_1(t) \\
    \end{array}\right),
\end{equation}
where
\begin{equation*}
    g(x_1(t), x_2(t), \alpha_1, \alpha_2) =  a_1\alpha_2 x_0^2(t) + 
    b_1\alpha_2 x_0(t) x_1(t) + ex_0^{2}(t)x_1(t) + dx_0^3(t). \\
\end{equation*}

First we will focus on the predictors for the parameters. Using the procedure
outlined in \cref{subsec:center_manifold_tranformation_orbital}, we obtain that the
coefficients for the parameter transformation $K$ are given by
\begin{equation*}
\begin{gathered}
				K_{10} = \begin{pmatrix} 1 \\ \frac{ae -bd}{a^2} \end{pmatrix},\quad
				K_{01} = \begin{pmatrix} 0 \\ 1 \end{pmatrix},\quad
				K_{11} = \frac{3a_1b-4ab_1+2d}{ab}\begin{pmatrix}  1
								\\ \frac{(ae-bd)}{a^2} \end{pmatrix},\quad \\
				K_{02} = \begin{pmatrix} 0 \\ \frac{2a_1b-2ab_1+d}{ab} \end{pmatrix},\quad
				K_{03} = \begin{pmatrix} 0 \\ 0 \end{pmatrix}.
\end{gathered} 
\end{equation*}

From equation \cref{eq:alpha_espilon} we obtain the following approximation
\begin{equation}
\label{eq:compare:alpha_orbital}
\begin{pmatrix}
\bar \alpha_1^o \\[5pt] \bar \alpha_2^o
\end{pmatrix}
=
\begin{pmatrix}
-\frac{4a^3}{b^4} \epsilon^4
        + \frac{40a^3(-3a_1b+4ab_1-2d)}{7b^6} \epsilon^6
        + \frac{1152 a^3(-3a_1b + 4ab_1 -2d)}{2401 b^6} \epsilon^8
				+ \mathcal{O}(\epsilon^{10}) \\
\frac{10a}{7b} \epsilon^2
				+ \left( \frac{288b}{2401a} + \frac{50a(2a_1b-2ab_1+d)}{49b^3}
				+ \frac{4a(bd-ae)}{b^4} \right) \epsilon^4
				+ \mathcal{O}(\epsilon^{6})
\end{pmatrix}.
\end{equation}
Since the equation \cref{eq:compare:f_normal_form} is the smooth normal form, we can 
directly use \cref{eq:third_order_predictor_LP_tau_smooth} to obtain the
approximation
\begin{equation}
\label{eq:compare:alpha_smooth}
\begin{pmatrix}
\bar \alpha_1^s \\[5pt] \bar \alpha_2^s
\end{pmatrix}
=
\begin{pmatrix}
-\frac{4}{a} \epsilon^4 \\[5pt]
\frac{10b}{7a} \epsilon^2
				+ \frac{98 b (50 a b_1+73 d)-9604 a e-2450 a_1 b^2+288 b^3}{2401 a^2 b} \epsilon^4
				+ \mathcal{O}(\epsilon^6)
\end{pmatrix}.
\end{equation}
To compare these two predictors, we eliminate the parameter $\epsilon$ from both
equations. It would be tempting to first make a substitution for $\epsilon^2$ in
the equations. However, since for the orbits we need odd powers in $\epsilon$,
we will continue without this substitution. We assume $\alpha_1$ to be
positive, which implies that the coefficient $a$ is negative. The case that
$\alpha_1$ is negative is treated similarly and has been verified as well.

To eliminate $\epsilon$ from \cref{eq:compare:alpha_orbital}, we expand $\epsilon$
as a function of $\sqrt[4]{\alpha_1}$. Solving the resulting equation for real
positive $\epsilon$ we obtain
\begin{equation}
    \label{eq:compare:epsilon0}
    \epsilon^o(\alpha_1) = \frac{\sqrt[4]{-a} b}{\sqrt{2}a}\sqrt[4]{\alpha_1}-\frac{5 b  (-4 a
    b_1+3 a_1 b+2 d)}{28 \sqrt{2} a^2 \sqrt[4]{-a}}\sqrt[4]{\alpha_1} ^3 + \mathcal{O}(\sqrt[4]{\alpha_1}^5).
\end{equation}
Obviously, for the smooth predictor we obtain
\begin{equation*}
    \epsilon^s(\alpha_1) = \frac{\sqrt[4]{-a}}{\sqrt2} \sqrt[4]{\alpha_1}.
\end{equation*}
Here we added superscripts $o$ and $s$ to distinguish the different
$\epsilon$'s in the orbital and smooth homoclinic predictors,
respectively.

Substituting $\epsilon=\epsilon^o(\alpha_1)$ into the second equation of
\cref{eq:compare:alpha_orbital} yields
\begin{align*}
    \alpha_2(\alpha_1) ={}& \frac{5b}{7\sqrt{-a}} \sqrt{\alpha_1}
    + \frac{-49 b (50 a b_1+73 d)+4802 a e+1225 a_1 b^2-144 b^3}{4802 a^2} \alpha_1 \\ 
                          & + \mathcal{O}(\alpha_1^{3/2}).
\end{align*}
It can readably be seen that by eliminating $\epsilon$ from
\cref{eq:compare:alpha_smooth} using that $a<0$ we obtain the same expression,
i.e., the predictors agree up to the desired order.

Next, we turn our attention to the approximation of the homoclinic orbits. Due
to the various time transformations involved in the predictors, we compare the
asymptotic expansions of the orbital with the smooth homoclinic predictor.
Therefore, we can directly use the asymptotic obtained from the regular
perturbation method for both predictors. Thus, for the orbital predictor,
we will use 
\begin{equation}
   \label{eq:w0_w1_eta}
   \left\{
   \begin{aligned}
       w_0(\eta)  &= \frac{a}{b^2} \left( \sum_{i=0}^3 u_i(\frac{a}{b}\epsilon\eta) \epsilon^i +
       \mathcal{O}(\epsilon^4) \right)   \epsilon^2, \\
       w_1(\eta)  &= \frac{a^2}{b^3} \left( \sum_{i=0}^3 \dot u_i(\frac{a}{b}\epsilon\eta) \epsilon^i +
       \mathcal{O}(\epsilon^4) \right)   \epsilon^3, \\
   \end{aligned} 
   \right.
\end{equation}
with $u_i(i=1,2,3)$ is given by \cref{eq:u_orbital}, while for the smooth
predictor we will use \cref{eq:third_order_predictor_smooth_RPM_tau} instead.

Following the procedure as outlined in
\cref{subsec:center_manifold_tranformation_orbital}, we obtain that the coefficients
of the transformations $H$ and $\theta$ for the smooth orbital normal form
\cref{eq:normal_form_orbital} are given by
\begin{equation*}
\begin{gathered}
   q_{0} = \begin{pmatrix} 1 \\  0 \end{pmatrix}\!, \quad
   q_{1} = \begin{pmatrix} 0 \\  1 \end{pmatrix}\!, \quad
H_{2000} = \begin{pmatrix} -\frac{d}{2 a} \\ 0 \end{pmatrix}\!, \quad 
H_{1100} = \begin{pmatrix} \frac{-3 b d + 4 a e}{12 a^2} \\  0 \end{pmatrix}\!, \\
H_{0200} = \begin{pmatrix} 0 \\  \frac{-3 b d + 4 a e}{6 a^2} \end{pmatrix}\!, \quad
H_{3000} = \begin{pmatrix} 0 \\  \frac{-3 b d}{2 a} + 2 e \end{pmatrix}\!, \quad
H_{2100} = \begin{pmatrix} 0 \\  \frac{b (-3 b d + 4 a e)}{6 a^2} \end{pmatrix}\!, \\
H_{0010} = \begin{pmatrix} \frac{d}{4 a^2} \\  0 \end{pmatrix}\!, \quad
H_{1001} = \begin{pmatrix} \frac{-2 a b_1 + a_1 b + d}{a b} \\  0 \end{pmatrix}\!, \qquad
H_{0101} = \begin{pmatrix} 0 \\  \frac{-6 a b_1 + 4 a_1 b + 3 d}{2 a b} \end{pmatrix}\!, \\
H_{1101} = \begin{pmatrix} 0 \\  \frac{-3 (6 a b_1 - 4 a_1 b + b^2 - 3 d) d + 4 a b e}{12 a^2 b} \end{pmatrix}\!, \quad
H_{0102} = \begin{pmatrix} 0 \\  \frac{(6 a b_1 - 4 a_1 b - 3 d) (2 a b_1 - 2 a_1 b - d)}{2 a^2 b^2} \end{pmatrix}\!, \\
H_{1010} = \begin{pmatrix} 0 \\  \frac{-3 b d + 4 a e}{12 a^2} \end{pmatrix}\!, \quad
\theta_{1000} = -\frac{d}{2 a}, \quad
\theta_{0001} = -\frac{-2 a b_1 + 2 a_1 b + d}{2 a b},
\end{gathered}
\end{equation*}
while
\begin{align*}
H_{0001} = H_{2001} = H_{0002} = H_{1002} = H_{0003} = H_{0011} = H_{0110} = 
\begin{pmatrix} 0 \\ 0 \end{pmatrix}\!.
\end{align*}

Thus, the third-order homoclinic predictor using the smooth orbital normal form
in $\eta$ is given by 
\begin{equation*}
\begin{aligned}
    \bar x_\epsilon^o(\eta) =& \begin{pmatrix} 1 \\  0 \end{pmatrix} w_0(\eta) +
\begin{pmatrix} 0 \\ 1 \end{pmatrix} w_1(\eta) + 
\begin{pmatrix} -\frac{d}{2 a} \\ 0 \end{pmatrix} w_0^2(\eta) + 
\begin{pmatrix} \frac{-3 b d + 4 a e}{12 a^2} \\  0 \end{pmatrix} w_0(\eta) w_1(\eta) + \\
        & \begin{pmatrix} 0 \\ \frac{-3 b d + 4 a e}{6 a^2} \end{pmatrix} w_1^2(\eta) +
          \begin{pmatrix} 0 \\ \frac{-3 b d}{2 a} + 2 e \end{pmatrix} w_0^3(\eta) +
          \begin{pmatrix} 0 \\ \frac{b (-3 b d + 4 a e)}{6 a^2} \end{pmatrix} w_0^2(\eta) w_1(\eta) + \\
        & \begin{pmatrix} \frac{d}{4 a^2} \\  0 \end{pmatrix} \beta_1 +
          \begin{pmatrix} \frac{-2 a b_1 + a_1 b + d}{a b} \\  0 \end{pmatrix} w_0(\eta) \beta_2 +
          \begin{pmatrix} 0 \\ \frac{-6 a b_1 + 4 a_1 b + 3 d}{2 a b} \end{pmatrix} w_1(\eta) \beta_2 + \\
        & \begin{pmatrix} 0 \\ \frac{-3 (6 a b_1 - 4 a_1 b + b^2 - 3 d) d + 4 a b e}{12 a^2 b} \end{pmatrix} w_0(\eta) w_1(\eta) \beta_2 + \\
        & \begin{pmatrix} 0 \\ \frac{(6 a b_1 - 4 a_1 b - 3 d) (2 a b_1 - 2 a_1 b - d)}{2 a^2 b^2} \end{pmatrix} w_1(\eta) \beta_2^2 + 
          \begin{pmatrix} 0 \\ \frac{-3 b d + 4 a e}{12 a^2} \end{pmatrix} w_0(\eta) \beta_1,
\end{aligned}
\end{equation*}
where $w_{0,1}$ are given by \cref{eq:w0_w1_eta} and $\beta_{1,2}$ by
\cref{eq:third_order_predictor_RPM_tau}.
Since \cref{eq:compare:f_normal_form} is the smooth normal form
\cref{eq:BT_smooth_nf} we obtain from
\cref{eq:third_order_predictor_smooth_RPM_tau} the third-order homoclinic
approximation in $t$
\begin{equation*}
\bar x_\epsilon^s(t) =  \frac1{a} \sum_{i=0}^3
\begin{pmatrix}
 \left(      u_i(\epsilon t) \epsilon^i + \mathcal{O}(\epsilon^4) \right) \epsilon^2 \\
 \left( \dot u_i(\epsilon t) \epsilon^i + \mathcal{O}(\epsilon^4) \right) \epsilon^3
\end{pmatrix}.
\end{equation*}
To relate the smooth orbital predictor $\bar x_\epsilon^o(\eta)$ with the smooth
predictor $\bar x_\epsilon^s(t)$, we need to consider the time transformation
\begin{equation*}
\begin{aligned}
    t_\epsilon(\eta) = \eta \left( 1 + \theta_{0001}\frac{a}{b}\epsilon^2\left(\tau_0 + \tau_2 \epsilon^2 \right) \right) + 
        \theta_{1000} \frac{1}{b} \epsilon \int u(s) \, ds,
\end{aligned}
\end{equation*}
where
\[
    \theta_{1000} = -\frac{d}{2 a}, 
    \quad \theta_{0001} = -\frac{-2 a b_1 + 2 a_1 b + d}{2 a b}, 
    \quad s = \frac{a}{b} \eta \epsilon
\] 
and the integral is given by \cref{eq:int_u_s_regular_perturbation}. We
eliminate $\epsilon$ from $t_\epsilon(\eta)$ by substituting $\epsilon$ by
$\epsilon^o(\alpha_1)$ defined in \cref{eq:compare:epsilon0}. Subsequently, we
substitute $t_{\epsilon^o(\alpha_1)}(\eta)$ into $\bar x_{\epsilon^s(\alpha_1)}^s(t)$.
Thus, we now have two approximations, both parametrized by $\tau$. To compare these
approximations, we first rescale $\eta$ by $\frac{b}{a}\frac{\eta}{\epsilon}$,
otherwise the expansions in $\alpha_1$ become polynomial. We thus arrive at the
following equation which should be satisfied
\begin{equation*}
    \bar x_{\epsilon^s(\alpha_1)}^s\left(t_{\epsilon^o(\alpha_1)}\left(\frac{b}{a}\frac{\eta}{\epsilon}\right)\right) =
    \bar x_{\epsilon^o(\alpha_1)}^o\left(\frac{b}{a}\frac{\eta}{\epsilon} \right) +
    \begin{pmatrix}
        \mathcal{O}(\alpha_1^{3/2}) \\
        \mathcal{O}(\alpha_1^{7/4})
    \end{pmatrix}.
\end{equation*}

Expanding and simplifying the first component of
$\bar{x}_{\epsilon^o(\alpha_1)}^o(\frac{b}{a}\eta/\epsilon)$ in $\alpha_1$
gives
\begin{equation*}
    \begin{aligned}
        &\left(\bar x_{\epsilon^o(\alpha_1)}^o\left(\frac{b}{a}\frac{\eta}{\epsilon}\right)\right)_1 
        = \frac{3 \sech^2\eta-1}{\sqrt{-a}} \sqrt{\alpha_1} +\frac{18 \sqrt{2} b \tanh \eta \sech^2\eta \log (\cosh \eta)}{7 (-a)^{5/4}} \alpha_1^{3/4} \\
        &\qquad \frac{1}{196 a^2}\left[-6 \sech^2\eta \left(7 \left(5 a_1 b+6 b^2+7 d\right)-72 b^2 (\log (\cosh\eta)-1)\log (\cosh\eta) \right. \right. \\ 
        &\qquad \left. -36 b^2 \eta \tanh \eta\right)+70 a_1 b + 9 \sech^4\eta \left(24 b^2 (2-3 \log (\cosh\eta )) \log (\cosh\eta) \right. \\
        &\qquad \left. \left. +30 b^2+49 d\right)+98 d \right] \alpha_1 + \left[ \tanh \eta \left\{-18 b \log (\cosh \eta) \left(-980 a b_1+1225 a_1 b\right. \right. \right. \\ 
        &\qquad \left. -336 b^2(\log (\cosh \eta)-3) \log (\cosh \eta) +312 b^2+1176 d\right)-21 \sech^2\eta \left(-1372 a e \right. \\
        &\qquad +36 b \log (\cosh \eta) \left(6 b^2 \log (\cosh \eta) (4 \log (\cosh \eta)-7)-18 b^2-49 d\right)+234 b^3 \\
        &\qquad \left. \left. +1029 b d\right)-9604 a e+4914 b^3+7203 b d\right\} -4536 b^3 \eta  \left(-2 \log (\cosh \eta) \right. \\
        &\qquad \left. \left.+\sech^2\eta (3 \log (\cosh \eta)-1)+1\right)\right] \frac{\sech^2\eta}{4802 \sqrt{2} (-a)^{11/4}} \alpha_1^{5/4}
            + \mathcal{O}(\alpha_1^{3/2}).
    \end{aligned}
\end{equation*}
Similarly, for the second component of $\bar x_{\epsilon^o(\alpha_1)}^o(\eta/\epsilon)$ we obtain
\begin{equation*}
    \begin{aligned}
        & \left(\bar x_{\epsilon^o(\alpha_1)}^o\left(\frac{b}{a}\frac{\eta}{\epsilon}\right)\right)_2 
        = \frac{3 \sqrt{2} \tanh \eta \sech^2\eta}{\sqrt[4]{-a}} \alpha_1^{3/4} + \frac{9 b \sech^4\eta }{7 a}(\cosh (2 \eta)-2 (\cosh (2 \eta)-2) \\
        &\qquad \log (\cosh \eta)-1) \alpha_1 + -\frac{3 b \sech^4\eta}{196 \sqrt{2} (-a) ^{7/4}} (\sinh (2 \eta) (35 a_1-144 b (\log(\cosh \eta)-2) \\
        &\qquad \log (\cosh \eta)+48 b)+72 b (2 \eta-\eta \cosh (2 \eta)+\tanh \eta (2 \log (\cosh \eta) (6 \log (\cosh \eta)-7)\\
        &\qquad -3))) \alpha_1^{5/4} +\left(2 \left(36 b \log (\cosh \eta)  \left(490 a b_1-490 a_1 b+84 b^2 \log (\cosh \eta) (2 \log ( \right. \right. \right. \\
        &\qquad \left. \left. \cosh \eta)-9)+222 b^2-245 d\right)+90 b \left(-98 a b_1+98 a_1 b+111 b^2\right)-9604 a e \right. \\
        &\qquad \left. +11613 b d\right) + 3 \sech^2\eta \left(36 b \log (\cosh \eta) \left(245 (d-2 a b_1)+490 a_1 b+168 b^2 (14 \right. \right. \\
        &\qquad \left. -5 \log (\cosh \eta)) \log (\cosh \eta)-138 b^2\right)-3 b \left(-1960 a b_1+1960 a_1 b +6462 b^2 \right. \\
        &\qquad \left. +12985 d\right)+7 \sech^2\eta \left(-6860 a e+216 b^3 \log (\cosh \eta) (\log (\cosh\eta) (20 \log(\cosh \eta) \right. \\
        &\qquad \left. \left. -47)-1)+1818 b^3+5145 b d\right) + 48020 a e\right)-4536 b^3 \eta \tanh \eta \left(-4 \log (\cosh \eta) \right. \\
        &\qquad \left. \left. + \sech^2\eta (12 \log(\cosh \eta)-7)+4\right)\right) \frac{\sech^2\eta}{9604 (-a)^{5/2}}  \alpha_1^{3/2}
            + \mathcal{O}(\alpha_1^{7/4}).
    \end{aligned}
\end{equation*}
By expanding $\bar
x_{\epsilon^s(\alpha_1)}^s\left(t_{\epsilon^o(\alpha_1)}\left(\frac{b}{a}\frac{\eta}{\epsilon}\right)\right)$
and simplifying we obtain 
\begin{align*}
    \bar x_{\epsilon^s(\alpha_1)}^s\left(t_{\epsilon^o(\alpha_1)}\left(\frac{b}{a}\frac{\eta}{\epsilon}\right)\right) 
    &= \bar x_{\epsilon^o(\alpha_1)}^o\left(\frac{b}{a}\frac{\eta}{\epsilon} \right) +
        \begin{pmatrix}
            \frac{(3 b d - 4 a e) \sech^2(\eta) \tanh( \eta )}{\sqrt2 (-a)^{ 11/4 }} \alpha_1^{5/4} \\
            \frac{(3 b d - 4 a e)(\cosh (2 \eta )-2) \sech^4(\eta) }{2 (-a)^{5/2}} \alpha_1^{3/2}
        \end{pmatrix} \\
    &= \bar x_{\epsilon^o(\alpha_1)}^o\left(\frac{b}{a}\frac{\eta}{\epsilon} \right) +
        \begin{pmatrix}
            \frac{(3 b d - 4 a e) u_0'(\eta)}{24\sqrt2 (-a)^{ 11/4 }} \alpha_1^{5/4} \\
            \frac{(3 b d - 4 a e) u_0''(\eta)}{24 (-a)^{5/2}} \alpha_1^{3/2}
        \end{pmatrix},
\end{align*}
i.e., the predictors differ by a phase shift. In fact, by using the freedom in
the constants of integration in \cref{eq:u_int_s}, we can let 
\[
    t(\eta) \rightarrow t(\eta) + \theta_{1000} \frac{1}{b} \frac{2}{3}
    \left(\frac{4 a e}{b d}-3\right) \epsilon^2.
\] 
In this case, we will have equivalence between the predictors up to the desired
order. 

\begin{remark}
It is important to note here the phase condition used in the orbital
predictor isn't preserved under the transformation $H$. Therefore,
any improvements obtained in the approximation to the homoclinic solutions in a
normal form due to a different phase condition is, in general, not preserved
when lifting the approximations to the center manifold.
\end{remark}

\subsection{Implementation}%
\label{sec:implementation}
In this section, we first briefly review the method used in MatCont to continue
homoclinic solutions in autonomous ordinary differential equations 
\cref{eq:ODE} in two parameters as described in~\cite{DeWitte2012}. Then we
describe our implementation in MatCont of the algorithm to start the continuation of
the homoclinic solutions emanating from a codimension two Bogdanov-Takens point
using the derived above homoclinic predictors and parameter-dependent center manifold.

\subsubsection{Continuation of homoclinic solutions in MatCont}
\label{sec:cont_hom_matcont}
MatCont uses a correction-prediction continuation method applied to a defining
system~\cite{MatCont}. The defining system for the continuation of homoclinic
solutions in ordinary differential equations of the form \cref{eq:ODE} in two
parameters in MatCont are given by
\begin{equation}
\begin{aligned}
    \label{eq:definingSystem}
    \dot{x}(t)-2 T f(x(t), \alpha)=0,  \\
    f\left(s_{0}, \alpha\right)=0, \\
    \int_{0}^{1} \widetilde{x}(t)[x(t)-\widetilde{x}(t)] d t=0, \\
    Q^{U^{\perp}, \mathrm{T}}\left(x(0)-s_{0}\right)=0, \\
    Q^{S^{\perp}, \mathrm{T}}\left(x(1)-s_{0}\right)=0, \\
    T_{22 U} Y_{U}-Y_{U} T_{11 U}+T_{21 U}-Y_{U} T_{12 U} Y_{U}=0, \\
    T_{22 S} Y_{S}-Y_{S} T_{11 S}+T_{21 S}-Y_{S} T_{12 S} Y_{S}=0, \\
    \left\|x(0)-s_{0}\right\|-\epsilon_{0}=0, \\
    \left\|x(1)-s_{0}\right\|-\epsilon_{1}=0,
\end{aligned}
\end{equation}
see~\cite{DeWitte2012}.  Here the infinite time interval
$\interval[open]{-\infty}{ \infty}$ of the homoclinic orbit is truncated to a
finite interval $[-T,T]$, where $T>0$ is called the half-return time. The
truncated interval is rescaled to the interval $[0,1]$ and divided into
\mintinline{matlab}{ntst} mesh-intervals. Each mesh interval is further
subdivided by equidistant fine mesh points  where the solution is approximated
by a vector polynomial. Each mesh interval contains a number of
\mintinline{matlab}{ncol} collocation points where the first equation in
\cref{eq:definingSystem} must be satisfied. The second equation in
\cref{eq:definingSystem} locates the saddle point $s_0$ of the homoclinic
orbit. The last two equations in \cref{eq:definingSystem} define the distance
$\epsilon_0$ and  $\epsilon_1$ between the saddle and the homoclinic solution
at $t=0$ and $t=1$, respectively. The half-return time $T$, $\epsilon_0$,
$\epsilon_1$ are referred to as the \emph{homoclinic parameters}. Either one or
two of the homoclinic parameters must be allowed to vary. If two homoclinic
parameters are selected to vary, the third equation (the phase condition) in
\cref{eq:definingSystem} is added. The fourth and fifth equations in
\cref{eq:definingSystem} place the solution at the endpoints in the unstable
and stable eigenspace of linearization of the saddle point, respectively. The
matrices $Q^{U^{\perp}} \in \mathbb{R}^{n \times n_{S}}$ and $Q^{S^{\perp}} \in
\mathbb{R}^{n \times n_{U}}$ are not recalculated each continuation step, but
constructed from the lower dimensional matrices $Y_{U} \in \mathbb{R}^{n_{S}
\times n_{U}}$ and $Y_{S} \in \mathbb{R}^{n_{U} \times n_{S}}$. The fifth and
sixth equations in \cref{eq:definingSystem}, referred to as algebraic Ricatti
equations, keep track of the lower dimensional matrix $Y_U$ and $Y_S$,
see~\cite{Friedman2001Continuation}. The matrices $Y_U$ and $Y_S$ are initially
set to zero.

Thus, in order to start continuation of homoclinic solutions near a codimension
two Bogdanov-Takens point in \cref{eq:ODE}, one needs to provide an initial
approximation to
\begin{itemize}
    \item the discretized orbit on the rescaled and truncated interval $[0,1]$,
    \item the parameter values $\alpha$,
    \item the saddle point $s_0$,
    \item the half-return time $T$,
    \item the initial distances $\epsilon_0$ and $\epsilon_1$,
    \item and an initial tangent vector for the next prediction.
\end{itemize}

Since the homoclinic predictors depend on the coefficients of the normal form,
we first calculate the parameter-dependent center manifold transformation.

\subsubsection{Multilinear forms}
Unfortunately, not all multilinear forms in the expansion
\cref{eq:f_expansion} needed for the derivation of the coefficients for the
transformations $H$, $K$, and $\theta$, were previously implemented in MatCont.  We,
therefore, developed new scripts which generate the necessary multilinear
forms symbolically if the symbolic toolbox for MATLAB or for GNU Octave is
installed.  The multilinear forms can be generated with the graphical user
interface of MatCont, or via the command-line interface.  Examples are given in
the Supplementary Materials.  If the symbolic toolbox is not
available, finite differences in combination with polarization identities are
used instead.  Note that symbolical derivatives generated with either GNU
Octave or MatCont can be used interchangeably.

\subsubsection{Coefficients of the parameter-dependent center manifold}
Next, we compute the coefficients in the expansion $G, H, K$, and $\theta$ as
derived in \cref{sec:Center_manifold_reduction_ODE}. Since we allow
using finite differences, the results may become inaccurate. We, therefore,
provide a warning message if one or more of the systems are not satisfied with
a prescribed accuracy. The code for calculating the coefficients of the
orbital, smooth, and hyper normal form can be found in the scripts
\mintinline{julia}{BT_nmfm_orbital.m}, \mintinline{julia}{BT_nmfm.m}, and
\mintinline{julia}{BT_nmfm_without_e_b1.m}, respectively. These scripts can be
called independently once a Bogdanov-Takens point is located.

\subsubsection{The perturbation parameter and the half-return time}
To provide the data listed at the end of \cref{sec:cont_hom_matcont}, one first
needs to select a suitable perturbation parameter $\epsilon$ in the homoclinic
predictors in \cref{sec:homoclinic_asymptotics_n_dimension}.
In~\cite{Al-Hdaibat2016} a geometrically motived approach is used to determine
the perturbation parameter $\epsilon$ and half-return time $T$. The user first
provides the amplitude $A_0$ of the homoclinic orbit, which can be approximated
by
\begin{equation*}
    A_0 = \|\bar x(0,\epsilon)-s_0\|.
\end{equation*}
From the amplitude $A_0$ the initial perturbation parameter $\epsilon$ can then
be estimated by truncating the homoclinic predictors in
\cref{sec:homoclinic_asymptotics_n_dimension} up to second order in $\epsilon$.
Then, by using that the norm of the eigenvector $q_0$ is of unit length, see
\cref{eq:q0}, we obtain
 \[
     \epsilon^o  = |b|\sqrt{ \frac{A_0}{6|a|}},
     \quad \mbox{and} \quad 
     \epsilon^s  = \sqrt{ \frac{A_0|a|}{6}},
\] 
for the orbital and smooth homoclinic predictor, respectively.

Secondly, the user provides the distance $k$, also referred to as
\mintinline{matlab}{TTolerance} in MatCont, between the endpoints of the
truncated homoclinic orbit and the saddle point
\begin{equation}
    k = \|\bar x(\pm T,\epsilon)-s_0\|. \label{eq:TTolerance}
\end{equation}
From this equation the half-return time $T$ is solved by again truncating the
homoclinic predictors given in \cref{sec:homoclinic_asymptotics_n_dimension} up
to second order in $\epsilon$. We obtain
 \[
     \eta(T) = \left|\frac{b}{a} \frac{1}{\epsilon}
     \arcsech\left(\frac{|b|}{\epsilon} \sqrt{\frac{k}{6|a|}}\right)\right|,
     \quad \mbox{and} \quad 
     T  = \frac{1}{\epsilon} \arcsech\left(\sqrt{\frac{k}{A_0}}\right),
\] 
for the orbital and smooth homoclinic predictor, respectively. Thus, for the
orbital predictor, the half-return time $T$ has to be computed by numerical
inverting $\eta(T)$, either using \cref{eq:tOfTauRegular} or \cref{eq:tOfTau},
depending on the perturbation method used.

Note that by truncating the homoclinic orbit up to second order, only
the zeroth-order solution of $u$ is used. Furthermore, since only the
eigenvector $q_0$ is used in the approximation of the amplitude $A_0$ and the
half-return time $T$, one should not expect to obtain an accurate approximation
of the amplitude in general $n$-dimensional systems. For the half-return time
$T$, there is an additional problem in the approximation. Namely, the homoclinic
predictors are not even functions in $t$, i.e., we have the inequality
\begin{equation*}
    \|\bar x(T,\epsilon)-s_0\| \neq \|\bar x(-T,\epsilon)-s_0\|.
\end{equation*}
Thus, the more accurate interpretation of the amplitude $A_0$ and half-return
time $T$ is that they represent approximations for the amplitude and
half-return time of $u_0$ on the center subspace and not for the homoclinic
orbit on the center manifold. This, however, does not influence the convergence
of the initial prediction for the homoclinic orbit, as long as the derived
perturbation parameter $\epsilon$ is within the radius of convergence.

Instead of requiring the user to provide the amplitude $A_0$ and distance $k$,
we determine the perturbation parameter $\epsilon$ automatically.  Motivated by
\cref{sec:case_study_BT2} an initial good guess for the orbital homoclinic
predictor would be obtained with $\epsilon=0.1$. The higher-order terms, not
taken into account for the predictor, would have to behave very badly to not
lead to convergence. The distance $k$ we set to $k = \epsilon 10^{-4}$.  In
case that Newton does not converge, the perturbation $\epsilon$ is halved and
$k$ is updated.  This process is then repeated until convergence is obtained,
or by the maximum prescribed number of tries. In fact, for all buy one model in
\cref{sec:examples} and the Supplementary Materials, setting the perturbation
parameter $\epsilon=0.1$ and the distance $k=10^{-4}$ while using the orbital
predictor with either the Lindstedt-Poincar\'e or the regular perturbation
method lead to convergence to the true homoclinic solution.

\subsubsection{Homoclinic solution}
After the perturbation parameter $\epsilon$ and the half-return time $T$ have
been determined and using the coefficients of $H,K$, and possibly $\theta$, the
homoclinic approximation is obtained by evaluating one of the homoclinic
approximations given in \cref{sec:homoclinic_asymptotics_n_dimension} on the
fine mesh points $f_i$ for $i=0,\dots \text{ntst}\times\text{ncol}$ on the
rescaled discretized interval $[0,1]$.  Note that we do need to translate the
approximation to the Bogdanov-Takens phase point $x_0$ under consideration.

\subsubsection{Saddle point}
To obtain an approximation for the saddle point of the homoclinic orbit, we
simply let $t$ go to infinity in the homoclinic approximation obtained in the
previous step. Note that $\tau$ goes to infinity as $t$ does. Thus, using
\cref{eq:x_eta_espilon} we obtain
\begin{equation*}
\begin{aligned}
    s_0 ={}& x_0 + q_0 w_0^{\infty} + \bar H_{0010} \beta_1 + \bar H_{0001} \beta_2 +
            \frac{1}{2} \bar H_{2000} \left(w_0^{\infty}\right)^2  \\
           & + \bar H_{1010} w_0^{\infty} \beta_1 + \bar H_{1001} w_0^{\infty} \beta_2 +
             \frac{1}{2} \bar H_{0002} \beta_2^2 + \bar H_{0011} \beta_1 \beta_2 \\
           & + \frac{1}{6} \bar H_{3000} \left(w_0^{\infty}\right)^3 +
             \frac{1}{2} \bar H_{2001} \left(w_0^{\infty}\right)^2 \beta_2 
             + \frac{1}{6} \bar H_{0003} \beta_2^3 + \frac{1}{2} \bar H_{1002} w_0^{\infty} \beta_2^2
\end{aligned}
\end{equation*}
where
\[
  w_0^\infty = 2\frac{a}{b^2} \epsilon^2,
\] 
for the orbital predictor the approximation and
\[
    w_0^\infty = \frac1a \left( 2 - 2 \frac{5 a_1 b + 7 d}{7a^2} \epsilon^2
    \right) \epsilon^2,
\] 
for the smooth predictor.

\subsubsection{Tangent}
Once a good initial prediction for the homoclinic solution to
\cref{eq:definingSystem} has been obtained, a normalized tangent vector is
needed to start the continuation process. The simplest method to obtain the
normalized tangent vector is by calculating the one-dimensional null space of
the sparse rectangular Jacobian of the defining system \cref{eq:definingSystem}.
Although this is easy to implement numerically via a QR-decomposition of the
transpose, the drawback is that we do not have any control of the orientation
of the tangent vector. Thus, we do not know in which direction the continuation
starts.  Obviously, we want to avoid continuing the homoclinic curve towards
the Bogdanov-Takens point. One way to obtain the correct direction is by
obtaining two approximations with close, but different perturbation parameters,
from which the normalized tangent can be approximated.  Computationally cheaper
and more accurate is to first compute the null space $V$ of the Jacobian as
described above and subsequently inspect the sign of the $\alpha_1$ component
in the vector $V$. This sign should be equal to the sign of the derivative of
the first component of $\bar\alpha$ with respect to $\epsilon$, see
\cref{eq:alpha_espilon}. A simple calculation yields
\[
\bar\alpha'(\epsilon)  = K_{10} \beta_1'(\epsilon) + K_{01} \beta_2'(\epsilon) +
     K_{02} \beta_2 \beta_2'(\epsilon) + K_{11} (\beta_1'(\epsilon) \beta_2 +
     \beta_1 \beta_2'(\epsilon)) + \frac{1}{2} K_{03}\beta_2^2 \beta_2'(\epsilon),
\] 
where $\beta_1'(\epsilon)$ and $\beta_2'(\epsilon)$ are given by
\[
\left\{
\begin{aligned}
    \beta_1'(\epsilon) ={}& -16 \frac{a^3}{b^4} \epsilon^3, \\
    \beta_2'(\epsilon) ={}& \frac{a}{b} (2 \tau_0 + 4 \tau_2 \epsilon^2) \epsilon, \\
\end{aligned}
\right.
\] 
with $\tau_{0,1}$ are given in \cref{eq:tau_orbital} for the orbital
predictor, and by
\[
\left\{
\begin{aligned}
    \beta_1'(\epsilon) ={}&  -\frac{16}{a} \epsilon^3, \\
    \beta_2'(\epsilon) ={}& \frac{b}{a} (2 \tau_0 + 4 \tau_2 \epsilon^2) \epsilon, \\
\end{aligned}
\right.
\] 
where $\tau_{0,1}$ are given in \cref{eq:tau_smooth} for the smooth homoclinic
predictor. Thus, we simply change the sign of the vector $V$ if the product of
the $\alpha_1$ component of $V$ with the first component of $\bar\alpha'(\epsilon)$
is negative, otherwise, we leave the vector $V$ unchanged.

\section{Examples}
\label{sec:examples}

In this section, we will compare the different methods at different orders to
approximate the homoclinic solution in \cref{eq:ODE} near a generic codimension two
Bogdanov-Takens bifurcation point. In the first example, we consider the
topological normal form \cref{eq:universal_unfolding}. By using convergence plots
we will show that by consi\-de\-ring different phase conditions in both the regular
perturbation method and the Lindstedt-Poincar\'e method influences the accuracy
of the approximation, see \cref{fig:RP_vs_RPL2} and \cref{fig:LPM_vs_LPM},
respectively. In this example, we will also compare the regular perturbation
method against the Lindstedt-Poincar\'e method with a higher-order
approximation of the non-linear time transformation as derived
in this paper, or without, as done in~\cite{Al-Hdaibat2016}. We do this twofold:
using convergence plots, see \cref{fig:RP_vs_LP2016_vs_LP}, and by inspecting
the predicted homoclinic profiles with the Newton corrected homoclinic
profiles, see \cref{fig:RP_vs_LP2016_vs_LP_profiles}.

Next, we will consider two four-dimensional models in which generic codimension
two Bogdanov-Takens bifurcations are present. Here we will show that the
approximation order of the homoclinic asymptotic lifts correctly to the
parameter-dependent center manifold, see
\cref{fig:HodgkinHuxleyConvergencePlot,fig:Lorenz84ConvergencePlot}. In the
second example, we also compare the predicted with the Newton corrected
homoclinic orbits in a projection onto a three-dimensional slice of the full
system. It is shown that with an amplitude of 0.1 the approximation is still
very accurate. This should be compared with~\cite[section 6.2]{Al-Hdaibat2016}
in which the amplitude needed to be set to $4\times 10^{-5}$ to obtain
convergence.

Here we will not demonstrate how to actually start the continuation of the
homoclinic orbits with MatCont. For this we refer to the 
\href{https://mmbosschaert.github.io/MatCont7p2NewInitBTHom-/}{online Jupyter Notebook}.
In the Jupyter Book, a total of nine different models are considered
demonstrating in detail how to start continuation from either an explicitly
derived, or encountered during continuation, codimension two Bogdanov-Takens
bifurcation point. Each model is treated in a separate Jupyter Notebook, which
can be executed to reproduce the results obtained. We do like to note that in
all cases the curve of homoclinic solutions could be started with the default
settings without the need to adjust any parameters, see
\cref{sec:implementation}. This shows that the asymptotics obtained in this
paper are very robust.

\subsection{Topological normal form}
\label{sec:topological_normal_form}

\begin{figure}
    \centering
    \ifcompileimages%
  \tikzsetnextfilename{BTParameterdependentnormalformConvergencePlotRPMvsRPM}%
  \input{tikz/BTParameterdependentnormalformConvergencePlotRPMvsRPM}%

    \else
        \includegraphics{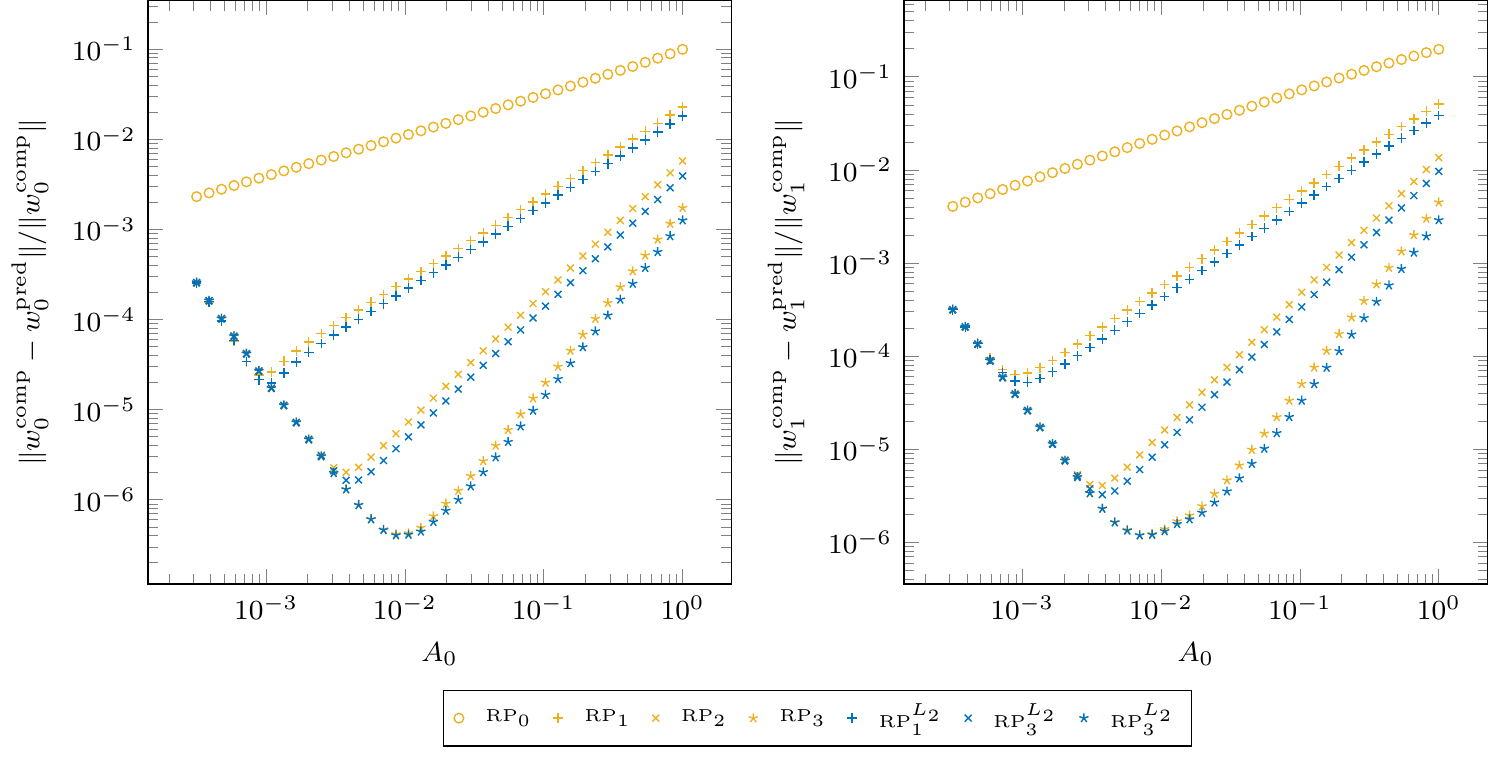}
    \fi
    \caption{Log-log convergence plot comparing the different phase conditions
        when using the regular perturbation method (RP) for approximation the
        homoclinic solutions. The subscript in RP$_i$, $0\leq i \leq 3$, refers
        to the approximation order. The superscript $L_2$ refers to the phase
        condition \cref{eq:u_i_L2_phase_condition}. } 
    \label{fig:RP_vs_RPL2}
\end{figure}

In this example we compare five different methods to approximate the homoclinic
solution present in the universal unfolding \cref{eq:universal_unfolding}: 
\begin{itemize}
    \item the regular perturbation method, 
    \item the regular perturbation method with $L_2$ phase condition,
    \item the Lindstedt-Poincar\'e method without a higher-order time
        approximation as in~\cite{Al-Hdaibat2016},
    \item the Lindstedt-Poincar\'e method with a higher-order time
        approximation as derived here, 
    \item and the Lindstedt-Poincar\'e method with a different phase condition.
\end{itemize} 

In \cref{fig:RP_vs_RPL2} a log-log convergence plot is shown comparing the
asymptotics derived in~\cite{Al-Hdaibat2016} using the regular perturbation
method with phase condition $\dot u = 0$ against the asymptotics derived here
with the phase condition given in \cref{eq:u_i_L2_phase_condition}.  On the
abscissa is the amplitude $A_0$ and on the ordinate is the relative error
$\delta$ between the components $w_0$ and $w_1$ of the predicted solution and
the Newton corrected solution. We see that the $L_2$ phase condition is
slightly, but noticeably, more accurate at each order, confirming the geometric
intuition.

Next, we compare the regular perturbation method with the Lindstedt-Poincar\'e
method to approximate the homoclinic solution in log-log plot in
\cref{fig:RP_vs_LP2016_vs_LP}.  It is seen that the first order regular
perturbation method slightly outperforms the Lindstedt-Poincar\'e method, while
for the second and third-order the Lindstedt-Poincar\'e method  are clearly
better approximations than the regular perturbation method.  The third-order
approximation by the Lindstedt-Poincar\'e method without including a
higher-order approximation of the non-linear time transformation results in the
same order of accuracy as the zeroth-order regular perturbation method.

\begin{figure}
    \centering
    \ifcompileimages%
  \tikzsetnextfilename{BTParameterdependentnormalformConvergencePlot}%
  \input{tikz/BTParameterdependentnormalformConvergencePlot}%

    \else
        \includegraphics{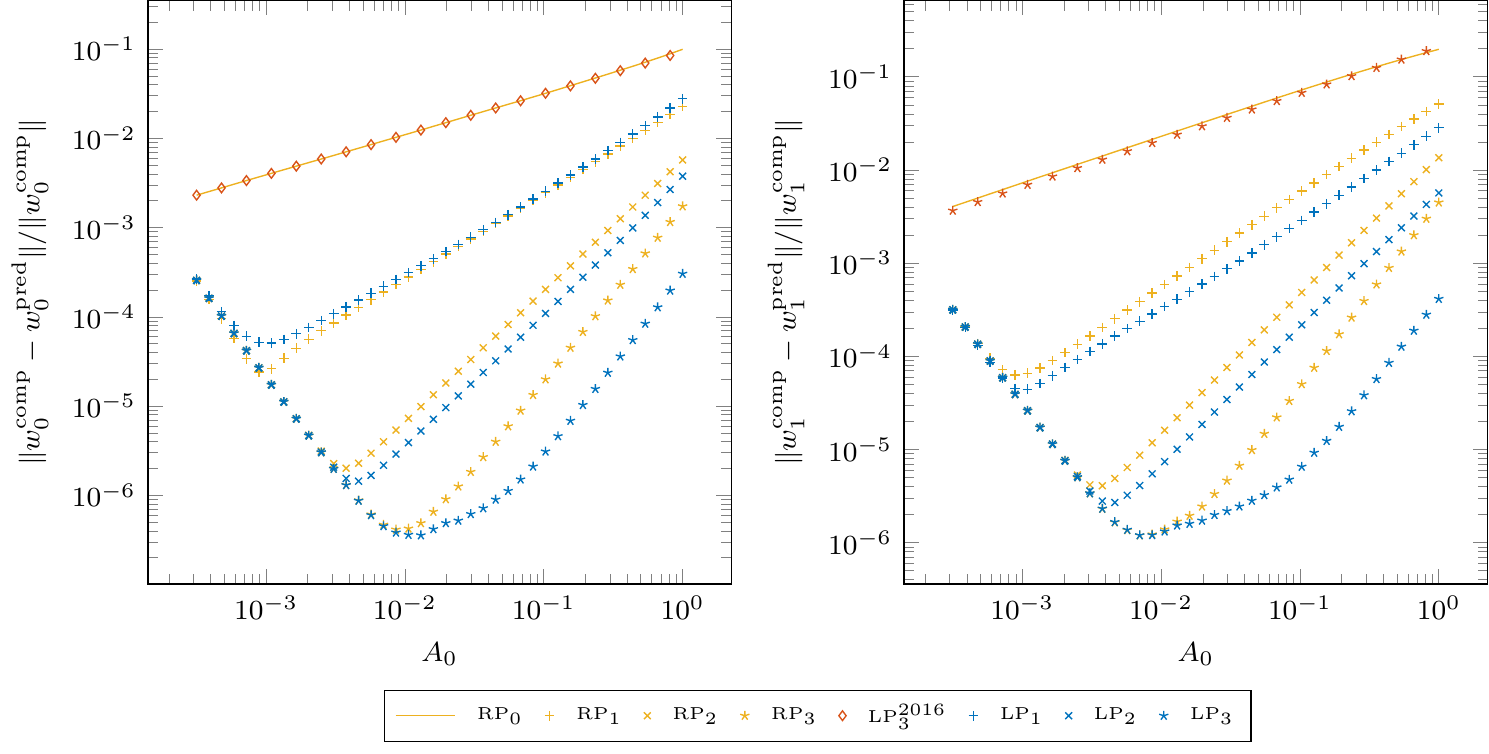}
    \fi
    \caption{Log-log convergence plot comparing the relative errors of the computed
        homoclinic $w_0$ and $w_1$ component with the predicted solution in the
        topological normal form using four different methods: Regular
        Perturbation ($RP$, yellow), Lindstedt-Poincar\'e without higher-order time
        approximation ($LP_3^{2016}$, red), and Lindstedt-Poincar\'e combined
        with higher-order time approximation ($LP$, blue).}
    \label{fig:RP_vs_LP2016_vs_LP}
\end{figure}

It is thus essential to include a higher-order approximation of the non-linear
time transformation. To make it even more clear we plotted the profiles of the
third-order approximations using the Lindstedt-Poincar\'e method as
in~\cite{Al-Hdaibat2016}, the regular perturbation method, and the
Lindstedt-Poincar\'e, together with the Newton corrected solutions in
\cref{fig:RP_vs_LP2016_vs_LP_profiles}. We see that the Lindstedt-Poincar\'e
method as in~\cite{Al-Hdaibat2016} approximates the solution rather poorly,
whereas the approximation derived in
\cref{sec:third_order_homoclinic_approximation_LP} is very accurate. Note that
when plotting these homoclinic approximations and corrections in $(w_0,w_1)$
phase-space, this difference is not visible at all. This explains why this has been
unnoticed in~\cite{Al-Hdaibat2016}.

\begin{figure}[!ht]
    \centering
    \ifcompileimages%
  \tikzsetnextfilename{BTParameterdependentnormalformCompareProfiles}%
  \input{tikz/BTParameterdependentnormalformCompareProfiles}%

    \else
        \includegraphics{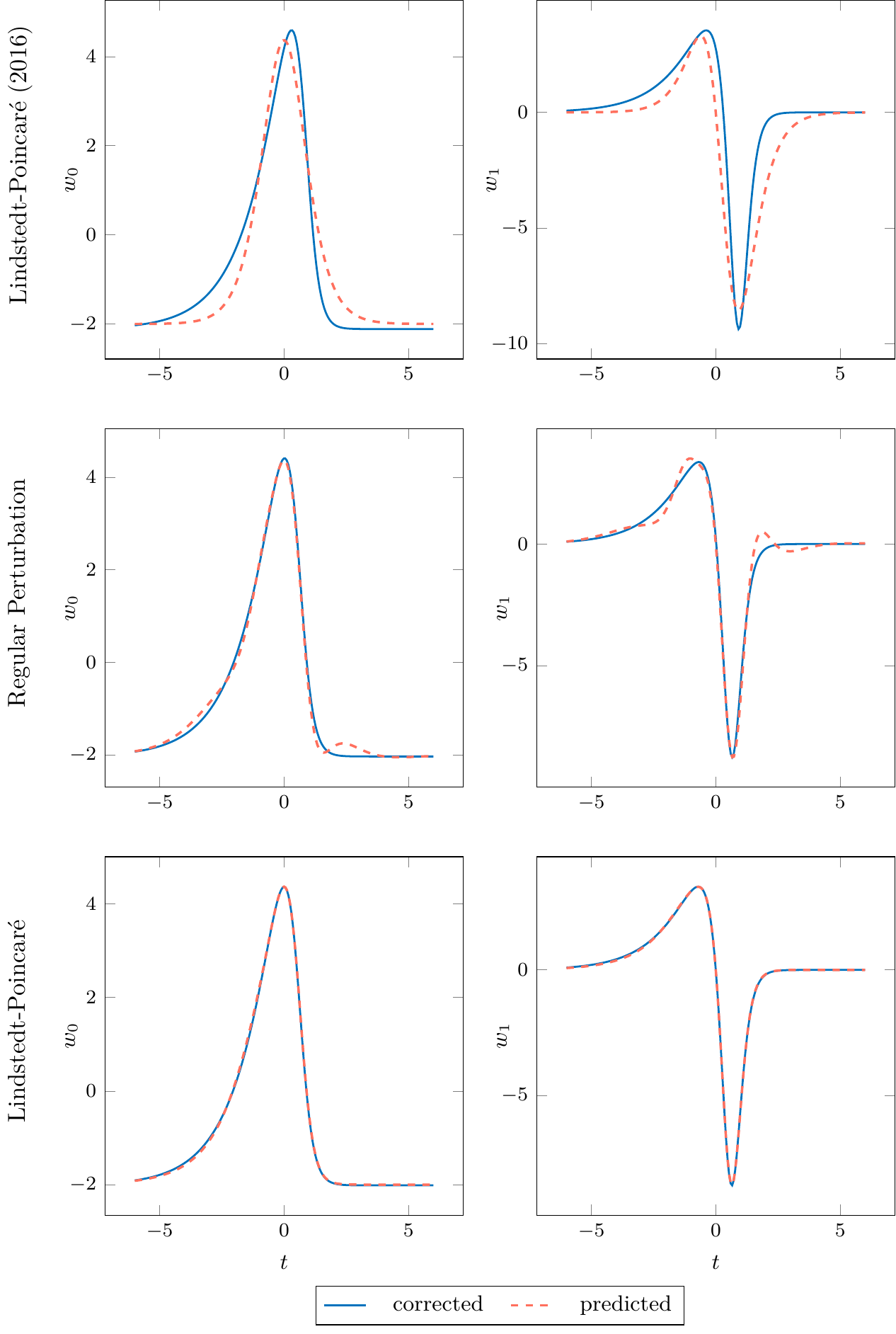}
    \fi
    \caption{Comparison of the profiles of the predicted and corrected
        homoclinic orbit for \cref{eq:universal_unfolding} using different
        approximation methods, see
    \cref{sec:topological_normal_form} for a full description.}
    \label{fig:RP_vs_LP2016_vs_LP_profiles}
\end{figure}

Lastly, in \cref{fig:LPM_vs_LPM}, we compare the two different phase conditions
when using the Lindstedt-Poincar\'e method in a log-log plot. It is clearly
seen that the phase condition used in \cref{sec:phase_condition} improves,
rather significantly, the accuracy of the third-order predictor. However, in
contrast with the different phase conditions used in the regular perturbation
method, we do not have any geometric (or analytical) explanation for this
improvement.

\begin{figure}
    \centering
    \ifcompileimages%
  \tikzsetnextfilename{BTParameterdependentnormalformConvergencePlotLPMvsLPM}%
  \input{tikz/BTParameterdependentnormalformConvergencePlotLPMvsLPM}%

    \else
        \includegraphics{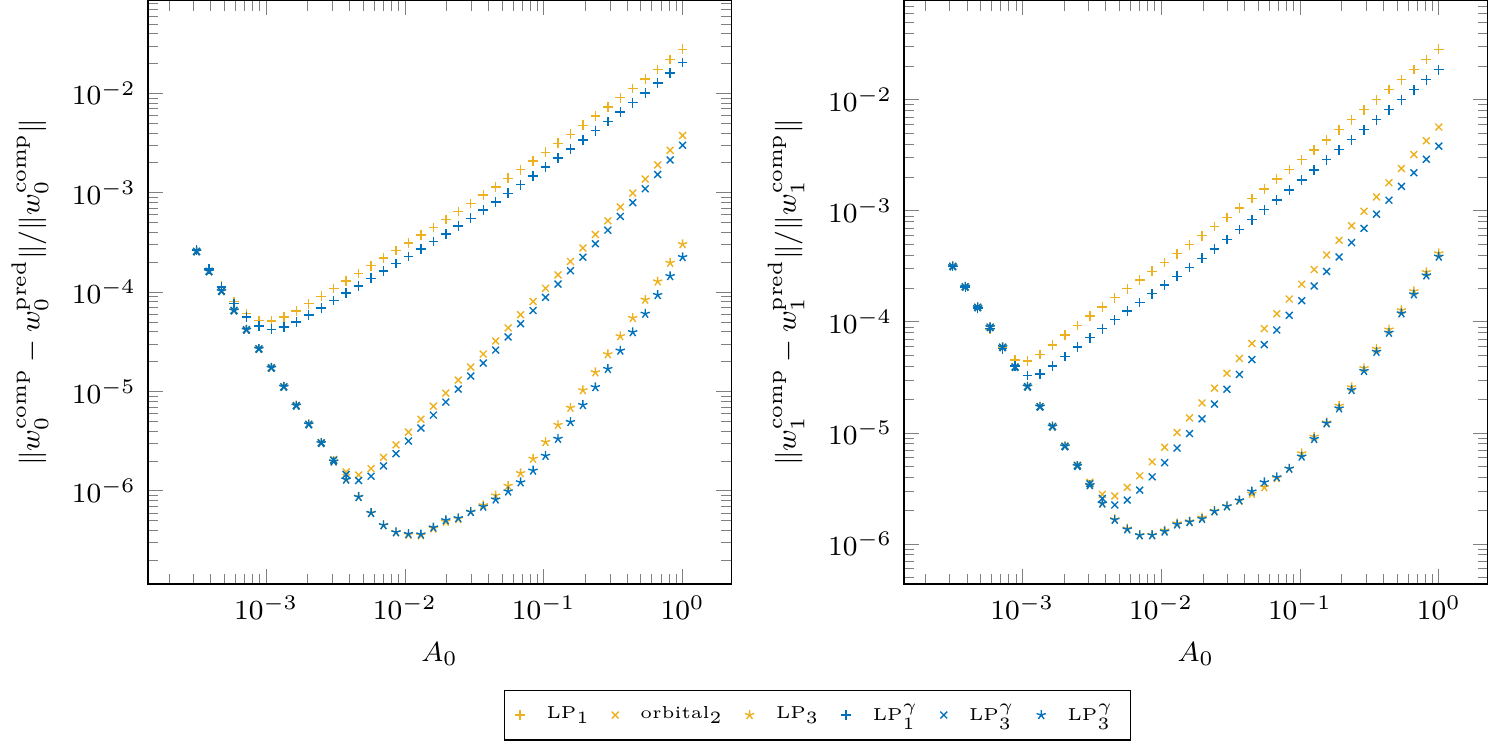}
    \fi
    \label{fig:LPM_vs_LPM}
    \caption{Log-log convergence plot comparing the relative errors of the computed
        homoclinic $w_0$ and $w_1$ component with the predicted solution in the
        topological normal form using Lindstedt-Poincar\'e with two different
        phase-condition, see \cref{sec:phase_condition}.}
\end{figure}
Notice that here we do not compare the homoclinic predictors derived with
different normal forms. Indeed, when considering the universal unfolding
\cref{eq:universal_unfolding} the normal forms coincide, resulting in identical
predictors.

\subsection{Hodgking-Huxley equations}

The Hodgkin-Huxley equations~\cite{HodgkinHuxley1952} relate the difference in electric
potential across the cell membrane $V$ and gating variables $m, n$ and $h$
for ion channels to the stimulus intensity $I$ and temperature $T$, as
follows:
\begin{equation}
\label{eq:HodgkinHuxleyEquations}
\begin{cases}
\begin{aligned}
    \dot{V} &= -G(V, m, n, h)+I, \\
    \dot{m} &= \Phi(T)\left[(1-m) \alpha_{m}(V)-m \beta_{m}(V)\right], \\
    \dot{n} &= \Phi(T)\left[(1-n) \alpha_{n}(V)-n \beta_{n}(V)\right], \\
    \dot{h} &= \Phi(T)\left[(1-h) \alpha_{h}(V)-h \beta_{h}(V)\right],
\end{aligned}
\end{cases}
\end{equation}
where 
\begin{align*}
    \Phi(T) & = 3^{({T}-6.3) / 10}, \\
    G(V, m, n, h) & =\bar{g}_{\mathrm{Na}} m^{3}
    h\left(V-\bar{V}_{\mathrm{Na}}\right)+\bar{g}_{\mathrm{K}}
    n^{4}\left(V-\bar{V}_{\mathrm{K}}\right)+\bar{g}_{\mathrm{L}}\left(V-\bar{V}_{\mathrm{L}}\right).
\end{align*}
The equations modeling the variation of membrane permeability are:
\begin{align*}
    \alpha_{m}(V) =& \Psi\left(\frac{V+25}{10}\right), & \beta_{m}(V) &= 4 e^{V / 18}, \\
    \alpha_{n}(V) =& 0.1 \Psi\left(\frac{V+10}{10}\right), & \beta_{n}(V) &= 0.125 e^{V / 80}, \\
    \alpha_{h}(V) =& 0.07 e^{V / 20}, & \beta_{h}(V) &= \left(1+e^{(V+30) / 10}\right)^{-1},
\end{align*} with
\begin{equation*}
    \Psi(x) = \begin{cases}
        x /\left(e^{x}-1\right), & \text { if } x \neq 0, \\
        1, & \text { if } x=0.
    \end{cases}
\end{equation*}
The parameters $\bar{g}_{\text{ion}}$ and $\bar{V}_{\text{ion}}$ representing
maximum conductance and equilibrium potential for the ion were obtained from
experimental data by Hodgkin and Huxley, with the values given below:
\[
\begin{array}{lll}
\bar{g}_{\mathrm{Na}}=120 \mathrm{mS} / \mathrm{cm}^{2}, 
& \bar{g}_{\mathrm{K}}=36 \mathrm{mS} / \mathrm{cm}^{2}, 
& \bar{g}_{\mathrm{L}}=0.3 \mathrm{mS} / \mathrm{cm}^{2}, \\
\bar{V}_{\mathrm{Na}}=-115 \mathrm{mV},
& \bar{V}_{\mathrm{K}}=12 \mathrm{mV}, 
& \bar{V}_{\mathrm{L}}=10.599 \mathrm{mV}.
\end{array}
\]
The values of $\bar{V}_{\mathrm{Na}}$ and $\bar{V}_{\mathrm{K}}$ can be
controlled experimentally~\cite{HodgkinHuxley1952a,Jack1975ElectricCurrentFlow}.
The temperature is set to $T=6.3^{\circ}$.

It is easy to see that the equilibria of \cref{eq:HodgkinHuxleyEquations} can be
parametrized by $V$
\begin{equation*}
    \begin{aligned}
        I(V) &= G(V, m(V), n(V), h(V)) \\
        y(V) &= \alpha_y(V)/(\alpha_y(V)+\beta_y(V)),
    \end{aligned}
\end{equation*}
where $y\in\{m,n,h\}$, see also~\cite{Guckenheimer@1993}. By calculating the Jacobian $A$ of
\cref{eq:HodgkinHuxleyEquations} at the equilibrium, we can derive the
characteristic polynomial $\rho_A(\lambda)$. The equation $\rho_A(0)=0$ can be
solved analytically for $\bar V_K$. Using this solution for $\bar V_k$ and
plotting the curve $\rho'(0)$ reveals two potential candidates for Bogdanov-Takens
points. Inspecting the geometric multiplicity of these two points narrows the
possibilities down to the point
\begin{equation}
\label{eq:HodgkinHuxleyBTpoint}
\begin{pmatrix}
    V \\m \\n \\h \\ \bar V_k \\ I 
\end{pmatrix}
\approx
\begin{pmatrix}
-2.835463618170097 \\ 0.07351498630356315 \\ 0.361877602925177 \\ 0.494859128785482 \\
-4.977020454108788 \\ -0.06185214966177632
\end{pmatrix}.
\end{equation}  
Inspecting the coefficients of the normal form shows that
\[
a = 2.5515\cdot 10^{-5}, \qquad b =  -0.0075.
\] 
Thus, provided the transversality conditions are satisfied, we can use MatCont to
start continuation of the homoclinic orbits emanating from this point.

\begin{figure}
    \centering
    \ifcompileimages
        \subfloat[]{%
  \tikzsetnextfilename{HodgkinHuxleyConvergencePlotLP}%
  \input{tikz/HodgkinHuxleyConvergencePlotLP}%
}
        \hfill
        \subfloat[]{%
  \tikzsetnextfilename{HodgkinHuxleyConvergencePlot}%
  \input{tikz/HodgkinHuxleyConvergencePlot}%
}
    \else
        \subfloat[]{\includegraphics{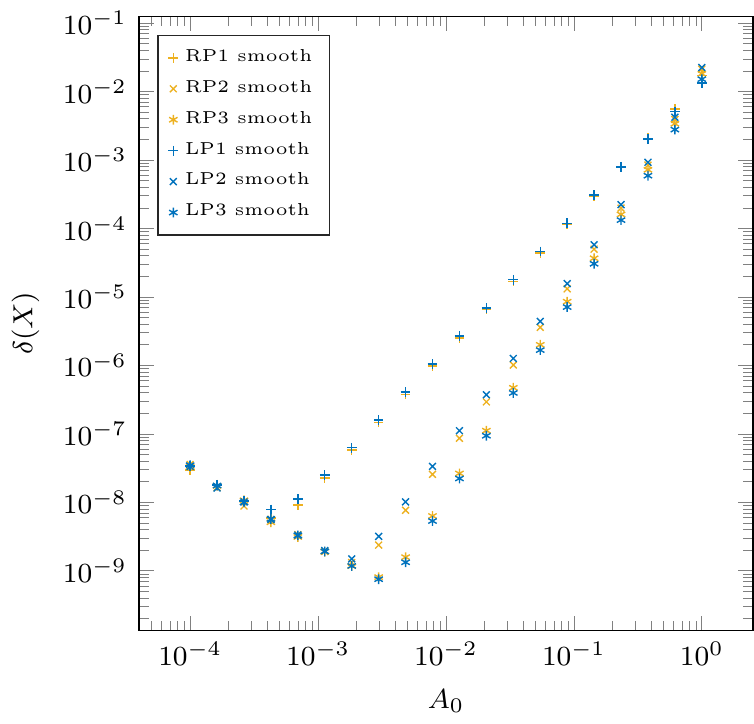}}
        \hfill
        \subfloat[]{\includegraphics{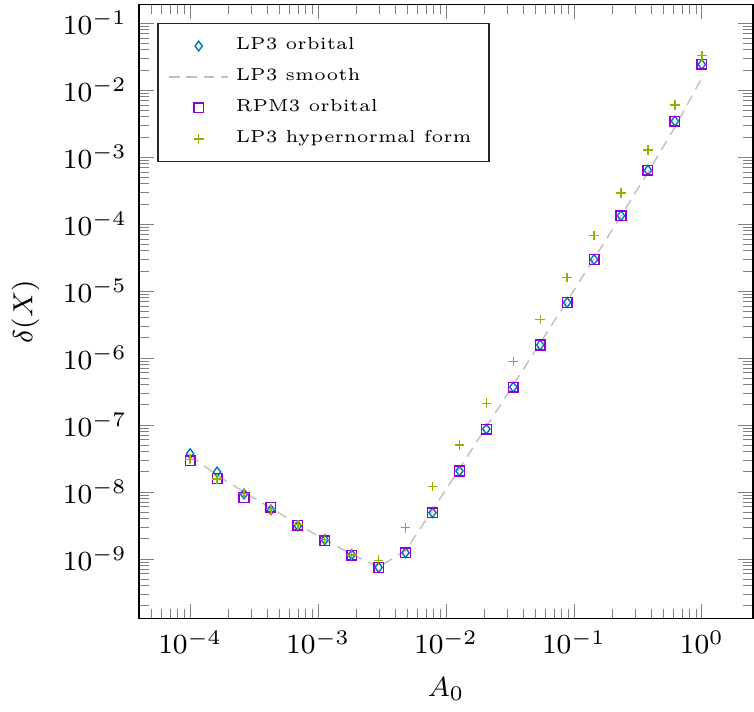}}
    \fi
    \caption{Convergence plot for the homoclinic predictors near the
    Bogdanov-Takens bifurcation \cref{eq:HodgkinHuxleyBTpoint} in the
    Hodgkin-Huxley equations \cref{eq:HodgkinHuxleyEquations}.}
    \label{fig:HodgkinHuxleyConvergencePlot}
\end{figure}

In \cref{fig:HodgkinHuxleyConvergencePlot} there are two log-log convergence
plots shown. Note that in this and the next example, we show the relative error
$\delta(X)$ between the predicted and corrected Newton solution to the defining
system \cref{eq:definingSystem}.  In the left plot (a) we compare the regular
perturbation method with the Lindstedt-Poincar\'e method. We see that compared
with the previous example the Lindstedt-Poincare\'e method is slightly less
accurate than the regular perturbation method and the second order.
Nevertheless, we clearly see that the order of convergence lifts from the
normal form to the two-dimensional center manifold in $\mathbb R^4$. In the
plot right (b) we compare four different third approximations to the homoclinic
orbit
\begin{itemize}
    \item the Lindstedt-Poincar\'e method using the smooth orbital normal form
        (the blue diamond),
    \item the Lindstedt-Poincar\'e method using the smooth normal form
        (the dashed light gray line),
    \item the regular perturbation method using the smooth normal form
        (the pink square), 
    \item the Lindstedt-Poincar\'e method using the hyper-normal form
        (the green plus).
\end{itemize}

We see that both the Lindstedt-Poincar\'e method and the regular perturbation
method using the smooth orbital normal form are in perfect agreement with the
Lindstedt-Poincar\'e method using the smooth normal form. Only the homoclinic
predictor using the hyper-normal form is slightly less accurate.

\subsection{Homoclinic RG flows}
In~\cite{Jepsen@2021}  an $\mathcal N = 1$ supersymmetric model of interacting
scalar superfields $\Phi_{ab}^i$ that is invariant under the action of an $O(N)
\times O(M)$ group in $d = 3 - \epsilon$ dimensions is considered.
The coupling constants $g_i(i=1,\dots,4)$ satisfy the following differential
equations
\begin{equation}
    \label{eq:HomoclinicRGflows} 
    \dot g = -\epsilon g + \beta^{(2)}(g,M,N) + \mathcal O(g^5),  \qquad g\in\mathbb R^4,
\end{equation} 
where the two-loop contributions $\beta_i^{(2)}(i=1,\dots,4)$ are cubic in the
coupling and the parameter $\epsilon$ is scaled to $1$.  The exact expression for
$\beta_i^{(2)}$ are quite long can be found in~\cite[Appendix B]{Jepsen@2021}
or in the Supplementary Materials. 

In~\cite{Jepsen@2021} a Bogdanov-Takens point near the parameter values
$M=0.2945$ and $N = 4.036$ is located. Using these parameter values we locate an
equilibrium at 
\[
\begin{pmatrix}
    g_1 \\ g_2 \\ g_3 \\ g_4
\end{pmatrix} = 
\begin{pmatrix}
    0.0701457361241472 \\ -0.06520883770451065 \\ 0.001823543197553845 \\ 0.22874527306411319
\end{pmatrix}.
\]
By continuing the equilibrium in the parameter $M$ we detect several limit
points and two Hopf points. We continue the second Hopf point at $M\approx0.2958$ 
in parameters $M$ and $N$. Several Bogdanov-Takens points are detected. The
first Bogdanov-Takens point is located at
\[
\begin{pmatrix}
    g_1 \\ g_2 \\ g_3 \\ g_4 \\ 
\end{pmatrix} = 
\begin{pmatrix}
    -0.715157316845187 \\ -0.250968103603174 \\ 0.510051114588271 \\ -0.391935453715783 \\
\end{pmatrix},
\]
with parameter values
\[
    (M, N) = (0.294477255737036, 4.035536108506390).
\]

In \cref{fig:Lorenz84ConvergencePlot} we have created similar log-log
convergence plots as in the previous example. The plots look very alike, only
the homoclinic predictor using the hyper-normal form is in this model slightly
more accurate than the other homoclinic approximations.

Lastly, in \cref{fig:Lorenz84PhasePhase}, there are two additional plots. The
left plot (a) compares the predicted (dashed, red) with the corrected (solid,
blue) homoclinic orbits using the Lindstedt-Poincar\'e method with the smooth
orbital normal form for amplitudes $A_0 = 10^{-3}$ to $A_0=0.1$. We see that
they are in excellent agreement. In the right plot (b) we compared the computed
homoclinic bifurcation curve (solid, blue) with the predicted values in
parameter-space $(M, N)$. Most important to notice here is that the
predictor given in \cite{Al-Hdaibat2016} (the yellow crosses)
is less accurate than the second order predictor (blue plus signs) obtained in
this paper.

\begin{figure}[ht]
    \centering
    \ifcompileimages%
        \subfloat[]{%
  \tikzsetnextfilename{HomoclinicRGFlowsConvergencePlotRPvsLP}%
  \input{tikz/HomoclinicRGFlowsConvergencePlotRPvsLP}%
} \hfill
        \subfloat[]{%
  \tikzsetnextfilename{HomoclinicRGFlowsConvergencePlotAll}%
  \input{tikz/HomoclinicRGFlowsConvergencePlotAll}%
}
    \else
        \includegraphics{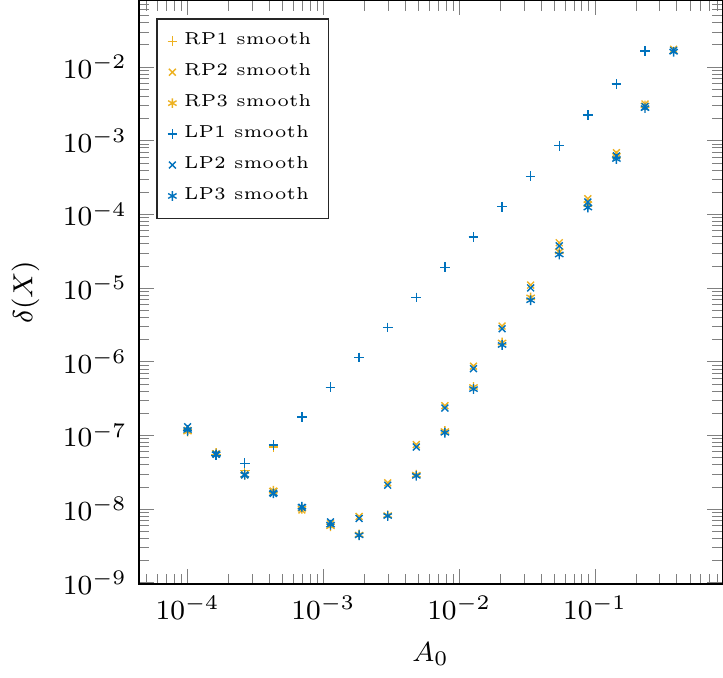} \hfill
        \includegraphics{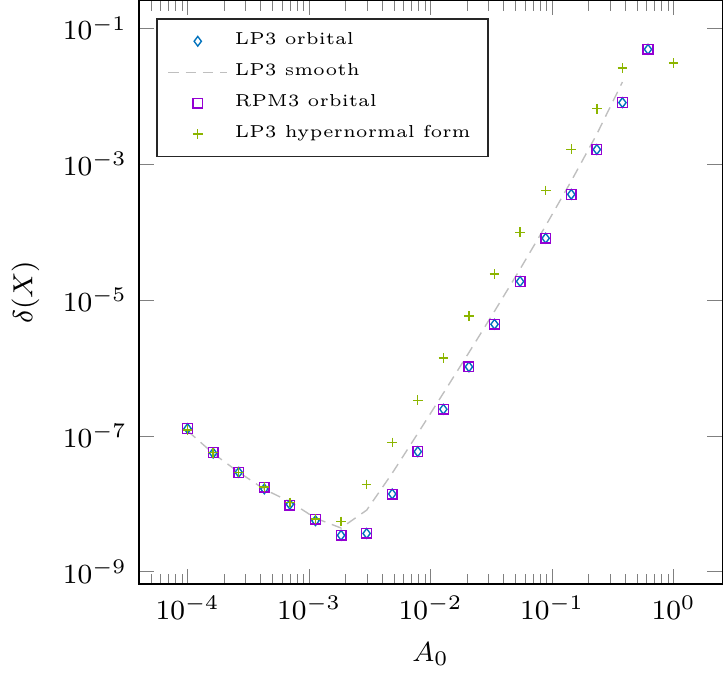}
    \fi
    \caption{Convergence plot for the homoclinic predictors near one of the
        two Bogdanov-Takens bifurcation in the Homoclinic RG flows model 
        \cref{eq:HomoclinicRGflows}.}
    \label{fig:Lorenz84ConvergencePlot}
\end{figure}
\begin{figure}[ht]
    \centering
    \ifcompileimages%
        \subfloat[]{%
  \tikzsetnextfilename{HomoclinicRGflowsCompareOrbits3D_LP}%
  \input{tikz/HomoclinicRGflowsCompareOrbits3D_LP}%
} \hfill
        \subfloat[]{%
  \tikzsetnextfilename{HomoclinicRGflowsCompareParameters}%
  \input{tikz/HomoclinicRGflowsCompareParameters}%
}
    \else
        \includegraphics{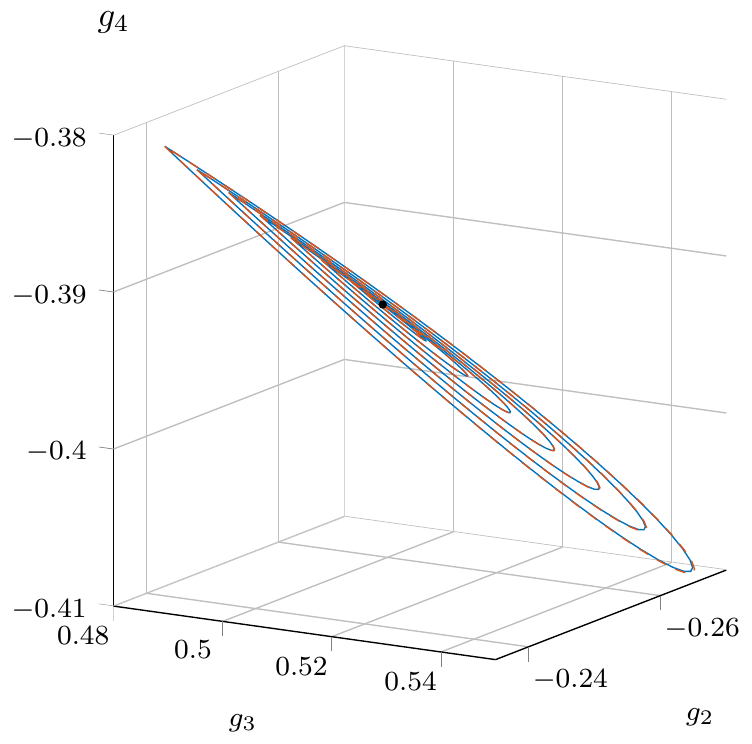} \hfill
        \includegraphics{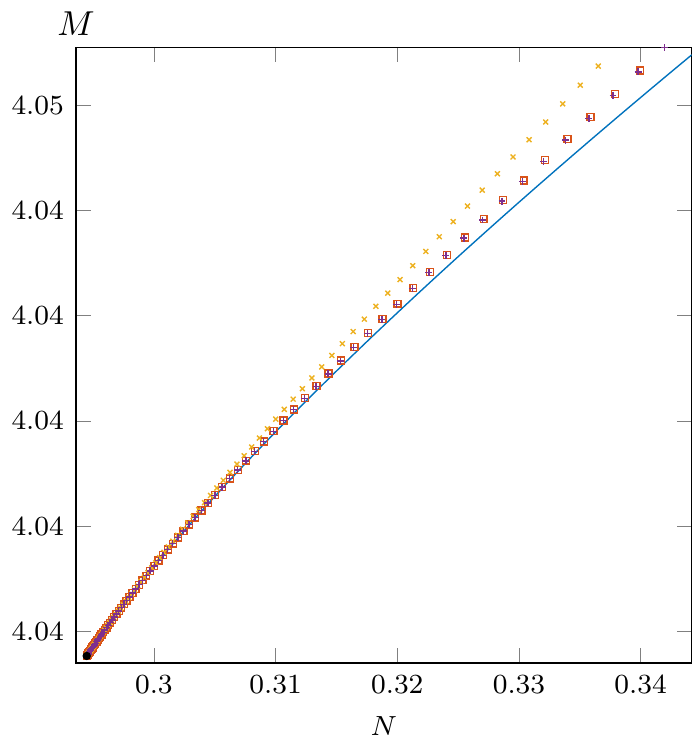}
    \fi
    \caption{Plot (a) compares the predicted (dashed, red) with the corrected
        (solid, blue) homoclinic orbits using the Lindstedt-Poincar\'e method
        with the smooth orbital normal form for amplitudes $A_0 = 10^{-3}$ to
        $A_0=1$. Plot (b) compares the computed homoclinic bifurcation curve
        (solid, blue) with the predicted values in parameter-space $(M,N)$.
        The yellow crosses are the second order predictor obtained with the
        transformation as given in \emph{\cite{Al-Hdaibat2016}}. The blue plus signs
        are the second order predictor obtained in this paper. Lastly, the red
        boxes are the third-order predictor obtained in this paper.}
    \label{fig:Lorenz84PhasePhase}
\end{figure}

\section{Discussion}%
We have derived third-order predictors for the homoclinic curve emanating from
the generic codimension two Bogdanov-Takens bifurcation in general
$n$-dimensional autonomous ODEs. By considering the smooth orbital normal form
\cref{eq:normal_form_orbital} and incorporating the time-reparametrization in the
homological equation \cref{eq:homological_equation} we were able to derive
the third-order asymptotic of the homoclinic curve independent of any
coefficients. However, for this simplification, there is a price to pay.
Firstly, the systems to be solved to obtain the coefficients for the
parameter-dependent center manifold becomes more difficult, see
\cref{subsec:center_manifold_tranformation_orbital}. Ideally, there should be an
automatic algorithm in line with~\cite{Murdock@2003}. However, to the best of
our knowledge, such algorithms do not exist yet.  Secondly, the translation of
time in the homological equation needs to be inverted numerically. This,
however, can be done relatively cheap and is very accurate as shown by the
examples.

We have explained how to obtain the correct transformation to the
parameter-dependent center manifold by carefully inspecting which terms are
in, \emph{and are not in}, the normal form that affects the homoclinic
asymptotic up to certain order. The comparison in
\cref{sec:comparison_homoclinic_predictors} and the examples in
\cref{sec:examples} together with the examples in the supplementary materials
show that we indeed have obtained the correct transformation.

The additional non-linear transformation
\cref{eq:second_order_nonlinear_oscillator} greatly simplifies the computation
of the coefficients in the Lindstedt-Poincar\'e method since all calculations
become essentially polynomial, which is ideal for computers to work with.
Nonetheless, there the algorithm complexity grows exponentially as the order
increases linearly. Also, the radius of convergence is clearly finite as shown
in \cref{sec:case_study_BT2}. One way to increase the convergence radius is by
using transformations as in~\cite{Milton@1974}. However, we didn't include any
results in this direction since it would distract too much from our main
objectives.

Using different phase conditions can improve the accuracy of the homoclinic
approximation. However, this only holds true when applied directly to the system
considered. Indeed, the phase condition isn't invariant under the
parameter-dependent center manifold transformation. Thus, its applicability is
very limited. Furthermore, using a different phase condition may, somewhat
unexpectedly, result in difficult integrals to be solved, see
\cref{sec:RPM_norm_minimizing_phase condition}.

The higher-order approximation to the non-linear time transformation in the
Lindstedt-Poincar\'e method turns out to be essential to obtain higher-order
approximations to the homoclinic solutions. This is clearly seen by inspecting
the profiles of the homoclinic solution in
\cref{fig:RP_vs_LP2016_vs_LP_profiles} and in the convergence plot in
\cref{fig:RP_vs_LP2016_vs_LP}. Without the higher-order approximation the same
convergence order as the unperturbed Hamiltonian solution, i.e., the
zeroth-order solution. It should be noted that the higher-order approximation
of the non-linear time transformation is more difficult to obtain. Therefore,
we conclude that there seem to be \emph{no benefits} of the
Lindstedt-Poincar\'e method over the regular perturbation method for starting
continuation of homoclinic orbits. Indeed, the numerical comparisons in
\cref{sec:examples} show similar accuracy of convergence at each order.

By comparing the convergence order of the regular perturbation method with the
Lindstedt-Poincar\'e method we see that contrary to what one might expect, the
regular perturbation method may result in better accuracy at the same order. A
possible explanation for this might be that although the Lindstedt-Poincar\'e
method provides a uniform approximation in time along the homoclinic orbit of
the numerical solution is truncated to a finite interval in which the
`parasitic turn' doesn't give a significant contribution. After all it then
simply depends on the higher-order non-linear terms in the system which favor
one method over the other. 

\appendix
\section{Explicit example demonstrating incorrect predictor}
\label{app:incorrect_predictor}

Although the second-order homoclinic approximation derived
in~\cite{Al-Hdaibat2016} for the smooth normal form \cref{eq:BT_smooth_nf} is
correct, the parameter and center manifold transformation are incorrect. To see
this, we suppose \cref{eq:ODE} is given by
\begin{align}
\label{eq:btnormalform_with_alpha2^3}
\dot x = f(x,\alpha) = \begin{pmatrix}   x_1 \\
 \alpha_1 + \alpha_2 x_1 + x_0^2 + x_0 x_1 + c_1 \alpha_2^3
 \end{pmatrix},
\end{align}
for some arbitrary nonzero constant $c_1 \in \mathbb R$.  We will now compare
two different methods for obtaining a second-order approximation to the
homoclinic solution in \cref{eq:btnormalform_with_alpha2^3}. To keep the
exposition as clear as possible, we focus solely on the parameters. For the
first method we directly apply the singular rescaling 
\[
\alpha_1=-4\epsilon^4, \quad \alpha_2 = \eta \epsilon^2,
\quad x_0= \epsilon^2, \quad x_1= \epsilon^3, \quad
s=\epsilon t,
\]
to \cref{eq:btnormalform_with_alpha2^3}. This yields the system
\[
\begin{cases}
\begin{aligned}
  \dot u ={}& v, \\
  \dot v ={}& -4 + u^2 + v\left( u + \tau \right)\epsilon +  c_1 \tau^3
	\epsilon^2, \\
\end{aligned}
\end{cases}
\]
where the dot $\dot{}$ now represents the derivative with respect to $s$.
Then, using the generalized Lindstedt-Poincar\'e method we obtain the approximation
\begin{equation}
\label{eq:first_predictor}
				(\alpha_1, \alpha_2) = \left(-4 \epsilon^4, \frac{10}7 \epsilon^2 
+\frac{288-1250 c_1}{2401}\epsilon^4 + \mathcal O(\epsilon^5) \right)
\end{equation}
for the parameters. For the second method we use the predictor
from~\cite{Al-Hdaibat2016}. That is, we use the second-order homoclinic
predictor derived for the smooth normal form \cref{eq:BT_smooth_nf}. Then
calculate the center manifold transformation, which for the two-dimensional
systems reduces to a near-identity transformation, and parameter transformation
to transfer the predictor the original system. We obtain that the near-identity
and parameter transformation are just the identities. Thus, we obtain the
predictor
\begin{equation}
				\label{eq:second_predictor}	
				(\alpha_1, \alpha_2) = \left(-4 \epsilon^4, \frac{10}7 \epsilon^2 
				+\frac{288}{2401}\epsilon^4 + \mathcal O(\epsilon^5) \right).
\end{equation}
Obviously, this result is wrong.  To see why the latter second predictor
doesn't contain the term $c_1$ we consider the near-identity transformation
\begin{align}
\label{eq:near_identity_tranformation}
  \begin{cases}
  \begin{aligned}
    x &{}= w, \\ 
    \alpha &{}= \beta+ \begin{pmatrix} -c_1 \\ 0 \end{pmatrix}\beta_2^3.
  \end{aligned}
  \end{cases}
\end{align}
Then system \cref{eq:btnormalform_with_alpha2^3} becomes
\begin{align*}
  \begin{cases}
  \begin{aligned}
     \dot w_0 &{}= w_1, \\
		 \dot w_1 &{}= \beta_1 + \beta_2 w_1 + w_0^2 + w_0 w_1.
  \end{aligned}
  \end{cases}
\end{align*}
Using the second-order predictor from~\cite{Al-Hdaibat2016} for the smooth
normal form we obtain
\[
    (\beta_1, \beta_2) = \left(-4 \epsilon^4, \frac{10}7 \epsilon^2 
    +\frac{288}{2401}\epsilon^4 + \mathcal O(\epsilon^5) \right).
\]
Then using the near-identity transformation
\cref{eq:near_identity_tranformation} yields the predictor
\begin{equation*}
				(\alpha_1, \alpha_2) = \left(-4 \epsilon^4 
								- c_1 \left(\frac{10}7 \epsilon^2 
								+\frac{288}{2401}\epsilon^4 + \mathcal O(\epsilon^5) \right)^3, 
								\frac{10}7 \epsilon^2 
				+\frac{288}{2401}\epsilon^4 + \mathcal O(\epsilon^5) \right).
\end{equation*}
To compare this predictor with \cref{eq:first_predictor} we eliminate $\epsilon$
in both equations. This yields
\begin{equation*}
				\alpha_2(\alpha_1) = - \frac57 \sqrt{-\alpha_1} +
				\frac{288-1250c_1}{2401} \alpha_1 + \mathcal O(\alpha_1^{\frac32}).
\end{equation*}

We conclude, as expected, that if the correct transformation is used between the
normal form and the original system we keep the correct order of accuracy for
the approximation. In~\cite{Al-Hdaibat2016} the coefficients $H_{0003}$ and
$K_{03}$ (among other coefficients) are not incorporated into the
parameter-dependent center manifold transformation
\cref{eq:H_expansion,eq:K_expansion} leading to the incorrect
predictor \cref{eq:second_predictor}.

\section{Integrals from
\texorpdfstring{\cref{sec:third_order_homoclinic_approximation_RP}}{subsection 3.1.1}}%
%TODO : check is the hard coded reference is up-to-date.
\label{sec:I_n}
Making the substitution $s=\log(u)$ in \cref{eq:I_n} yields
\begin{equation*}
    I_n = 2^n \int_1^\infty \log^3\left(\frac{u^2+1}{u}\right) \frac{u^{n-1}}{(u^2+1)^n} \, du.
\end{equation*}
Then, by making the reciprocal substitution $u \to \frac{1}{u}$, we can show
that
\begin{equation*}
    I_n = 2^{n-1} \int_0^\infty \log^3\left(\frac{u^2+1}{u}\right)
            \frac{u^{n-1}}{(u^2+1)^n} \, du.
\end{equation*}
The last integral can be separated into the four following integrals
\begin{align}
    I_n^{(1)} =& \int_0^\infty \log^3(u^2+1) \frac{u^{n-1}}{(u^2+1)^n} du
              \label{eq:In1}, \\
    I_n^{(2)} =& \int_0^\infty \log^2(u^2+1)\log u \frac{u^{n-1}}{(u^2+1)^n} du 
              \label{eq:In2}, \\
    I_n^{(3)} =& \int_0^\infty \log(u^2+1)\log^2 u \frac{u^{n-1}}{(u^2+1)^n} du 
              \label{eq:In3}, \\
    I_n^{(4)} =& \int_0^\infty \log^3 u \frac{u^{n-1}}{(u^2+1)^n} du
              \label{eq:In4}.
\end{align}
The integral \cref{eq:In1} can easily be solved by first applying the
substitutions $u\to u^2+1$ and $u\to\frac1u$ consecutively to obtain
\begin{align*}
    I_n^{(1)} 
    &{}= - \frac12 \int_0^1 \log^3 (u) 
        \left(1-u\right)^{\frac{n}{2}-1}u^{\frac{n}{2}-1} \, du.
\end{align*}
Then using the binomial theorem yields
\begin{equation*}
    I_n^{(1)} 
= - \frac12 \sum_{k=0}^{\frac{n}{2}-1} \binom{\frac{n}{2}-1}k (-1)^k 
            \int_0^1 \log^3 (u) u^{\frac{n}{2}-1+k} \, du
    = \frac12 \sum_{k=0}^{\frac{n}{2}-1} \binom{\frac{n}{2}-1}k
    \frac{(-1)^k 3!}{({\frac{n}{2}+k})^4},
\end{equation*}
where in the last equality we used the well-know identity
\begin{equation}
    \label{eq:int01umlogn}
    \int_0^1 u^m \log^n (u) \, du = (-1)^n \frac{n!}{(m+1)^{n+1}},
\end{equation}
for $n$ and $m$ natural numbers.

To solve the integral \cref{eq:In2} we make three consecutive
substitutions: $u\to u^2$, $u\to u-1$, and $u\to\frac{1}{u}$. This gives
\[
    I_n^{(2)}=
    -\frac{1}{4} \int_0^1 
        \left(\log^3(v)-\log^2(v)\log(1-v) \right) 
        (1-v)^{\frac{n}2-2} v^{\frac{n}2-1} dv. \\
\]
Then, by using the binomial theorem and expanding the logarithm we obtain
\ifsiam
\begin{multline*}
    I_n^{(2)}={}
\frac{1}{4} \sum_{k=0}^{\frac{n}{2}-1} \binom{\frac{n}{2}-1}k
        (-1)^k \left[ 
          \int_0^1 \log^3(v) (1-v)^{\frac{n}2-2} v^{\frac{n}2-1} dv \right. \\
        \left. 
        -\sum_{l=1}^\infty \frac1l \int_0^1 \log^2(v) v^{\frac{n}2-1+k+l} dv 
    \right].
\end{multline*}
\else
\begin{multline*}
    I_n^{(2)}={}
\frac{1}{4} \sum_{k=0}^{\frac{n}{2}-1} \binom{\frac{n}{2}-1}k
        (-1)^k \left[ 
          \int_0^1 \log^3(v) (1-v)^{\frac{n}2-2} v^{\frac{n}2-1} dv
        -\sum_{l=1}^\infty \frac1l \int_0^1 \log^2(v) v^{\frac{n}2-1+k+l} dv 
    \right].
\end{multline*}
\fi
Using equality \cref{eq:int01umlogn} once more yields
\begin{equation*}
    I_n^{(2)}={}
\frac{1}{4} \sum_{k=1}^{\frac{n}{2}-1} \binom{\frac{n}{2}-1}k
    (-1)^k \left[ \frac{3!}{(\frac{n}{2}+k)^4}
     -2\sum_{l=1}^\infty \frac1{l(\frac{n}2+k+l)^3}\right].
\end{equation*}
Fractional decomposition shows that the innermost summation in the last
equation is equal to 
\begin{equation*}
    \sum_{l=1}^\infty \frac1{l(\frac{n}2+k+l)^3}
        = \frac{8}{2k+n} \left(\frac{H_{\frac{n}{2} + k}}{(2k+n)^2} +
            \frac{H^{(2)}_{k+\frac{n}{2}} - \zeta(2)}{2(2k+n)} +
            \frac{H^{(3)}_{k+\frac{n}{2}} - \zeta(3)}4 
        \right).
\end{equation*}
Thus, the integral in \cref{eq:In2} is equal to
\begin{multline*}
    I_n^{(2)} = \frac{1}{4} \sum_{k=1}^{\frac{n}{2}-1} 
        \binom{\frac{n}{2}-1}k (-1)^k 
        \left[ \frac{3!}{(\frac{n}{2}+k)^4} \right.  
        - \frac{2}{2k+n} 
            \left(\frac{4}{(2k+n)^2} H_{\frac{n}{2}+k} \right. \\ 
            \left.\left. + \frac{2}{(2k+n)} H^{(2)}_{\frac{n}{2}+k} 
            + H^{(3)}_{\frac{n}{2}+k} - \frac{2\zeta(2)}{2k+n} - 4\zeta(3)
    \right)\right].
\end{multline*}
Now most work is done, since subtracting two times \cref{eq:In2} from
\cref{eq:In3} is equal to
\begin{equation*}
    \int_0^\infty \log^2(u^2+1)\log (u) \log\left(1+\frac{1}{u^2}\right)
        \frac{u^{n-1}}{(u^2+1)^n} du.
\end{equation*}
The reportorial substitution $u\to\frac{1}{u}$ shows that this integral
vanishes. The same substitution also shows that the integral in  \cref{eq:In4}
vanishes.  We thus obtain the closed-form expression 
\begin{multline*}
    I_n = 2^{n-3} 3 \sum_{k=0}^{\frac{n}{2}-1} \binom{\frac{n}{2}-1}k  
            (-1)^k  \\
        \left[\frac1{(\frac{n}{2}+k)^4} +
            \frac{8}{2k+n} \left(\frac{H_{\frac{n}{2} + k}}{(2k+n)^2}
                + \frac{H^{(2)}_{\frac{n}{2}+k} - \zeta(2)}{2(2k+n)}
            + \frac{H^{(3)}_{\frac{n}{2}+k} - \zeta(3)}{4}
        \right)\right],
\end{multline*}
where $\zeta$ is the Riemann zeta function and $H_n^{(m)}$ is the $n$-th
generalized harmonic number of order $m$.

\section{Asymptotics for homoclinic solution in the smooth normal form}
\label{sec:asymptotics-for-homoclinic-solution-for-the-smooth-normal-form}

Following the procedure outlined in \cref{sec:PolynomailLindstedtPoincare} to
the second-order nonlinear differential equation
\cref{eq:second_order_nonlinear_oscillator_smooth_normalform} obtained from the
smooth normal form \cref{eq:BT_smooth_nf}. For the third-order homoclinic
predictor we obtain
\begin{align}
				\sigma ={}& 6 + \frac{3 \left(-70 a_1 b+6 b^2+49 d\right)}{49 a^2} \epsilon^2
                    + \mathcal{O}(\epsilon^4),
				\nonumber \\
				\delta ={}& -4 + \frac{140 a_1 b-18 b^2-245 d}{49 a^2} \epsilon^2
                    + \mathcal{O}(\epsilon^4),
				\nonumber \\
				\label{eq:tau_smooth}
				\tau   ={}& \frac{10}{7} + \frac{98 b (50 a b_1+73 d)-9604 a e-2450 a_1 b^2+288 b^3}{2401 a^2 b} \epsilon^2
								+ \mathcal{O}(\epsilon^4), \\
				\tilde \omega(\zeta) ={}& 1 - \frac{6b}{7a}\zeta  \epsilon	+ 
				\frac{70 a_1 b+18 b^2 \left(3 \zeta ^2+1\right)+49 d \left(9 \zeta ^2-5\right)}{196 a^2} \epsilon^2 +
				\nonumber \\
        & \frac{\zeta}{2401 a^3} \left( \left(-147 b \left(20 a b_1-7 d
        \zeta ^2+11 d\right)-9604 a e \left(\zeta ^2-1\right) \right. \right. \nonumber \\
        & \left. \left. +1470 a_1 b^2+18 b^3 \left(7 \zeta ^2-11\right)\right) \right)\epsilon^3 + \mathcal{O}(\epsilon^4)
        \nonumber.
\end{align}
It follows that
\begin{align}
    \label{eq:third_order_uhat_smooth}
    \tilde {u}(\zeta) 
    ={}& 6 \zeta ^2-4 + \frac{-70 a_1 b \left(3 \zeta ^2-2\right)+18 b^2 \left(\zeta
    ^2-1\right)+49 d \left(3 \zeta ^2-5\right)}{49 a^2} \epsilon^2 
        + \mathcal{O}(\epsilon^4), \\
    \label{eq:third_order_vhat_smooth}
	\tilde  v(\zeta) 
  ={}& -2 \tilde\omega(\zeta) \sigma (1-\zeta^2)\zeta
	= -\left[ -12 + \frac{72b}{7a} \zeta \epsilon
			 - \frac{3}{49 a^2} \left(70 a_1 b-6 b^2 \left(9 \zeta ^2+5\right)
       \right. \right. \\
     & \left. -147 d \left(3 \zeta ^2-1\right)\right)\epsilon^2 + \frac{12
         \zeta}{2401 a^3}  \left(147 b \left(20 a b_1-7 d \zeta ^2+18
         d\right) \right. \nonumber \\
    & \left. \left. + 9604 a e \left(\zeta ^2-1\right)-2940 a_1 b^2-18 b^3 \
\left(7 \zeta ^2-18\right)\right) \epsilon^3 \right]
     (1-\zeta^2) \zeta + \mathcal{O}(\epsilon^4) \nonumber.
\end{align}

The relation $\xi(s)$ can be obtained by solving the ODE
\begin{equation}
				\label{eq:third_order_dxi_ds_smooth}
				\frac{d\xi}{ds}(s) = \tilde \omega(\tanh(\xi(s))).
\end{equation}
Thus, we substitute 
\begin{equation*}
				\xi(s) = s + \xi_1(s)\epsilon + \xi_2(s)\epsilon^2
				+ \xi_3(s)\epsilon^3 + \mathcal{O}(\epsilon^4),
\end{equation*}
into \cref{eq:third_order_dxi_ds_smooth} and expand the resulting equation in
$\epsilon$ to obtain
\begin{align*}
\frac{d\xi_1}{ds}(s) ={}& -\frac{6b\tanh(s)}{7a}, \\
\frac{d\xi_2}{ds}(s) ={}& \frac{-168 a b \sech^2(s) \xi_1(s)+70 a_1 b+9 \left(6 b^2+49 d\right) \tanh ^2(s)+18 b^2-245 d}{196 a^2} , \\
\frac{d\xi_3}{ds}(s) ={}& \frac{\sech^3(s)}{4802 a^3} \left(4116 a^2 b \sinh (s) (\xi_1(s))^2-4116 a^2 b \cosh (s) \xi_2(s) \right. \\
                        & +441 a \left(6 b^2+49 d\right) \sinh (s)\xi_1(s)+2 \sinh (s) \left(-3 b \cosh (2 s) \right. \\
                        & \left(98 (5 a b_1+d)-245 a_1 b+12 b^2\right)-1470 a b b_1+9604 a e+735 a_1 b^2 \\
                        & \left. \left. -162 b^3-1323 b d\right)\right).
\end{align*}
Here we directly used that $\xi_0(s)=s$. By solving these equations recursively
we obtain
\begin{equation*}
\begin{aligned}
\xi_1(s) ={}& - \frac{6b}{7a}\log(\cosh(s)), \\
\xi_2(s) ={}& \frac{2 s \left(35 a_1 b-36 b^2+98 d\right)+9 \tanh (s) \left(16 b^2 \log (\cosh (s))+10 b^2-49 d\right)}{196 a^2}, \\
\xi_3(s) ={}& \frac{1}{4802 a^3}\left(-7 \sech^2(s) \left(1372 a e-27 b \left(6 b^2+49 d\right) \log (\cosh (s)) \right. \right. \\
            & \left. +216 b^3 \log ^2(\cosh(s))-234 b^3-147 b d\right)-5880 a b b_1 \log (\cosh (s)) \\
            & + 9604 a e+42 b s \tanh (s) \left(-35 a_1 b+36 b^2-98 d\right)+4410 a_1 b^2 \log (\cosh (s)) \\
            & \left. -1656 b^3 \log (\cosh (s))-1638 b^3+2940 b d \log (\cosh (s))-1029 b d\right).
\end{aligned}
\end{equation*}
Here we used the phase condition that $\xi_i(0)=0,i=1,2,3$. This results in the
constraint $v(0)=0$. Substituting the above expression for $\xi$ into
\cref{eq:blowup_smooth} we obtain the third-order predictor
\begin{equation}
\label{eq:third_order_predictor_LP_tau_smooth}
\begin{cases}
\begin{aligned}
w_0(t)  &= \frac{1}{a} \tilde{u}\left(\tanh\left(\xi(\epsilon t)\right)\right) \epsilon^2, \\
w_1(t)  &= \frac{1}{a} \tilde{v}\left(\tanh\left(\xi(\epsilon t)\right)\right) \epsilon^3, \\
\beta_1    &= -\frac{4}{a}\epsilon^4, \\
\beta_2    &= \frac{b}{a}\tau\epsilon^2,
\end{aligned}
\end{cases}
\end{equation}
where $\tau,\tilde{u}$ and $\tilde{v}$ are given by
\cref{eq:tau_smooth,eq:third_order_uhat_smooth,eq:third_order_vhat_smooth},
respectively.

Note that by expanding $\tilde{u}\left(\tanh\left(\xi(s)\right)\right)$ in
\cref{eq:third_order_predictor_LP_tau_smooth} up to order three in $\epsilon$ 
we obtain
\begin{equation}
\label{eq:u_i_RP_smooth}
\begin{aligned}
    u_0(s) &= 6 \tanh ^2(s)-4, \\
    u_1(s) &= -\frac{72 b \tanh (s) \sech^2(s) \log (\cosh (s))}{7 a}, \\
    u_2(s) &= \left(12 s \sinh (2 s) \left(35 a_1 b-36 b^2+98 d\right)+8
        \cosh (2 s) \left(7 \left(5 a_1 b+9 b^2-56 d\right) \right. \right. \\
              & \; \left. -108 b^2 \log ^2(\cosh (s))+108 b^2 \log (\cosh (s))\right)
              +9 \left(35 a_1 b+192 b^2 \log^2(\cosh (s)) \right. \\
              & \; \left.\left. -96 b^2 \log (\cosh(s))-64 b^2+245 d\right)-7
              \cosh (4 s) (5 a_1 b+7 d)\right)\frac{\sech^4(s)}{196 a^2}, \\
        u_3(s) &=  \left(-2 \sinh (s) \left(\cosh (2 s) \left(-6 b \log (\cosh (s))
                        \left(-980 (a b_1+3 d)+1225 a_1 b+312 b^2\right)
                        \right.\right.\right. \\ 
                  & \; \left. +7 \left(-1372 a e+234 b^3+147 b d\right)+2016 b^3 \log^3(\cosh (s)) -6048 b^3 \log ^2(\cosh (s))\right)\\
                  & \; +6 b \log(\cosh(s)) \left(980 a b_1-1225 a_1 b+1200 b^2 -9408 d\right)+7 \left(1372 a e-234 b^3 \right. \\
                  & \; \left. \left. - 147 b d\right)-10080 b^3 \log ^3(\cosh (s)) +15120
                  b^3 \log ^2(\cosh (s))\right)  \\
                  & \; +42 b s \cosh ( 3 s)\left(35 a_1 b -36 b^2+98 d\right) (2 \log (\cosh (s)) -1) \\
                  & \; \left. +42 b s \cosh (s) \left(-35 a_1 b+36 b^2-98 d\right) (6 \log (\cosh (s))-1)\right)\frac{3 \sech^5(s)}{4802 a^3}.
\end{aligned}
\end{equation}
Together with \cref{eq:tau_smooth}, this is the solution obtained by using the
regular perturbation method to the second-order nonlinear oscillator
\cref{eq:second_order_nonlinear_oscillator_smooth_normalform} obtained from the
smooth normal form with phase condition $\dot u(0)=0$. This gives us the third
order homoclinic predictor
\begin{equation}
\label{eq:third_order_predictor_smooth_RPM_tau}
\begin{cases}
\begin{aligned}
w_0(t)  &= \frac1{a} \left( \sum_{i=0}^3 u_i(\epsilon t) \epsilon^i +
\mathcal{O}(\epsilon^4) \right)   \epsilon^2, \\
w_1(t)  &= \frac1{a} \left( \sum_{i=0}^3 \dot u_i(\epsilon t) \epsilon^i +
\mathcal{O}(\epsilon^4) \right)   \epsilon^3, \\
\beta_1    &= - \frac{4}{a} \epsilon^4, \\
\beta_2    &= \frac{b}{a}\epsilon^2 \tau,
\end{aligned}
\end{cases}
\end{equation}
where $\tau$ is given by \cref{eq:tau_smooth} and $u_i(i=0,\dots,3)$ are given
by \cref{eq:u_i_RP_smooth}.

\section{Case study of the quadratic codim 2 Bogdanov-Takens normal form}
\label{sec:case_study_BT2}
In this section, we will numerically study the algorithm outlined in
\cref{sec:PolynomailLindstedtPoincare} for the quadratic codimension 2
Bogdanov-Takens normal form \cref{eq:universal_unfolding}.  Since the algorithm
only relies on arithmetic and calculus on polynomials over the field $\mathbb
Q$, see \cref{corllary:rational_coefficients} there is no need to use propriety
software for the implementation. We choose the relative new programming
language Julia~\cite{bezanson2017julia}. Julia natively supports arbitrary
precision rational numbers. We use the package
\mintinline{julia}{Polynomials.jl}~\cite{Polynomials} to handle the
differentiation and integration of polynomials, as well as polynomial division.
Since the programming language Julia starts indexing arrays at $1$, we use the
package \mintinline{julia}{OffsetArrays}~\cite{OffsetArrays} to lower the index
to $0$ to keep the indexing identical. The code is given in the listing below.

\begin{minted}[breaklines,escapeinside=||,mathescape=true,
numbersep=3pt, gobble=2, frame=lines, fontsize=\small, framesep=2mm]{julia}
module BTQuadraticHomoclinic

using Polynomials, OffsetArrays

function z(i, τ, σ, ω)
    if i == 1
        p = 24*Polynomial([0,-2, 0,  3])
    else
        p =  Polynomial([0,2])*sum(σ[l]*τ[i-1-l] for l in 1:i-1)
        p += Polynomial([0,2])*sum(σ[l]*ω[k]*τ[i-1-l-k] 
                for k in 1:i-1 for l in 0:i-1-k)
        p -= Polynomial([2,0,-6])*sum(σ[i-l]*ω[l] for l in 1:i-1)
        p -= 2*sum(σ[i-l-k]*ω[l]*derivative(
                Polynomial([0,1,0,-1])*ω[k]) for k in 1:i-1 for l in 0:i-k)
        p += Polynomial([-1,0,1])*Polynomial([0,2])*sum(σ[l]*
               σ[i-1-l-k]*ω[k] for k in 0:i-1 for l in 0:i-2-k)
        p += Polynomial([-4,0,6])*Polynomial([0,2])*sum(σ[i-1-k]*ω[k]
                for k in 0:i-1)
        p += Polynomial([1,0,-1])*sum(σ[k]*σ[i-k] for k in 1:i-1)
    end
    Polynomial([1,0,-1])*p
end

function solve(;order = n)
    σ = OffsetArray(zeros(Rational{BigInt}, order), 0:order-2)
    τ = OffsetArray(zeros(Rational{BigInt}, order-1), 0:order-1)
    ω = OffsetArray(Array{Polynomial}(undef, order), 0:order-1)

    σ[0], ω[0]  = 6, 1
    for i=1:order-1
        gi = integrate(Polynomial([0,12//1])*z(i, τ, σ, ω))
        if i%2 == 1
            τ[i-1] = -10//192*gi(1)
            ω[i] = (τ[i-1]*144*Polynomial([2//15, 0, 0, 1//3, 0, -1//5]) 
                      + gi) ÷ (Polynomial([1,0,-1])*Polynomial([0,12]))^2
        else
            σ[i] = -gi(-1)//12
            ω[i] = -σ[i]//6 + (σ[i]*Polynomial([-1,0,1])*
                (Polynomial([-4,0,6])^2-4) + gi + 12*σ[i]) ÷ 
                        (Polynomial([1,0,-1])*Polynomial([0,12]))^2
        end
    end
    τ, σ, ω
end

end
\end{minted}

In~\cite{Algaba_2019} the impression is given that one can obtain higher-order
approximations very fast using their algorithm. However, our algorithm, which
should be superior, shows that the order of approximation and the computational
cost is not linear related. We performed a benchmark to obtain a 20th-order
approximation to $\tau$, see \cref{tab:benchmark}.
\begin{table}[tpb]
    \centering
    \begin{minted}[breaklines,escapeinside=||,mathescape=true,
    numbersep=3pt, gobble=2, frame=lines, fontsize=\small, framesep=2mm]{julia}
    julia> @benchmark BTQuadraticHomoclinic.solve(order=20)
    BenchmarkTools.Trial:
      memory estimate:  107.36 MiB
      allocs estimate:  5384379
      --------------
      minimum time:     382.111 ms (21.56% GC)
      median time:      414.436 ms (24.41% GC)
      mean time:        459.831 ms (26.81% GC)
      maximum time:     577.254 ms (32.91% GC)
      --------------
      samples:          11
      evals/sample:     1
    \end{minted}
    \caption{Benchmark to obtain a 20th-order approximation to $\tau$ in the
    quadratic normal form \cref{eq:universal_unfolding} using the algorithm
    outlined in \cref{sec:PolynomailLindstedtPoincare}. See in particular
    \cref{corollary:quadraticBTsigma_delta_relation} and the algorithm above
    .}\label{tab:benchmark}
\end{table}
These results were obtained on the mobile CPU Intel i5-6200U (4) @ 2.800GHz
with 11407MiB of memory. The coefficients for $\tau$ are shown in
\cref{table:tauCoeffients}. We see that the length of the numerator and
denominator increases at each (even) order. This also holds true for the
coefficients of $\sigma$ and the coefficients of the polynomials $\omega$.
Performing operations on these rational numbers become increasingly more
difficult for the computer to deal with. In
\cref{fig:BTQuadraticNormalFormTimings} a log-linear plot is shown, plotting
the time in seconds to solve the $ith$ order equation
\cref{eq:ith_order_equation}. It took nearly 10 hours to obtain the 200th order
approximation of $\tau$. A linear regression on the last 50 data points
indicates the time increases exponentially. Extrapolation yields that it would
take more than 17 years to solve the first 500 terms if memory doesn't become
an issue. Since the algorithm is embarrassingly parallelizable we could speed
up the process. However, exponential growth cannot be escaped.
\begin{table}[bh!]
\label{table:tauCoeffients}
\centering
\def\arraystretch{1.2}
\begin{filecontents*}{data/coefficients_first10.csv}
i,tau
0,$\frac{10}{7}$
2,$\frac{288}{2401}$
4,-$\frac{240192}{45294865}$
6,-$\frac{37647386496}{11108339166925}$
7,$\frac{5154131788676352}{2724264638992521625}$
8,-$\frac{1828854022199737517568}{1622558397660750917241875}$
10,$\frac{4788476595594706542555091968}{6764713681983284597882721784375}$
12,-$\frac{333207890299772418358595124320937984}{750205319977359363432143692632890996875}$
14,$\frac{384324328969886017955349568884165743978164224}{1366817906371309055843841557575267655935039046875}$
16,-$\frac{2241496149333347510622026401379280877412634073959456768}{12451199287907137102778709466943415347212749443284171484375}$
18,$\frac{2640541575683401164949776492089908615592018826498686833538777088}{22685152569996135996229496037325376035529199986794205676507927734375}$
20,-$\frac{90999911407271629445902884163495991850622440695336663926865318101431386112}{1198588820362905667811216420985132469928184621078674921926335765852723935546875}$
\end{filecontents*}
\pgfplotstabletypeset[
    col sep=comma,
    string type,
    every head row/.style={%
        before row=\hline,
        after row=\hline
    },
    every last
    row/.style={after
    row=\hline},
    columns/i/.style={column
        name=$i$,
        column
    type=r},
    columns/tau/.style={column
        name=$\tau_i$,
        column
    type=l}
    ]{data/coefficients_first10.csv}
    \caption{First 20 coefficients of $\tau$}
\end{table}
\begin{figure}[ht]
\label{fig:BTQuadraticNormalFormTimings}
\centering
\ifcompileimages%
  \tikzsetnextfilename{BTQuadraticNormalFormTimings}%
  \input{tikz/BTQuadraticNormalFormTimings}%

\else
    \includegraphics{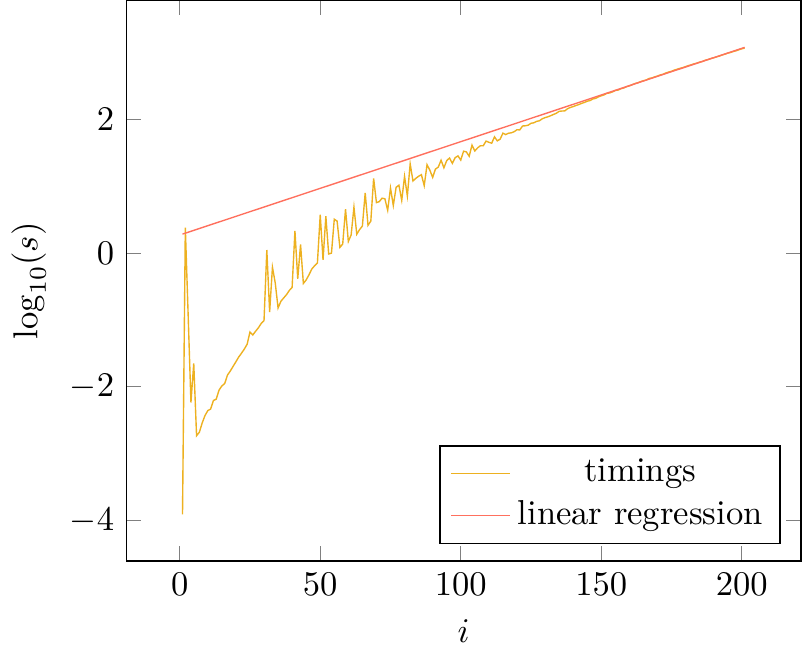}
\fi
\caption{Log-linear plot of the order $i$ and the time in seconds it took to
compute the coefficients.}
\end{figure}

\begin{figure}[ht!]
\label{fig:BTQuadraticNormalFormApproximations}
\centering
\ifcompileimages
  \tikzsetnextfilename{BTQuadraticNormalFormApproximations}%
  \input{tikz/BTQuadraticNormalFormApproximations}%

    \else
    \includegraphics{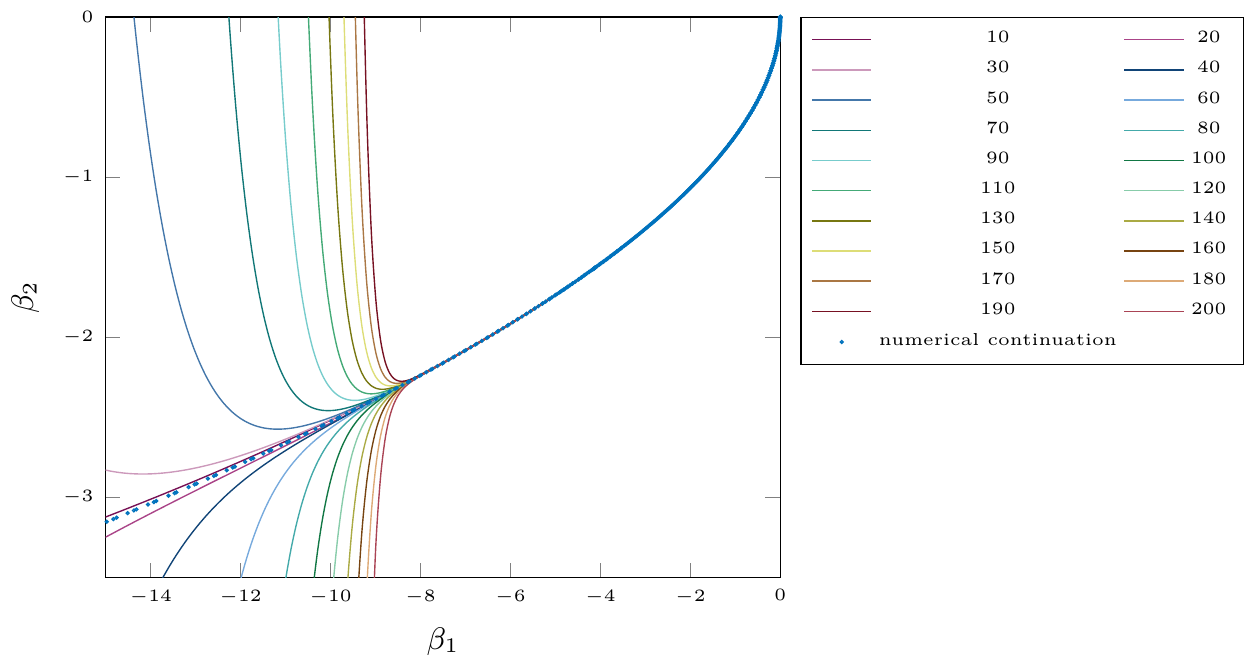}
\fi
\caption{Comparison between the numerical obtained continued homoclinic
    bifurcation curve emanating from the Bogdanov-Takens point in
    \cref{eq:universal_unfolding} with $a=b=1$ and the parameter approximations
    with orders ranging from 10 to 200.  For $\beta_1 \lesssim -8$ the
    approximation starts diverging. This indicates that the radius of
    convergence in the perturbation parameter $\epsilon$ is less than or equal
    to $\sqrt[4]{2}$. Note that for $\beta_1 > -8$ the tenth order is
    already indistinguishable from the numerical obtained parameters.}
\end{figure}

Next, we would like to make some comments on the radius of convergence of the
asymptotic approximation to the homoclinic orbit in the quadratic
Bogdanov-Takens bifurcation. In~\cite{Algaba_2019} there is the remark that the
higher-order approximation can greatly improve the accuracy of the
approximation for large parameter values. However, this fully depends
on the radius of convergence of the series. To obtain a first impression, we
compare the parameters computed from the numerical continued solution to the
homoclinic solution in \cref{eq:universal_unfolding} using MatCont with the
predicted parameters. \cref{fig:BTQuadraticNormalFormApproximations} reveals 
a typical situation  in perturbation series. Increasing the order of the
perturbation parameter improves the approximation for small parameters, but for
larger parameters, the approximation becomes much worse. Reproducing~\cite[Fig
3c]{Algaba_2019}, but increasing the order, shows that this solution is outside
the radius of convergence.

\section*{Acknowledgments}
The authors would like to thank Prof. Peter De Maesschalck (Hasselt University)
for multiple useful discussions during this research project.

\bibliographystyle{siamplain}
\bibliography{references}
\end{document}